\documentclass[
reqno]{amsart} 
\usepackage[T1]{fontenc}
\usepackage[english]{babel}
\usepackage{amssymb,bbm,enumerate}
\usepackage{color}
\usepackage[linktocpage=true,colorlinks=true, linkcolor=blue, citecolor=red, urlcolor=green]{hyperref}

\usepackage{xypic}

\usepackage{tikz}

\renewcommand{\phi}{\varphi}
\renewcommand{\ker}{\Ker}
\def\ch{\mathrm{ch}}
\def\cl{\mathrm{cl}}
\def\fn{\mathrm{fn}}
\def\dd{\partial}
\def\al{\alpha}

\def\si{\sigma}
\def\ph{\varphi}

\def\Ga{\Gamma}
\def\la{\lambda}
\def\La{\Lambda}

\newcommand{\mc}[1]{\mathcal{#1}}
\newcommand{\mf}[1]{\mathfrak{#1}}
\newcommand{\mb}[1]{\mathbb{#1}}

\newcommand{\id}{\mathbbm{1}}
\newcommand{\tint}{{\textstyle\int}}

\DeclareMathOperator{\Hom}{Hom}
\DeclareMathOperator{\End}{End}

\DeclareMathOperator{\ad}{ad}
\DeclareMathOperator{\im}{Im}

\DeclareMathOperator{\Der}{Der}
\DeclareMathOperator{\Inder}{Inder}
\DeclareMathOperator{\Cas}{Cas}
\DeclareMathOperator{\Ker}{Ker}

\DeclareMathOperator{\fil}{F}
\DeclareMathOperator{\gr}{gr}

\newcommand{\as}{\mathop{\rm as }}

\newcommand{\vac}{|0\rangle}

\theoremstyle{plain}
\newtheorem{theorem}{Theorem}[section]
\newtheorem{lemma}[theorem]{Lemma}
\newtheorem{proposition}[theorem]{Proposition}
\newtheorem{corollary}[theorem]{Corollary}

\theoremstyle{definition}
\newtheorem{definition}[theorem]{Definition}
\newtheorem{example}[theorem]{Example}

\theoremstyle{remark}
\newtheorem{remark}[theorem]{Remark}

\numberwithin{equation}{section}

\definecolor{light}{gray}{.9}

\begin{document}

\title{An operadic approach to vertex algebra and Poisson vertex algebra cohomology}

\author{Bojko Bakalov}
\address{Department of Mathematics, North Carolina State University,
Raleigh, NC 27695, USA}
\email{bojko\_bakalov@ncsu.edu}
%
\author{Alberto De Sole}
\address{Dipartimento di Matematica, Sapienza Universit\`a di Roma,
P.le Aldo Moro 2, 00185 Rome, Italy}
\email{desole@mat.uniroma1.it}
\urladdr{www1.mat.uniroma1.it/$\sim$desole}
\author{Reimundo Heluani}
\address{IMPA, Rio de Janeiro, Brasil}
\email{rheluani@gmail.com}
\author{Victor G. Kac}
\address{Department of Mathematics, MIT,
77 Massachusetts Ave., Cambridge, MA 02139, USA}
\email{kac@math.mit.edu}



\begin{abstract}
We translate the construction of the chiral operad by Beilinson and Drinfeld
to the purely algebraic language of vertex algebras.
Consequently,
the general construction of a cohomology complex associated to a linear operad
produces a vertex algebra cohomology complex.
Likewise, the associated graded of the chiral operad
leads to a classical operad, which produces
a Poisson vertex algebra cohomology complex.
The latter is closely related to the variational Poisson cohomology studied by two 
of the authors.
\end{abstract}

\maketitle

\tableofcontents

\pagestyle{plain}

\section{Introduction}\label{sec:1}

The universal Lie superalgebra associated to a vector superspace $V$
is defined as a $\mb Z$-graded Lie superalgebra 
$$W(V)=\bigoplus_{j\geq-1}W_j(V)
\,,\,\,\text{ with }\,\,W_{-1}(V)=V\,,
$$
such that for any $\mb Z$-graded Lie superagebra $\mf g=\bigoplus_{j\geq-1}\mf g_j$,
with $\mf g_{-1}=V$,
there is a unique grading preserving homomorphism $\mf g\to W(V)$
identical on $V$.
It is easy to see that
$$
W_j(V)
=
\Hom(S^{j+1}(V),V)
\,,
$$
for all $j\geq-1$.
The Lie superalgebra bracket on $W(V)$ is given by
\begin{equation}\label{eq:intro1}
[X,Y]
=
X\Box Y-(-1)^{p(X)p(Y)}Y\Box X
\,,
\end{equation}
were $p$ is the parity on $W(V)$, and, for $X\in W_n(V)$, $Y\in W_m(V)$,
\begin{equation}\label{eq:intro2}
\begin{split}
& (X\Box Y)(v_0\otimes\dots\otimes v_{m+n})
\\
& =
\!\!\!\!\!\!\!\!\!
\sum_{\substack{i_0<\dots<i_m \\ i_{m+1}<\dots<i_{m+n}}}
\!\!\!\!\!\!\!\!\!
\epsilon_v(i_0,\dots,i_{m+n})
X(Y(v_{i_0}\otimes\dots\otimes v_{i_m})\otimes v_{i_{m+1}}\otimes\dots\otimes v_{i_{m+n}})
\,.
\end{split}
\end{equation}
Here
$\epsilon_v(i_0,\dots,i_{m+n})$ is non-zero only if $i_0,\dots,i_{m+n}$ are distinct,
and in this case it is equal to $(-1)^N$,
where $N$ is the number of interchanges of indices of odd $v_i$'s in the permutation.

Clearly, $W_0(V)=\End V$ and $W_1(V)=\Hom(S^2V,V)$,
so that any even element of the vector superspace $W_1(V)$
defines a commutative superalgebra structure on $V$,
and this correspondence is bijective.
On the other hand, any odd element $X$ of the vector superspace $W_1(\Pi V)$
defines a skewcommutative superalgebra structure on $V$
by the formula
\begin{equation}\label{eq:intro3}
[a,b]
=
(-1)^{p(a)}X(a\otimes b)
\,,\qquad
a,b\in V
\,.
\end{equation}
Here and further $\Pi V$ stands for the vector superspace $V$ with reversed parity.
Moreover, \eqref{eq:intro3}
defines a Lie superalgebra structure on $V$
if and only if $[X,X]=0$ in $W(\Pi V)$.
Thus, given a Lie superalgebra structure on $V$,
considering the corresponding odd element $X\in W_1(\Pi V)$,
we obtain a cohomology complex
\begin{equation}\label{eq:intro4}
\big(
C^\bullet
=
\bigoplus_{j\geq0}C^j
,\ad X
\big)
\,\,,\,\,\text{ where }\,\,
C^j=W_{j-1}(\Pi V)
\,,
\end{equation}
which coincides with the cohomology complex of the Lie superalgebra $V$
with the bracket defined by $X$,
with coefficients in the adjoint representation.
This construction for $V$ purely even goes back to the paper \cite{NR67}
on deformations of Lie algebras;
for a general superspace $V$ it was explained in \cite{DSK13}.
Note also that, more generally, given a module $M$ over the Lie superalgebra $V$,
one considers instead of $V$ the Lie superalgebra $V\ltimes M$
with $M$ an abelian ideal,
and by a simple reduction procedure constructs the cohomology complex 
of the Lie superalgebra $V$ with coefficients in $M$.

In the paper \cite{DSK13}, this point of view on cohomology has been also applied
to several other algebraic structures.
The most important for the present paper is that of a Lie conformal (super)algebra
and the corresponding cohomology complex
introduced in \cite{BKV99};
see also \cite{BDAK01}, \cite{DSK09}.
The complex is constructed in \cite{DSK13} as follows.
Assume that the vector superspace $V$ carries an even endomorphism $\partial$.
For each integer $k\geq0$,
denote by $\mb F_-[\lambda_1,\dots,\lambda_k]$
the space of polynomials in the $k$ variables $\lambda_1,\dots,\lambda_k$
of even parity with coefficients in the field $\mb F$,
endowed with the structure of a left $\mb F[\partial]^{\otimes k}$-module
by letting
$P_1(\partial)\otimes\dots\otimes P_k(\partial)$
act as multiplication by
$P_1(-\lambda_1)\cdots P_k(-\lambda_k)$.
This space carries also a right $\mb F[\partial]$-module structure,
for which $\partial$ acts as multiplication by 
$-\lambda_1-\dots-\lambda_k$.
Then we let for $k\geq0$:
\begin{equation}\label{eq:intro5}
P^\partial_k(V)
=
\Hom_{\mb F[\partial]^{\otimes k}}
\big(
V^{\otimes k}
,
\mb F_-[\lambda_1,\dots,\lambda_k]
\otimes_{\mb F[\partial]} V
\big)
\,.
\end{equation}
The symmetric group $S_k$
acts on the vector superspace $P^\partial_k(V)$
by simultaneous permutation of the factors of the vector superspace 
$V^{\otimes k}$
and of the $\lambda_i$'s.
We denote by $W^\partial_k(V)$
the subspace of fixed points in $P^\partial_{k+1}(V)$, $k\geq-1$.
Then the ``conformal'' analogue of $W(V)$ is the vector superspace
$$
W^\partial(V)
=
\bigoplus_{j\geq-1}
W^\partial_j(V)
\,,
$$
with a $\mb Z$-graded Lie superalgebra structure 
similar to \eqref{eq:intro1}-\eqref{eq:intro2}.
Note that we have:
$$
W_{-1}^\partial(V)=V/\partial V
\,\,,\,\,\,\,
W_{0}^\partial(V)=\End_{\mb F[\partial]}V
\,.
$$
Moreover, the even elements in the vector superspace $W^\partial_1(V)$
are identified, letting $\lambda_1=\lambda$ and $\lambda_2=-\lambda-\partial$,
with maps
$$
X\colon V\otimes V\to V[\lambda]
\,,\quad
a\otimes b\mapsto X_\lambda(a\otimes b)
\,,
$$
which satisfy certain sesquilinearity and commutativity conditions.

Proceeding in exactly the same way as in the Lie superalgebra case, 
consider the Lie superalgebra $W^\partial(\Pi V)$.
Then we get a bijection between odd elements $X\in W^\partial_1(\Pi V)$,
such that $[X,X]=0$,
and the \emph{Lie conformal superalgebra} structures on $V$,
i.e., $\lambda$-brackets on $V$ satisfying sesquilinearity 
\begin{equation}\label{eq:intro6}
[\partial a_\lambda b]=-\lambda[a_\lambda b]
\,,\quad
[a_\lambda\partial b]=(\lambda+\partial)[a_\lambda b]
\,,
\end{equation}
skewcommutativity
\begin{equation}\label{eq:intro8}
[b_\lambda a]
=
-(-1)^{p(a)p(b)}[a_{-\lambda-\partial}b]
\,,
\end{equation}
and Jacobi identity
\begin{equation}\label{eq:intro9}
[a_\lambda[b_\mu c]]
-(-1)^{p(a)p(b)}
[b_\mu[a_\lambda c]]
=
[[a_\lambda b]_{\lambda+\mu}c]
\,.
\end{equation}
This bijection is similar to \eqref{eq:intro3}:
\begin{equation}\label{eq:intro10}
[a_\lambda b]
=
(-1)^{p(a)}X_\lambda(a\otimes b)
\,.
\end{equation}
Moreover, similarly to \eqref{eq:intro4},
we obtain the cohomology complex of the Lie conformal superalgebra $V$ with $\lambda$-bracket
given by $X_\lambda$ via \eqref{eq:intro10},
with coefficients in the adjoint representation.
One defines the cohomology of $V$ with coefficients in a $V$-module $M$
in a similar way as well.

The most relevant to \cite{DSK13} construction is obtained by endowing the $\mb F[\partial]$-module
$V$ with a structure of a (commutative associative) differential superalgebra.
In this case one considers the $\mb Z$-graded subalgebra
$W^{\partial,\as}(V)=\bigoplus_{j\geq-1}W^{\partial,\as}_j$
of $W^\partial(V)$,
where $W^{\partial,\as}_j=W^{\partial}_j$ for $j=-1,0$,
while $W^{\partial,\as}_j$ for $j\geq1$ consists of the maps from $W^{\partial}_j$
satisfying the Leibniz rule.
The odd elements $X\in W^{\partial,\as}_1(\Pi V)$, such that $[X,X]=0$,
correspond bijectively to Poisson vertex algebra (PVA) structures on $V$
with the given differential algebra structure,
and using this, one constructs the variational Poisson cohomology of the PVA $V$.
Recall that a \emph{Poisson vertex (super)algebra} is a differential (super)algebra 
endowed with a Lie conformal (super)algebra 
$\lambda$-bracket satisfying the Leibniz rule
\begin{equation}\label{eq:intro11}
[a_\lambda bc]
=
[a_\lambda b]c
+
(-1)^{p(a)p(b)}
b[a_\lambda c]
\,.
\end{equation}

A somewhat different point of view on cohomology complexes of algebraic structures
is provided by the theory of linear unital symmetric superoperads,
which we call \emph{operads} in this paper for simplicity.
(It covers the first two above mentioned examples, but not the third one.)
One of its equivalent definitions is that it is a sequence of vector superspaces $\mc P(n)$, 
$n\in\mb Z_{\geq0}$, endowed with a right action of $S_n$ for $n\geq1$,
and bilinear 
parity preserving products
$$\circ_i\colon\mc P(n)\times\mc P(m)\to\mc P(n+m-1) \,, \qquad i=1,\dots,n \,,$$
satisfying the associativity axioms given by formula \eqref{eq:operad25} below
and the equivariance axioms given by formula \eqref{eq:operad9}.
(There is also a unity $1\in\mc P(1)$, satisfying the unity axiom,
but this is irrelevant to our paper.)
See, e.g. \cite{MSS02}, \cite{LV12}.
The universal (to the operad $\mc P$) $\mb Z$-graded Lie superalgebra
$W(\mc P)=\bigoplus_{j\geq-1}W_j$
is defined by letting $W_n=\mc P(n+1)^{S_{n+1}}$
with the bracket \eqref{eq:intro1},
where 
$$
X\Box Y=\sum_{\sigma\in S_{m+1,n}}(X\circ_1 Y)^{\sigma^{-1}}
\,.
$$
Here $S_{m,n}$ denotes the set of $(m,n)$-shuffles in $S_{m+n}$;
see Section \ref{sec:3} for details.
The earliest reference to this construction that we know of is \cite{Tam02}.

The most popular example of an operad is $\mc P=\mc Hom$, for which
$$
\mc Hom(n)=\Hom(V^{\otimes n},V)
\,,
$$
for a vector superspace $V$.
The action of $S_n$ on $\mc P(n)$ is defined via its natural action on $V^{\otimes n}$
(taking into account the parity of $V$),
and the $i$-th product $X\circ_i Y$ of $X\in\mc Hom(n)$ and $Y\in\mc Hom(m)$
is defined by ($i=1,\dots,n$)
$$
(X\circ_iY)(v_1,\dots, v_{n+m-1})
=
X(v_1,\dots,v_{i-1},
Y(v_i,\dots,v_{i+m-1}),
v_{i+1},\dots,v_{n+m-1})
\,.
$$
It is easy to see that the Lie superalgebras $W(V)$ and $W(\mc Hom)$ are identical.
Likewise, for an $\mb F[\partial]$-module $V$,
one defines the operad $\mc Chom$,
for which $\mc Chom(n)$ is the space $P^\partial_n(V)$ defined by \eqref{eq:intro5},
and recovers thereby the associated Lie superalgebra $W^\partial(V)$.

Thus, the operads $\mc Hom$ and $\mc Chom$ ``govern'' the Lie algebras
and the Lie conformal superalgebras respectively.
In their seminal book \cite{BD04},
Beilinson and Drinfeld generalized the notion of a vertex algebra,
introduced by Borcherds \cite{Bor86},
by defining a chiral algebra in the language of $\mc D$-modules
on any smooth algebraic curve,
so that a vertex algebra is 
a weakly translation covariant chiral algebra on the affine line.
They also constructed the corresponding chiral operad and the cohomology theory of chiral algebras,
and the associated graded chiral operad.

In the present paper, we translate the construction of the chiral operad from \cite{BD04}
to the purely algebraic language of vertex algebras.
The resulting operad, which we denote by $P^\ch$,
not surprisingly turns out to be an extension of the operad $\mc Chom$
in the same spirit as $\mc Chom$ is an extension of the operad $\mc Hom$.

In order to explain the construction of $P^\ch$ (see Section \ref{sec:wch}),
let us introduce some notation.
For $k\in\mb Z_{\geq-1}$, let
$\mc O^T_{k+1}$ and $\mc O^{\star T}_{k+1}$
be respectively the algebras of polynomials and Laurent polynomials in
$z_{ij}=z_i-z_j$, where $0\leq i<j\leq k+1$,
and let $\mc D^T_{k+1}=\sum_{i=0}^k\mc O^T_{k+1}\partial_{z_i}$
be the algebra of translation invariant regular differential operators.
Let $V$ be an $\mb F[\partial]$-module.
The space $V^{\otimes(k+1)}\otimes\mc O^{\star T}_{k+1}$
carries the structure of a right $\mc D^T_{k+1}$-module by letting $z_{ij}$
act by multiplication on $\mc O^\star_{k+1}$,
and letting $\partial_{z_i}$ act by
$$
(v_0\otimes\dots\otimes v_k\otimes f)\partial_{z_i}
=
v_0\otimes\cdots\partial v_i\cdots\otimes v_k\otimes f
-
v_0\otimes\dots\otimes v_k\otimes \frac{\partial f}{\partial z_i}
\,.
$$
The space $\mb F_-[\lambda_0,\dots,\lambda_k]$ considered above
carries a structure of a $\mc D^T_{k+1}$-module as well,
by letting $z_{ij}$ act as $-\frac{\partial}{\partial\lambda_i}+\frac{\partial}{\partial\lambda_j}$
and $\partial_{z_i}$ act as multiplication by $-\lambda_i$.
Then $P^\ch(k+1)$ is defined as the space of all right $\mc D^T_{k+1}$-module
homomorphisms
$$
V^{\otimes(k+1)}\otimes\mc O^{\star T}_{k+1}
\longrightarrow
\mb F_-[\lambda_0,\dots,\lambda_k]
\otimes_{\mb F[\partial]} V
\,.
$$
The right action of the symmetric group $S_{k+1}$ on $P^\ch(k+1)$
is defined by simultaneous permutations in $V^{\otimes(k+1)}\otimes\mc O^{\star T}_{k+1}$
of factors of $V^{\otimes(k+1)}$ and the corresponding variables $z_0,\dots,z_k$ 
in $\mc O^{\star T}_{k+1}$.
The $\circ_1$ product in $P^\ch$ is defined by \eqref{circ1},
and the general composition by \eqref{circ5}.

We denote by $W^\ch(V)=\bigoplus_{j\geq-1}W^\ch_j(V)$ 
the $\mb Z$-graded Lie superalgebra associated to the operad $P^\ch$
for the $\mb F[\partial]$-module $V$.
It is clear that $W^\ch_j$ for $j=-1,0$
is the same as for the operad $\mc Chom$.
However, $W^\ch_1(\Pi V)$ is identified not with the space
of sesquilinear skewsymmetric $\lambda$-brackets
as for $\mc Chom$, but with their integrals; see Proposition \ref{20160719:prop}.
Moreover, the set of odd elements $X\in W^\ch_1(\Pi V)$
such that $[X,X]=0$ is identified with those integrals of $\lambda$-brackets
satisfying the ``integral'' Jacobi identity;
see Theorem \ref{20160719:thm}.
Thus, due to the integral of $\lambda$-bracket definition of a vertex algebra
introduced in \cite{DSK06},
such elements $X$ parametrize non-unital vertex algebra structures 
on the $\mb F[\partial]$-module $V$.

As explained above, we thus obtain a cohomology complex for any non-unital vertex algebra $V$
and its module $M$.
The low cohomology is as expected from any Lie-type cohomology.
Namely, the $0$-th cohomology parametrizes Casimirs (i.e., invariants)
of the $V$-module $M$,
the $1$-st cohomology is identified with the quotient 
of all derivations from $V$ to $M$
by the space of inner derivations,
and the $2$-nd cohomology parametrizes the $\mb F[\partial]$-split extensions
of $V$ by $M$ with a trivial structure of a non-unital vertex algebra
(see Theorem \ref{thm:lowcoho}).
The vertex algebra cohomology studied in \cite{Bor98}, \cite{Hua14} and \cite{Lib17} 
is rather of Harrison type; for example, their $1$-st cohomology is identified
with the space of all derivations from $V$ to $M$.

The $\mathbb{Z}$-graded Lie superalgebra associated to the operad $P^{\ch}$ and its corresponding differential complex associated to a non-unital vertex algebra structure on $V$ was defined in \cite{Tam02} in the context of chiral algebras as the complex governing deformations of the chiral algebra structure. 
It was later studied in \cite{yanagida} where the author introduces also a Lie algebra structure on the complex governing deformations of Poisson vertex algebras. Both \cite{Tam02} and \cite{yanagida} rely on the geometric language of \cite{BD04} to construct these Lie algebras. In particular, they associate a deformation complex to any smooth algebraic curve $X$ and any chiral (respectively, coisson) algebra on $X$. In this article we restrict to the the case when $X$ is the affine line and the chiral (respectively, coisson) algebra is translation equivariant, hence associated to a vertex algebra (respectively, Poisson vertex algebra). In this restricted case, we are able to give a more explicit linear algebraic and combinatorial description of these complexes, providing a suitable framework to carry out computations of (Poisson) vertex algebra cohomologies.

The algebras $\mc O_{k+1}^{\star T}$ carry a natural increasing filtration
by the number of poles,
which induces a decreasing filtration of the operad $P^\ch$.
We study the associated graded operad, denoted by $P^\cl$.
Its explicit description is quite involved
and uses the cooperad of graphs (see Theorem \ref{20170616:thm1}). 
One can show that the operad $P^\ch$
studied in our paper is (non-canonically) isomorphic to that in \cite{BD04}
in the case of the curve being the affine line.

We also consider a refinement of the above filtration of $P^\ch$,
associated to an increasing filtration $0\subset \fil^1V\subset \fil^2V\subset\cdots$
of the $\mb F[\partial]$-module $V$,
and show that the structures of a filtered vertex algebra on $V$
are in bijective correspondence with odd elements $X\in \fil^1 W^\ch_1(\Pi V)$
such that $[X,X]=0$ (see Theorem \ref{20160719:thm-ref}).
Moreover, one has an 
injective morphism
of complexes
$$
\big(
\gr W^\ch(\Pi V),\gr(\ad X)
\big)
\,\hookrightarrow\,
\big(
W^\cl(\gr\Pi V),\ad(\gr X)
\big)
$$
(see Theorem \ref{cor:nomore}),
which is an isomorphism at least in low degrees.

Next, we show that the structures of a Poisson vertex algebra on the $\mb F[\partial]$-module $V$
are in bijection with the odd elements $X\in W^\cl_1(\Pi V)$
such that $[X,X]=0$ (see Theorem \ref{20170616:thm2}).
Using this, we relate the cohomology of the corresponding complex,
called the PVA cohomology,
to the variation Poisson cohomology studied in \cite{DSK13}.
In particular, we show that the low vertex algebra cohomology is majorized 
by the variational Poisson cohomology.
Using this and a computation of the variational Poisson 
cohomology in \cite{DSK12}-\cite{DSK13},
we compute the Casimirs and derivations of the vertex algebra
of $N$ bosons.

Throughout the paper, the base field $\mb F$ is a field of characteristic $0$
and, unless otherwise specified, all vector spaces, their tensor products and Hom's 
are over $\mb F$.

\subsubsection*{Acknowledgments} 

We are grateful to Pavel Etingof for providing a proof of Lemma \ref{20160719:lem}.
We would like to acknowledge discussions with him as well as 
with Corrado De Concini, Andrea Maffei and Alexander Voronov.
We also thank the referee for thoughtful comments.
The research was partially conducted during the authors' visits 
to IHES, MIT, SISSA and the University of Rome La Sapienza.
We are grateful to these institutions for their kind hospitality.
The first author was supported in part by a Simons Foundation grant 279074.
The second author was partially supported by the national PRIN fund n. 2015ZWST2C$\_$001
and the University funds n. RM116154CB35DFD3 and RM11715C7FB74D63.
The third author was partially supported by the Bert and Ann Kostant fund.

\section{Preliminaries on vector superspaces and the symmetric group}\label{sec:2}

\subsection{Vector superspaces, tensor products and linear maps}\label{sec:2.0}

Recall that a vector superspace is a $\mb Z/2\mb Z$-graded vector space $V=V_{\bar 0}\oplus V_{\bar 1}$.
We denote by $p(v)\in\mb Z/2\mb Z=\{\bar 0,\bar 1\}$ the parity of a homogeneous element $v\in V$.
Given two vector superspaces $U,V$, their tensor product $U\otimes V$
and the space of linear maps $\Hom(U,V)$
are naturally vector superspaces,
with $\mb Z/2\mb Z$-grading induced by those of $U$ and $V$,
i.e. we have $p(u\otimes v)=p(u)+p(v)$, and $p(f)=p(f(u))-p(u)$,
for $u\in U$, $v\in V$, $f\in\Hom(U,V)$.
Let $g_i\colon U_i\to V_i$, $i=1,\dots,n$, be linear maps of vector superspaces.
One defines their tensor product 
$g_1\otimes\dots\otimes g_n\colon U_1\otimes\dots\otimes U_n\to V_1\otimes\dots\otimes V_n$,
by
\begin{equation}\label{20170804:eq1}
(g_1\otimes\dots\otimes g_n)(u_1\otimes\dots\otimes u_n)
=
(-1)^{\sum_{i<j}p(g_j)p(u_i)}g_1(u_1)\otimes\dots\otimes g_n(u_n)
\,.
\end{equation}
In other words, we follow the usual Koszul--Quillen rule:
every time two odd elements are switched, we change the sign.

\subsection{The action of the symmetric group on tensor powers}\label{sec:2.1}

The symmetric group $S_n$ is, by definition, the group of bijections
$\sigma\colon \{1,\dots,n\}\stackrel{\sim}{\longrightarrow}\{1,\dots,n\}$,
mapping $i\mapsto\sigma(i)$.

If $V=V_{\bar 0}\oplus V_{\bar 1}$ is a vector superspace,
the symmetric group $S_n$ acts linearly on $V^{\otimes n}$:
\begin{equation}\label{eq:operad6}
\sigma(v_1\otimes\dots\otimes v_n)
:=
\epsilon_v(\sigma) 
\,
v_{\sigma^{-1}(1)}\otimes\dots\otimes v_{\sigma^{-1}(n)}
\,,
\end{equation}
where 
\begin{equation}\label{eq:operad14}
\epsilon_v(\sigma)
=
\prod_{i<j\,|\,\sigma(i)>\sigma(j)}(-1)^{p(v_i)p(v_j)}
\,.
\end{equation}
(Again, we follow the Koszul-Quillen rule for the sign factor.)
Formula \eqref{eq:operad6} defines a left action of $S_n$ on $V^{\otimes n}$,
since the signs $\epsilon_v(\sigma)$ satisfy the relation
\begin{equation}\label{eq:operad15}
\epsilon_v(\sigma\tau)
=
\epsilon_{\tau(v)}(\sigma)
\epsilon_v(\tau)
\,,
\end{equation}
which can be easily checked.

We also have the corresponding \emph{right} action of $S_n$
on the space $\Hom(V^{\otimes n},V)$
of linear maps $f(v_1\otimes\dots\otimes v_n)$, given by
\begin{equation}\label{eq:operad7}
f^\sigma(v_1\otimes\dots\otimes v_n)
=
\epsilon_v(\sigma) f(v_{\sigma^{-1}(1)}\otimes\dots\otimes v_{\sigma^{-1}(n)})
\,\Big(
=f(\sigma(v_1\otimes\dots\otimes v_n))
\Big)
\,.
\end{equation}
Note that the same formula \eqref{eq:operad6} makes sense
when applied to an element $v_1\otimes\dots\otimes v_n\in W_1\otimes\dots\otimes W_n$,
where $W_1,\dots,W_n$ are different vector superspaces.
In this case, $\sigma\in S_n$ defines an (even) linear map
\begin{equation}\label{eq:operad18}
\sigma\colon
W_1\otimes \dots\otimes W_n
\stackrel{\sim}{\longrightarrow}
W_{\sigma^{-1}(1)}\otimes\dots\otimes W_{\sigma^{-1}(n)}
\,.
\end{equation}
\begin{lemma}\label{20170821:lem}
Let\/ $g_i\colon U_i\to V_i$, $i=1,\dots,n$, be linear maps of vector superspaces,
and let\/ $u_i\in U_i$, $i=1,\dots,n$.
For every\/ $\sigma\in S_n$, we have
\begin{equation}\label{20170821:eq1}
\sigma\big((g_1\otimes\dots\otimes g_n)(u_1\otimes\dots\otimes u_n)\big)
=
\big(\sigma(g_1\otimes\dots\otimes g_n)\big)(\sigma(u_1\otimes\dots\otimes u_n))
\,.
\end{equation}
\end{lemma}
\begin{proof}
Since $S_n$ is generated by transpositions,
it suffices to prove that equation \eqref{20170821:eq1} holds for $\sigma=(s,s+1)$, $s=1,\dots,n-1$.
In this case it is straightforward.
%
\end{proof}

We also define the (left) action of the symmetric group $S_n$ on an arbitrary 
ordered $n$-tuple of objects $(x_1,\dots,x_n)$ as follows:
\begin{equation}\label{20170615:eq2}
\sigma(x_1,\dots,x_n)
=
(x_{\sigma^{-1}(1)},\dots,x_{\sigma^{-1}(n)})
\,.
\end{equation}
In other words,
we put the object $x_1$ in position $\sigma(1)$,
the object $x_2$ in position $\sigma(2)$, and so on.
(Note that, if the objects $x_1,\dots,x_n$ are the numbers $1,\dots,n$,
this action is obtained by applying not $\sigma$ to each of the entries of the list,
but $\sigma^{-1}$.)

\subsection{Composition of permutations}\label{sec:2.2}

Let $n\geq1$ and $m_1,\dots,m_n\geq0$.
Given permutations $\sigma\in S_n$, $\tau_1\in S_{m_1},\dots,\tau_n\in S_{m_n}$,
we want to define their \emph{composition}
$\sigma(\tau_1,\dots,\tau_n)\in S_{m_1+\dots+m_n}$.
To describe it, it is easier to say how it acts on 
the tensor power $V^{\otimes(m_1+\dots+m_n)}$ of the vector superspace $V$,
in analogy to \eqref{eq:operad6}.
Let
\begin{equation}\label{20170821:eq2a}
M_i=\sum_{j=1}^im_j
\,,\,\,
i=0,\dots,n
\,.
\end{equation}
To apply $\sigma(\tau_1,\dots,\tau_n)$ to $v$,
we first apply each $\tau_i\in S_{m_i}$ to the vector 
$w_i=v_{M_{i-1}+1}\otimes\dots\otimes v_{M_i}\,\in V^{\otimes m_i}$ via \eqref{eq:operad6},
and then we apply $\sigma\in S_n$ to $w=\tau_1(w_1)\otimes\dots\otimes \tau_n(w_n)$,
again by the same formula \eqref{eq:operad6},
where we view $w$ as an element of $W_1\otimes\dots\otimes W_n$,
with $W_i=V^{\otimes m_i}$,
and we consider the generalization of \eqref{eq:operad6} defined in \eqref{eq:operad18}.
Summarizing this in a formula, we have:
\begin{equation}\label{eq:operad19}
(\sigma(\tau_1,\dots,\tau_n))(v)
=
\sigma\big(
\tau_1(v_{1}\otimes\dots\otimes v_{M_1})
\otimes\dots\otimes
\tau_n(v_{M_{n-1}+1}\otimes\dots\otimes v_{M_n})
\big)\,.
\end{equation}
\begin{remark}\label{rem:operad1}
We can write explicitly how $\sigma(\tau_1,\dots,\tau_n)\in S_{m_1+\dots+m_n}$
permutes the integers $1,\dots,m_1+\dots+m_n$.
An integer $k\in\{1,\dots,m_1+\dots+m_n\}$ can be uniquely decomposed in the form
\begin{equation}\label{eq:operad16}
k=m_1+\dots+m_{i-1}+j
\,,
\end{equation}
with $1\leq i\leq n$ and $1\leq j\leq m_i$.
Then, we have
\begin{equation}\label{eq:operad17}
(\sigma(\tau_1,\dots,\tau_n))(k)=m_{\sigma^{-1}(1)}+\dots+m_{\sigma^{-1}(\sigma(i)-1)}+\tau_{i}(j)
\,.
\end{equation}
\end{remark}
\begin{proposition}\label{20170607:prop1}
The composition of permutations satisfies the following associativity condition:
given\/ $\sigma\in S_n$, $\tau_i\in S_{m_i}$ for\/ $i=1,\dots,n$,
and\/ $\rho_j\in S_{\ell_j}$ for $j=1,\dots,M_n$,
we have
\begin{equation}\label{eq:operad22}
\begin{array}{l}
\displaystyle{
\vphantom{\Big(}
(\sigma(\tau_1,\dots,\tau_\ell))(\rho_1,\dots,\rho_{M_n})
} \\
\displaystyle{
\vphantom{\Big(}
=
\sigma\big(\tau_1(\rho_1,\dots,\rho_{M_1}),\dots,\tau_n(\rho_{M_{n-1}+1},\dots,\rho_{M_n})\big)
\,\in S_{\sum_{j}\ell_j}
\,.}
\end{array}
\end{equation}
\end{proposition}
\begin{proof}
Take a monomial $v_1\otimes\dots\otimes v_{L_{M_n}}$,
where we define $M_i$ as in \eqref{20170821:eq2a}, and let
\begin{equation}\label{20170821:eq2b}
L_j=\sum_{k=1}^j \ell_k
\,,\,\,
j=0,\dots,M_n
\,.
\end{equation}
By \eqref{eq:operad19}, 
when we apply either side of \eqref{eq:operad22} to such monomial,
we get
$$
\begin{array}{l}
\displaystyle{
\vphantom{\big(}
\sigma\Big(
\tau_1\big(
\rho_1(v_1\times\dots\times v_{L_1})
\otimes\dots\otimes
\rho_{M_1}(v_{L_{M_1-1}+1}\otimes\dots\otimes v_{L_{M_1}})
\big)
\otimes\dots
} \\
\displaystyle{
\vphantom{\big(}
\dots\otimes
\tau_n\big(
\rho_{M_{\!n\!-\!1}\!+1}(v_{L_{\!M_{\!n\!-\!1}}\!\!+1}\!\otimes\dots\otimes v_{L_{M_{\!n\!-\!1}\!\!+1}})
\otimes\dots\otimes
\rho_{M_n}(v_{L_{M_n\!-\!1}\!+1}\!\otimes\dots\otimes v_{L_{M_n}})
\big)
\!\Big)
.}
\end{array}
$$
The claim follows.
\end{proof}
\begin{proposition}\label{20170607:prop2}
The composition of permutations satisfies the following
equivariance condition:
\begin{equation}\label{eq:operad23}
(\varphi\sigma)(\psi_1\tau_1,\dots,\psi_n\tau_n)
=
\varphi(\psi_{\sigma^{-1}(1)},\dots,\psi_{\sigma^{-1}(n)})
\,
\sigma(\tau_1,\dots,\tau_n)
\,,
\end{equation}
for every\/ $\varphi,\sigma\in S_n$,
$\psi_1,\tau_1\in S_{m_1}$, $\dots$, $\psi_n,\tau_n\in S_{m_n}$.
\end{proposition}
\begin{proof}
It suffices to show that both sides of \eqref{eq:operad23}
give the same result when applied to a monomial in $V^{\otimes(m_1+\dots+m_n)}$.
When we apply the left-hand side of \eqref{eq:operad23}
to $v_1\otimes\dots\otimes v_{M_n}$,
we get
$$
\begin{array}{l}
\displaystyle{
\vphantom{\Big(}
\big((\varphi\sigma)(\psi_1\tau_1,\dots\psi_n\tau_n)\big)
(v_1\otimes\dots\otimes v_{M_n})
} \\
\displaystyle{
\vphantom{\Big(}
=
(\varphi\sigma)\big(
(\psi_1\tau_1)(v_1\otimes\dots\otimes v_{M_1})
\otimes\dots\otimes
(\psi_n\tau_n)(v_{M_{n-1}+1}\otimes\dots\otimes v_{M_n})
\big)
\,.}
\end{array}
$$
On the other hand, if we apply the right-hand side of \eqref{eq:operad23} 
to the same monomial, we get, 
$$
\begin{array}{l}
\displaystyle{
\vphantom{\Big(}
\big(\varphi(\psi_{\sigma^{-1}(1)},\dots,\psi_{\sigma^{-1}(n)})\big)
\big(\sigma(\tau_1,\dots,\tau_n)\big)
(v_1\otimes\dots\otimes v_{M_n})
} \\
\displaystyle{
\vphantom{\Big(}
=
\!
\big(\varphi(\psi_{\sigma^{-1}\!(1)},\dots,\psi_{\sigma^{-1}\!(n)})\big)
\Big(\!
\sigma\big(
\tau_1(v_1\!\otimes\dots\otimes\! v_{M_1})\!
\otimes\dots\otimes\!
\tau_n(v_{M_{\!n\!-1}\!\!+1}\!\otimes\dots\otimes\! v_{M_n})
\big)
\!\Big)
} \\
\displaystyle{
\vphantom{\Big(}
=
\varphi
\Big(
\sigma\Big(
\psi_1(\tau_1(v_1\otimes\dots\otimes v_{M_1}))
\otimes\dots\otimes
\psi_n(\tau_n(v_{M_{n-1}+1}\otimes\dots\otimes v_{M_n}))
\Big)
\Big)
\,.}
\end{array}
$$
For the second equality we used \eqref{20170615:eq2}
and Lemma \ref{20170821:lem}.
Equation \eqref{eq:operad23} follows.
\end{proof}

\subsection{$\circ_i$-products of permutations}\label{sec:2.3}

For $i=1,\dots,n$, we define the $\circ_i$ product of permutations
$\circ_i\colon S_n\times S_m\to S_{n+m-1}$ as follows
($\beta\in S_n$, $\alpha\in S_m$):
\begin{equation}\label{eq:perm1}
\beta\circ_i\alpha:=\beta(
\overbrace{1,\dots,1}^{i-1},\stackrel{\vphantom{\Big(}i}{\alpha},\overbrace{1,\dots,1}^{n-i}
)
\,.
\end{equation}
In other words, its action on the tensor power $V^{\otimes{m+n-1}}$ 
of the vector superspace $V$, is given by
\begin{equation}\label{eq:perm2}
(\beta\circ_i\alpha)(v_1\otimes\cdots\otimes v_{n+m-1})
=
\beta(v_1\otimes\cdots\stackrel{\vphantom{\Big(}i}{w}\cdots\otimes v_{n+m-1})
\,,\,\,
w=\alpha(v_i\otimes\,\cdots\,\otimes v_{i+m-1})
\,.
\end{equation}

As a consequence of Proposition \ref{20170607:prop1},
the $\circ_i$-products satisfy the following associativity conditions:
($\gamma\in S_n,\,\beta\in S_m,\,\alpha\in S_\ell$, $i=1,\dots,n$, $j=1,\dots,n+m-1$):
\begin{equation}\label{eq:operad25b}
(\gamma\circ_i\beta)\circ_j\alpha
=
\left\{\begin{array}{ll}
(\gamma\circ_j\alpha)\circ_{\ell+i-1}\beta & \text{ if } 1\leq j<i \,,\\
\gamma\circ_i(\beta\circ_{j-i+1}\alpha) & \text{ if } i\leq j<i+m \,,\\
(\gamma\circ_{j-m+1}\alpha)\circ_i\beta & \text{ if } i+m\leq j<n+m\,.
\end{array}\right.
\end{equation}
In particular, the $\circ_1$-product is associative.
Moroever, as a consequence of Proposition \ref{20170607:prop2},
the $\circ_i$-products satisfy the following equivariance condition
($\beta,\sigma\in S_n$, $\alpha,\tau\in S_m$ $i=1,\dots,n$):
\begin{equation}\label{eq:operad9b}
(\beta\sigma)\circ_i (\alpha\tau)
=
(\beta\circ_{\sigma(i)}\alpha)(\sigma\circ_i\tau)
\,.
\end{equation}

We shall denote the identity element of the symmetric group $S_n$ by $1_n$.
For every $m,n\geq1$ and $i=1,\dots,n$ we have
\begin{equation}\label{eq:perm3}
1_n\circ_i 1_m=1_{n+m-1}
\,.
\end{equation}
By \eqref{eq:operad9b}, we have 
$(1_n\circ_i\alpha)(1_n\circ_i\tau)=1_n\circ_i(\alpha\tau)$.
In other words, 
for each $n,m\geq1$, and each $i=1,\dots,n$, we have
the injective group homomorphism
\begin{equation}\label{eq:perm4}
S_m\hookrightarrow S_{n+m-1}
\,\,,\,\,\,\,
\alpha\mapsto 1_n\circ_i\alpha
\,.
\end{equation}
For $\alpha\in S_m$ and $i=1,\dots,n$, 
we can write explicitly the action of $1_n\circ_i\alpha\in S_{n+m-1}$
on $V^{\otimes(n+m-1)}$:
$$
(1_n\circ_i\alpha)(v_1\otimes\cdots\otimes v_{n+m-1})
=
v_1\otimes\cdots\otimes
\alpha(v_i\otimes\cdots\otimes v_{m-1+i})
\otimes\cdots\otimes v_{m+n-1}
\,.
$$
In particular, for $\alpha\in S_m$ and $\beta\in S_{n}$,
the actions of 
$$
1_{n+1}\circ_1\alpha
\,\text{ and }\,
1_{m+1}\circ_{m+1}\beta
\,\,\in S_{m+n}
\,\,\text{ commute.}
$$
In the special case $i=1$ it is particularly easy to describe
the image of the map \eqref{eq:perm4}
as a permutation of the numbers $\{1,\dots,n+m-1\}$.
We have, for $m,n\geq1$ and $\alpha\in S_m$,
\begin{equation}\label{eq:perm7}
(1_n\circ_1\alpha)(i)=
\left\{\begin{array}{ll}
\alpha(i)\, & \text{ if }\,\,1\leq i\leq m\,, \\
i & \text{ if }\, m+1\leq i\leq m+n-1\,.
\end{array}\right.
\end{equation}
By \eqref{eq:operad9b}, we also have 
\begin{equation}\label{20170608:eq7a}
(\beta\circ_{\sigma(i)}1_m)(\sigma\circ_i1_m)=(\beta\sigma)\circ_i1_m
\,.
\end{equation}
Hence, the injective map
$S_n\hookrightarrow S_{n+m-1}$ mapping $\sigma\mapsto\sigma\circ_i1_m$
is not a group homomorphism.
On the other hand, it becomes a group homomorphism when we restrict to the 
stabilizer $(S_n)_i=\{\sigma\in S_n\,|\,\sigma(i)=i\}$ of $i$:
\begin{equation}\label{eq:perm5}
(S_n)_i\hookrightarrow S_{n+m-1}
\,\,,\,\,\,\,
\sigma\mapsto\sigma\circ_i1_m
\,.
\end{equation}
As special cases of \eqref{eq:operad25b}
(with $\alpha=1_{\ell}$, $\gamma=1_2$, $i=2$ and $j=1$),
we get the following identity, which we shall need later
($\alpha\in S_m$):
\begin{equation}\label{eq:perm14}
1_{\ell+1}\circ_{\ell+1}\alpha
=
(1_2\circ_{2}\alpha)\circ_11_{\ell}
\,.
\end{equation}

Note that we can write the cyclic permutation $(1,\dots,m+1)$ mapping 
$1\mapsto 2\mapsto\dots\mapsto m+1\mapsto1$ in terms of the $\circ_i$-products as follows:
\begin{equation}\label{michelino:eq1}
(1,\dots,m+1)=(1,2)\circ_11_m\,\,\text{ in }\,\, S_{m+1}
\,.
\end{equation}
More generally, if we consider the cyclic permutation $(1,\dots,m+1)$
in the permutation group $S_{m+n}$, we have the identity
\begin{equation}\label{michelino:eq2}
(1,\dots,m+1)
=
1_n\circ_1(1,2)\circ_11_m\,\,\text{ in }\,\, S_{m+n}
\end{equation}
(there is no need to put parentheses in the right-hand side since the $\circ_1$-product is associative).
An equivalent way to write equation \eqref{michelino:eq2} is
\begin{equation}\label{michelino:eq3}
(1,\dots,m+1)
=
(1,2)(1_m,1_1,1_{n-1})
\,\,\text{ in }\,\, S_{m+n}
\,,
\end{equation}
where, in this case, we consider the transposition $(1,2)$ as an element of $S_3$.

\subsection{Shuffles}\label{sec:2.4}

A permutation $\sigma\in S_{m+n}$ is called an $(m,n)$-\emph{shuffle} if
\begin{equation}\label{eq:perm6}
\sigma(1)<\dots<\sigma(m)
\,,\,\,
\sigma(m+1)<\dots<\sigma(m+n)
\,.
\end{equation}
In equivalent terms,
when acting on the tensor power $V^{\otimes(m+n)}$ of the (purely even) vector space $V$,
a shuffle $\sigma$ maps the monomial $v=v_1\otimes\,\cdots\,\otimes v_{m+n}$
to a permuted monomial in which the factors $v_1,\dots,v_m$ appear in their order:
$$
\sigma(v)
=
\cdots\otimes\stackrel{\vphantom{\Big(}\sigma(1)}{v_1}\otimes\cdots\otimes
\stackrel{\vphantom{\Big(} \sigma(2)}{v_2}\otimes\cdots
\otimes\stackrel{\vphantom{\Big(} \sigma(m)}{v_m}\otimes\cdots
\,,
$$
and the factors $v_{m+1},\dots,v_{m+n}$ appear in order.
We shall denote by $S_{m,n}\subset S_{m+n}$ the subset (it is not a subgroup) of $(m,n)$-shuffles.
By definition, $S_{n,0}=S_{0,n}=\{1\}$ for every $n\geq0$
and, by convention, we let $S_{m,n}=\emptyset$ if either $m$ or $n$ is negative.

Similarly, we shall denote by $S_{\ell,m,n}\subset S_{\ell+m+n}$
the subset of $(\ell,m,n)$-shuffles, i.e. permutations $\sigma\in S_{\ell+m+n}$ satisfying
\begin{equation}\label{eq:perm6b}
\sigma(1)<\dots<\sigma(\ell)
\,,\,\,
\sigma(\ell+1)<\dots<\sigma(\ell+m)
\,,\,\,
\sigma(\ell+m+1)<\dots<\sigma(\ell+m+n)
\,,
\end{equation}
and the same for $(m_1,\dots,m_k)$-shuffles in $S_{m_1+\dots+m_k}$,
for arbitrary $k\geq2$.
\begin{proposition}\label{prop:perm0}
\begin{enumerate}[(a)]
\item
We have a bijection\/
$S_{m,n}\stackrel{\sim}{\longrightarrow}S_{n,m}$
given by\/
$\sigma\mapsto\sigma\cdot(1,2)(1_n,1_m)$.
\item
We have a bijection\/
$S_{\ell,m,n}\stackrel{\sim}{\longrightarrow}S_{m,\ell,n}$
given by\/
$\sigma\mapsto\sigma\cdot(1,2)(1_m,1_\ell,1_n)$.
\end{enumerate}
\end{proposition}
\begin{proof}
The permutation $(1,2)(1_n,1_m)$ switches the first $n$ factors of $V^{\otimes(m+n)}$
with the last $m$ factors, i.e. it maps
$$
\begin{array}{cccccc}
1&\dots&n&1+n&\dots&m+n \\
\downarrow&&\downarrow&\downarrow&&\downarrow \\
m+1&\dots&m+n&1&\dots&m
\end{array}
$$
Hence, the product $\sigma\cdot(1,2)(1_n,1_m)$ maps
$$
\begin{array}{cccccc}
1&\dots&n&1+n&\dots&m+n \\
\downarrow&&\downarrow&\downarrow&&\downarrow \\
\sigma(m+1)\,<&\dots\,<&\sigma(m+n)\,,&\sigma(1)\,<&\dots\,<&\sigma(m)
\end{array}
$$
so it lies in $S_{n,m}$, provided that $\sigma\in S_{m,n}$.
Claim (a) follows.
Similarly for claim (b).
\end{proof}
\begin{proposition}\label{prop:perm1}
For\/ $\ell,m,n\geq1$ we have the following bijections:
\begin{enumerate}[(a)]
\item
$S_{m,n}\times S_m\times S_n\stackrel{\sim}{\longrightarrow} S_{m+n}$,
mapping 
$$
\sigma,\alpha,\beta\mapsto\sigma\,\cdot(1_{n+1}\circ_1\alpha)\cdot(1_{m+1}\circ_{m+1}\beta)
\,,
$$
where\/ $\cdot$ denotes the product in the symmetric group\/ $S_{m+n}$.
\item
$S_{\ell+m,n}\times S_{\ell,m}\stackrel{\sim}{\longrightarrow} S_{\ell,m,n}$,
mapping 
$$
\sigma,\tau\mapsto\sigma\cdot(1_{n+1}\circ_1\tau)
\,.
$$
\item
$S_{m,n}\times S_{\ell,m+n}\stackrel{\sim}{\longrightarrow} S_{\ell,m,n}$,
mapping 
$$
\sigma,\tau\mapsto\tau\cdot(1_{\ell+1}\circ_{\ell+1}\sigma)
\,.
$$
\end{enumerate}
\end{proposition}
\begin{proof}
First, the two sets $S_{m,n}\times S_m\times S_n$ and $S_{m+n}$
have the same cardinality $(m+n)!$.
Hence, to prove that the map (a) is a bijection it suffices to prove that it is injective or surjective.
On the other hand, we can see how 
$$
X=\sigma\cdot(1_{n+1}\circ_1\alpha)\cdot(1_{m+1}\circ_{m+1}\beta)
$$
acts on the tensor power $V^{m+n-1}$,
of a vector space $V$.
First, we permute the factors of $v=v_1\otimes\,\dots\,\otimes v_m$ by $\alpha$
and the factors of $v_{m+1}\otimes\,\dots\,\otimes v_{m+n}$ by $\beta$.
Then, we shuffle the resulting monomial, 
in such a way that the factors of $\alpha(v_1\otimes\,\dots\,\otimes v_m)$
appear in the same order in $X(v)$, in positions $\sigma(1),\dots,\sigma(m)$,
and similarly the factors of $\beta(v_{m+1}\otimes\,\dots\,\otimes v_{m+n})$
appear in the same order in $X(v)$, in positions $\sigma(m+1),\dots,\sigma(m+n)$.
Now it is clear that the resulting monomial $X(v)$
is uniquely determined by the choice of $\sigma\in S_{m,n}$, $\alpha\in S_m$ and $\beta\in S_n$.
In other words, the map (a) is injective.

Let us now prove that the map (b) is bijective.
First note that two sets $S_{\ell+m,n}\times S_{\ell,m}$ and $S_{\ell,m,n}$
have the same cardinality $\frac{(\ell+m+n)!}{\ell!m!n!}$.
Next, we need to prove that,
for $\sigma\in S_{\ell+m,n}$ and $\tau\in S_{\ell,m}$,
the permutation $\sigma(1_{n+1}\circ_1\tau)$ is an $(\ell,m,n)$-shuffle.
Indeed, by \eqref{eq:perm7},
\begin{equation}\label{eq:perm8}
1\leq (1_{n+1}\circ_1\tau)(i)=\tau(i)\leq \ell+m
\,\text{ for }\,
i=1,\dots,\ell+m
\,,
\end{equation}
and
\begin{equation}\label{eq:perm9}
(1_{\ell+1}\circ_1\tau)(i)=i
\,\text{ for }\,
i=\ell+m+1,\dots,\ell+m+n
\,.
\end{equation}
On the other hand, since $\tau\in S_{\ell,m}$, we have
\begin{equation}\label{eq:perm10}
1\leq\tau(1)<\dots<\tau(\ell)\leq \ell+m
\,\text{ and }\,
1\leq\tau(\ell+1)<\dots<\tau(\ell+m)\leq \ell+m
\,,
\end{equation}
and since $\sigma\in S_{\ell+m,n}$, we have
\begin{equation}\label{eq:perm11}
\sigma(1)<\dots<\sigma(\ell+m)
\,\text{ and }\,
\sigma(\ell+m+1)<\dots<\sigma(\ell+m+n)
\,.
\end{equation}
Combining \eqref{eq:perm8}--\eqref{eq:perm11},
we get
\begin{equation}\label{eq:perm12}
\begin{array}{l}
\displaystyle{
\vphantom{\Big(}
\sigma(1_{n+1}\circ_1\tau)(1)=\sigma(\tau(1))<\dots<\sigma(1_{n+1}\circ_1\tau)(\ell)=\sigma(\tau(\ell))
\,,} \\
\displaystyle{
\vphantom{\Big(}
\sigma(1_{n+1}\circ_1\tau)(\ell+1)=\sigma(\tau(\ell+1))<\dots<\sigma(1_{n+1}\circ_1\tau)(\ell+m)
=\sigma(\tau(\ell+m))
\,,} \\
\displaystyle{
\vphantom{\Big(}
\!
\sigma(1_{n\!+\!1}\!\circ_1\!\tau)(\ell\!+\!m\!+\!1)=\sigma(\ell\!+\!m\!+\!1)<\dots<
\sigma(1_{n\!+\!1}\!\circ_1\!\tau)(\ell\!+\!m\!+\!n)=\sigma(\ell\!+\!m\!+\!n)
,} 
\end{array}
\end{equation}
namely, $\sigma(1_{n+1}\circ_1\tau)\in S_{\ell,m,n}$.
To prove that the map (b) is injective,
we just observe that, by the third line in \eqref{eq:perm12},
the values of $\sigma(1_{n+1}\circ_1\tau)$ on $\ell+m+1,\dots,\ell+m+n$
uniquely determine $\sigma(\ell+m+1),\dots,\sigma(\ell+m+n)$,
i.e. uniquely determine the shuffle $\sigma\in S_{\ell+m,n}$.
Then, since $\sigma$ is uniquely determined by $\sigma(1_{n+1}\circ_1\tau)$, 
it is clear that $\tau$ is uniquely determined as well.
A similar proof works for (c).
\end{proof}
\begin{proposition}\label{prop:perm2}
\begin{enumerate}[(a)]
\item
The set of shuffles\/ $S_{m,n}$ decomposes as
$$
S_{m,n}
=
\{\sigma\in S_{m,n}\,|\,\sigma(1)=1\}
\sqcup
\{\sigma\in S_{m,n}\,|\,\sigma(m+1)=1\}
\,.
$$
\item
We have a bijection\/
$\{\sigma\in S_{m+1,n-1}\,|\,\sigma(1)=1\}
\stackrel{\sim}{\longrightarrow}\{\sigma\in S_{m,n}\,|\,\sigma(m+1)=1\}$
given by
\begin{equation}\label{michelino:eq4}
\sigma\mapsto\sigma\cdot(1_n\circ_1(1,2)\circ_11_{m})
\,,
\end{equation}
where $\cdot$ is the product in the symmetric group\/ $S_{m+n}$.
\item
We have a bijection
$S_{m-1,n}\stackrel{\sim}{\longrightarrow}\{\sigma\in S_{m,n}\,|\,\sigma(1)=1\}$
given by
\begin{equation}\label{michelino:eq5}
\sigma\mapsto1_2\circ_2\sigma
\,.
\end{equation}
\end{enumerate}
\end{proposition}
\begin{proof}
Claim (a) is obvious since, if $\sigma$ is an $(m,n)$-shuffle,
then either $1=\sigma(1)$ or $1=\sigma(m+1)$.
For (b), recall that, by \eqref{michelino:eq2},
$(1_n\circ_1(1,2)\circ_11_{m})$ is the cyclic permutation $1\mapsto2\mapsto\dots\mapsto m+1\mapsto1$.
Hence, 
the product $\sigma\cdot(1_n\circ_1(1,2)\circ_11_{m})$ maps
$$
\begin{array}{ccccccc}
1&\dots&m&m+1&m+2&\dots&m+n \\
\downarrow&&\downarrow&\downarrow&\downarrow&&\downarrow \\
\sigma(2)\,<&\dots\,<&\sigma(m+1)\,\,,&\sigma(1)=1\,<&\sigma(m+2)\,<&\dots\,<&\sigma(m+n)
\end{array}
$$
It follows that $\sigma\cdot(1_n\circ_1(1,2)\circ_11_{m})$ lies in $S_{m,n}$,
provided that $\sigma\in S_{m+1,n-1}$,
and that it maps $m+1\mapsto1$.
On the other hand, the map \eqref{michelino:eq4} is clearly injective,
hence it is bijective since the two sets 
$\{\sigma\in S_{m+1,n-1}\,|\,\sigma(1)=1\}$
and $\{\sigma\in S_{m,n}\,|\,\sigma(m+1)=1\}$
have the same cardinality $\frac{(m+n-1)!}{m!(n-1)!}$.
This proves (b).

Next, let us prove claim (c).
By the definition of the $\circ_i$-products,
we have
$$
(1_2\circ_2\sigma)(1)=1
\,,\,\text{ and }\,\,
(1_2\circ_2\sigma)(1+i)=1+\sigma(i) \text{ for } i=1,\dots,m+n-1\,.
$$
In particular, $(1_2\circ_2\sigma)\in S_{m,n}$ provided that $\sigma\in S_{m-1,n}$.
On the other hand, 
the map \eqref{michelino:eq5} is clearly injective,
hence it is bijective since the two sets 
$S_{m-1,n}$ and $\{\sigma\in S_{m,n}\,|\,\sigma(1)=1\}$
have the same cardinality $\frac{(m+n-1)!}{(m-1)!n!}$.
This proves (c).
Finally, claim (d) follows from (b) and (c).
\end{proof}

\section{Superoperads and the associated $\mb Z$-graded Lie superalgebras}\label{sec:3}

In this section, we review the definition and some basic properties of superoperads,
which will be needed throughout the rest of the paper.
For extended reviews on the theory of operads,
see e.g. \cite{LV12,MSS02}.

\subsection{Definition of a superoperad}\label{sec:3.1}

Recall that 
a (linear, unital, symmetric) \emph{superoperad} $\mc P$ is a collection of vector superspaces $\mc P(n)$, $n\geq0$, with parity $p$,
endowed, for every $f\in\mc P(n)$ and $m_1,\dots,m_n\geq0$, with the \emph{composition} 
parity preserving linear map, 
\begin{equation}\label{eq:operad1}
\begin{split}
\mc P(n) \otimes \mc P(m_1)\otimes\dots\otimes\mc P(m_n)\,\,&\to\,\,\mc P(M_n) \,, \\
f \otimes g_1 \otimes\dots\otimes g_n \,\, &\mapsto \,\, f(g_1\otimes\dots\otimes g_n) \,,
\end{split}
\end{equation}
where $M_n:=m_1+\dots+m_n$ (cf. \eqref{20170821:eq2a}),
satisfying the following associativity axiom:
\begin{equation}\label{eq:operad2}
f\big(
(g_1\otimes\dots\otimes g_n)
(h_1\otimes\dots\otimes h_{M_n})
\big)
=
\big(f
(g_1\!\otimes\dots\otimes\! g_n)
\big)
(h_1\otimes\dots\otimes h_{M_n})
\,\in\mc P\Big(\sum_{j=1}^{M_n}\ell_j\Big)
\,,
\end{equation}
for every $f\in\mc P(n)$, $g_i\in\mc P(m_i)$ for $i=1,\dots,n$,
and $h_{j}\in\mc P(\ell_{j})$ for $j=1,\dots,M_n$.
In the left-hand side of \eqref{eq:operad2}
the linear map 
$$\bigotimes_{i=1}^n g_i\colon \bigotimes_{j=1}^{M_n}\mc P(\ell_j)
\to\bigotimes_{i=1}^n\mc P\Bigl(\sum_{j=M_{i-1}+1}^{M_i}\ell_{j}\Bigr)$$
is the tensor product of composition maps, defined by \eqref{20170804:eq1}, applied to
$$
h_1\otimes\dots\otimes h_{M_n} = (h_1\otimes\dots\otimes h_{M_1}) \otimes (h_{M_1+1}\otimes\dots\otimes h_{M_2})
\otimes\dots\otimes (h_{M_{n-1}+1}\otimes\dots\otimes h_{M_n}).
$$
We assume that $\mc P$ is endowed with a \emph{unit} element $1\in\mc P(1)$
satisfying the following unity axioms: 
\begin{equation}\label{eq:operad3}
f(1\otimes\dots\otimes 1)=1(f)=f
\,,\,\,\text{ for every }\,\,f\in\mc P(n)
\,.
\end{equation}
Furthermore, we assume that,
for each $n\geq1$, $\mc P(n)$ has a right action of the symmetric group $S_n$,
denoted $f^\sigma$, for $f\in\mc P(n)$ and $\sigma\in S_n$,
satisfying the following equivariance axiom
($f\in\mc P(n)$, $g_1\in\mc P(m_1),\dots,g_n\in\mc P(m_n)$,
$\sigma\in S_n$, $\tau_1\in S_{m_1},\dots,\tau_n\in S_{m_n}$):
\begin{equation}\label{eq:operad4}
f^\sigma(g_1^{\tau_1}\otimes \dots\otimes g_n^{\tau_n})
=
\big(f(\sigma(g_1\otimes\dots\otimes g_n))\big)^{\sigma(\tau_1,\dots,\tau_n)}
\,,
\end{equation}
where the composition $\sigma(\tau_1,\dots,\tau_n)\in S_{m_1+\dots+m_n}$ 
of permutation was defined in Section \ref{sec:2.2},
and the left action of $\sigma\in S_n$ on the tensor product of vector superspaces
was defined in \eqref{eq:operad18}.

For simplicity, from now on, we will use the term operad in place of superoperad.

\begin{example}\label{ex:operadS}
The \emph{symmetric group} operad $\mc S$ is defined as 
the collection of purely even superspaces $\mc S(n)=\mb F[S_n]$
for $n\geq1$ and $\mc S(0)=\mb F$,
with the composition maps obtained by exteding linearly \eqref{eq:operad19}
to the group algebras,
the unity $1\in S_1$,
and the action of right multiplication of $S_n$ on $\mb F[S_n]$.
This is an operad,
indeed the associativity axiom \eqref{eq:operad2} follows from Proposition \ref{20170607:prop1},
and the equivariance axiom \eqref{eq:operad4}
follows from Proposition \ref{20170607:prop2}.
\end{example}
\begin{example}\label{ex:operadH}
Given a vector superspace $V=V_{\bar 0}\oplus V_{\bar 1}$,
the operad $\mc{H}om$ is defined as 
the collection of superspaces
$\mc{H}om(n)=\Hom(V^{\otimes n},V)$, $n\geq0$,
endowed with the composition maps
($f\in\mc P(n)$, $g_i\in\mc P(m_i)$ for $i=1,\dots,n$,
$v_j\in V$ for $j=1,\dots,M_n:=m_1+\dots+m_n$):
\begin{equation}\label{eq:operad24}
(f(g_1\otimes \dots\otimes g_n))
(v_1\otimes\dots\otimes v_{M_n})
:=
f\big((g_1\otimes\dots\otimes g_n)(v_1\otimes\dots\otimes v_{M_n})\big)
\,,
\end{equation}
the unity $1=\id_V\in\End V$,
and the right action of $S_n$ on $\Hom(V^{\otimes n},V)$
given by  \eqref{eq:operad7}.
The associativity and unit axioms, for this example, are obvious.
Let us prove the equivariance axiom \eqref{eq:operad4}.
When applied to a monomial $v_1\otimes\dots\otimes v_{M_n}$,
the left-hand side of \eqref{eq:operad4} gives
(we use the notation \eqref{20170821:eq2a}) 
$$
\begin{array}{l}
\displaystyle{
\vphantom{\Big(}
\big(f^\sigma(g_1^{\tau_1}\otimes\dots\otimes g_n^{\tau_n})\big)
(v_1\otimes\dots\otimes v_{M_n})
} \\
\displaystyle{
\vphantom{\Big(}
=
f^{\sigma}\big(
(g_1^{\tau_1}\otimes\cdots\otimes g_n^{\tau_n})
(v_1\otimes\cdots\otimes v_{M_n})
\big)
} \\
\displaystyle{
\vphantom{\Big(}
=
f^{\sigma}\!\big(
(g_1\!\otimes\!\cdots\!\otimes g_n)
\big(
\tau_1(v_1\!\otimes\!\cdots\!\otimes v_{M_1})\!
\otimes\!\cdots\!\otimes
\tau_n(v_{M_{n-1}\!+\!1}\otimes\!\cdots\!\otimes v_{M_n})
\big)
} \\
\displaystyle{
\vphantom{\Big(}
=
f\big(
(\sigma(g_1\otimes\cdots\otimes g_n))
\big(
\sigma(\tau_1,\dots,\tau_n)
(v_1\otimes\cdots\otimes v_{M_n})
\big)
} \\
\displaystyle{
\vphantom{\Big(}
=
\big(f
(\sigma(g_1\otimes\dots\otimes g_n))
\big)^{\sigma(\tau_1,\dots,\tau_n)}
(v_1\otimes\dots\otimes v_{M_n})
\,.}
\end{array}
$$
For the third equality, we used Lemma \ref{20170821:lem}
and the definition \eqref{eq:operad19} of $\sigma(\tau_1,\dots,\tau_n)$.
\end{example}

Given an operad $\mc P$, one defines, for each $i=1,\dots,n$, 
the $\circ_i$-product
$\circ_i\colon\mc P(n)\times\mc P(m)\to\mc P(n+m-1)$
by insertion in position $i$, i.e.
\begin{equation}\label{eq:operad8}
f\circ_i g=f(
\overbrace{1\otimes\dots\otimes 1}^{i-1}
\otimes\stackrel{\vphantom{\Big(}i}{g}\otimes
\overbrace{1\otimes\dots\otimes1}^{n-i})\,.
\end{equation}
Of course, knowing all the $\circ_i$-products allows to reconstruct,
thanks to the associativity axiom \eqref{eq:operad2}, the whole operad structure, by
\begin{equation}\label{eq:operad8a}
f(g_1,\dots,g_n)
=
(\cdots((f\circ_1 g_1)\circ_{m_1+1}g_2)\cdots)\circ_{m_1+\dots+m_{n-1}+1} g_n
\,.
\end{equation}
Then, the unity axiom (i) becomes $1\circ_1 f=f\circ_i1=f$ for every $i=1,\dots,n$,
and the associativity axiom (ii) is equivalent to the following identities for the $\circ_i$-products,
cf. \eqref{eq:operad25b}
($f\in\mc P(n),\,g\in\mc P(m),\,h\in\mc P(\ell)$):
\begin{equation}\label{eq:operad25}
(f\circ_ig)\circ_j h
=
\left\{\begin{array}{ll}
\displaystyle{
\vphantom{\Big(}
(-1)^{p(g)p(h)}
(f\circ_j h)\circ_{\ell+i-1}g } & \text{ if } 1\leq j<i 
\,, \\
\displaystyle{
\vphantom{\Big(}
f\circ_i(g\circ_{j-i+1}h) } & \text{ if } i\leq j<i+m 
\,, \\
\displaystyle{
\vphantom{\Big(}
(-1)^{p(g)p(h)}
(f\circ_{j-m+1}h)\circ_ig } & \text{ if } i+m\leq j<n+m
\,.
\end{array}\right.
\end{equation}
In particular, the $\circ_1$-product is associative.
Note that the third identity in \eqref{eq:operad25}
is equivalent to the first one by flipping the equality.
Furthermore, the equivariance condition \eqref{eq:operad4}
and the supersymmetric equivariance condition \eqref{eq:operad4}
both become, in terms of the $\circ_i$-products, cf. \eqref{eq:operad9b}
\begin{equation}\label{eq:operad9}
f^{\sigma}\circ_i g^{\tau}
=
(f\circ_{\sigma(i)}g)^{\sigma\circ_i\tau}
\,,
\end{equation}
for $f\in\mc P(n)$, $g\in\mc P(m)$, $\sigma\in S_n$, $\tau\in S_m$,
where $\sigma\circ_i\tau\in S_{n+m-1}$ is defined by \eqref{eq:perm1}.

By definition, an operad $\mc P$ is \emph{filtered} if each vector superspace $\mc P(n)$ 
is endowed with a filtration $\fil^r\mc P(n)$, which is preserved by the action 
of the symmetric group and is preserved by the composition maps, i.e.,
\begin{equation}\label{eq:filop}
f^\sigma\in\fil^r \mc P(n), \quad f(g_1\otimes\dots\otimes g_n) \in\fil^{r+s_1+\cdots+s_n} \mc P(m_1+\cdots+m_n)
\end{equation}
for all $f\in\fil^r \mc P(n)$, $\sigma\in S_n$ and $g_i\in\fil^{s_i} \mc P(m_i)$.
In the case of a decreasing filtration of $\mc P$, the corresponding \emph{associated graded} operad is
\begin{equation}\label{eq:filop2}
\gr^r \mc P(n) = \fil^r \mc P(n) / \fil^{r+1} \mc P(n),
\end{equation}
with the induced action of the symmetric groups and composition maps.
This is a \emph{graded operad}, i.e., each superspace $\mc P(n)$ is graded and the analog of \eqref{eq:filop} holds degreewise.

A \emph{morphism} $\ph$ from an operad $\mc P$ to an operad $\mc Q$ is a collection of linear maps $\ph_n\colon\mc P(n)\to\mc Q(n)$ commuting with the action of the symmetric groups and compatible with the composition maps. When $\mc P$ and $\mc Q$ are filtered, the map $\ph_n$ is required to send $\fil^r\mc P(n)$ to $\fil^r\mc Q(n)$, and similarly for graded operads.

\subsection{Universal Lie superalgebra associated to an operad}\label{sec:3.2}

Let $\mc P$ be an operad.
We let $W=W(\mc P)$ be the $\mb Z$-graded vector superspace
$W=\bigoplus_{n\geq-1}W_n$,
where
\begin{equation}\label{20170603:eq3}
W_n=\mc P(n+1)^{S_{n+1}}
=\big\{f\in\mc P(n+1)\,\big|\,f^\sigma=f\,\forall \sigma\in S_{n+1}\big\}
\,.
\end{equation}
We define the $\Box$-product of $f\in W_n$ and $g\in W_m$ as follows:
\begin{equation}\label{eq:box}
f\Box g
=
\sum_{\sigma\in S_{m+1,n}}
(f\circ_1 g)^{\sigma^{-1}}
\,.
\end{equation}
Note that $S_{m+1,-1}=\emptyset$, hence $f\Box g=0$ if $f\in W_{-1}$.
\begin{example}\label{ex:box}
For $f,g\in W_1$, we have
\begin{equation}\label{eq:box1}
\begin{split}
f\Box g
&= f\circ_1 g + (f\circ_1 g)^{(23)} + (f\circ_1 g)^{(132)} \\
&= f\circ_1 g + f\circ_2 g + (f\circ_2 g)^{(12)}.
\end{split}
\end{equation}
The second equality follows from \eqref{eq:operad9} and the fact that $f$ and $g$ are symmetric, using
$(23)=(132)(12)$ and $(132)=(12)\circ_2 (1)$.
\end{example}
\begin{theorem}\label{20170603:thm2}
\begin{enumerate}[(a)]
\item
For\/ $f\in W_n$ and\/ $g\in W_m$, we have\/ $f\Box g\in W_{n+m}$.
\item
The associator of the\/ $\Box$-product is right supersymmetric, i.e.
\begin{equation}\label{20170608:eq2}
(f\Box g)\Box h-f\Box(g\Box h)
=(-1)^{p(g)p(h)}
(f\Box h)\Box g-f\Box(h\Box g)
\,.
\end{equation}
\item
Consequently, $W$ is a\/ $\mb Z$-graded Lie superalgebra
with Lie bracket given by\/ ($f\in W_n$, $g\in W_m$)
\begin{equation}\label{20170603:eq4}
[f,g]=f\Box g-(-1)^{p(f)p(g)}g\Box f
\,.
\end{equation}
\end{enumerate}
\end{theorem}
\begin{proof}
For claim (a), we need to prove that $f\Box g$ is fixed by the action of the symmetric group $S_{m+n+1}$.
We have
\begin{equation}\label{20170608:eq1}
\begin{array}{l}
\displaystyle{
\vphantom{\Big(}
\sum_{\sigma\in S_{m+n+1}}
(f\circ_1 g)^{\sigma^{-1}}
} \\
\displaystyle{
\vphantom{\Big(}
=
\sum_{\substack{\beta\in S_{m+1},\alpha\in S_n, \\ \sigma\in S_{m+1,n}}}
(f\circ_1 g)^{\big(\sigma\cdot(1_{n+1}\circ_1\beta)\cdot(1_{m+2}\circ_{m+2}\alpha)\big)^{-1}}
} \\
\displaystyle{
\vphantom{\Big(}
=
\sum_{\substack{\beta\in S_{m+1},\alpha\in S_n, \\ \sigma\in S_{m+1,n}}}
(f\circ_1 g)^{(1_{m+2}\circ_{m+2}\alpha^{-1})\cdot(1_{n+1}\circ_1\beta^{-1})\cdot\sigma^{-1}}
} \\
\displaystyle{
\vphantom{\Big(}
=
\sum_{\substack{\beta\in S_{m+1},\alpha\in S_n, \\ \sigma\in S_{m+1,n}}}
(f\circ_1 g)^{((1_{2}\circ_{2}\alpha^{-1})\circ_11_{m+1})\cdot(1_{n+1}\circ_1\beta^{-1})\cdot\sigma^{-1}}
} \\
\displaystyle{
\vphantom{\Big(}
=
\sum_{\substack{\beta\in S_{m+1},\alpha\in S_n, \\ \sigma\in S_{m+1,n}}}
(f^{(1_{2}\circ_{2}\alpha^{-1})}\circ_1 g^{\beta^{-1}})^{\sigma^{-1}}
} \\
\displaystyle{
\vphantom{\Big(}
=
(m+1)!n!\sum_{\sigma\in S_{m+1,n}}
(f\circ_1 g)^{\sigma^{-1}}
=
(m+1)!n!\,\,f\Box g
\,,}
\end{array}
\end{equation}
where we used
Proposition \ref{prop:perm1}(a) for the first equality,
the homomorphism property of the maps \eqref{eq:perm4} for the second equality,
identity \eqref{eq:perm14} for the third equality,
the equivariance conditions \eqref{eq:operad9} and the
obvious identities $(1_{2}\circ_{2}\alpha^{-1})(1)=1=1_{n+1}(1)$ for the fourth equality,
and the assumptions that $f\in W_n$ and $g\in W_m$ for the fifth equality.
Since the left-hand side of \eqref{20170608:eq1} is manifestly invariant with respect to the action of $S_{m+n+1}$,
we conclude that $f\Box g$ is invariant as well, proving (a).

Next, let us prove claim (b).
We have
\begin{equation}\label{20170608:eq3}
\begin{array}{l}
\displaystyle{
\vphantom{\Big(}
f\Box(g\Box h)
=
\sum_{\sigma\in S_{\ell+m+1,n}}\sum_{\tau\in S_{\ell+1,m}}
\big(f\circ_1(g\circ_1 h)^{\tau^{-1}}\big)^{\sigma^{-1}}
} \\
\displaystyle{
\vphantom{\Big(}
=
\sum_{\sigma\in S_{\ell+m+1,n}}\sum_{\tau\in S_{\ell+1,m}}
\big(f\circ_1(g\circ_1 h)\big)^{(1_{n+1}\circ_1\tau^{-1})\cdot\sigma^{-1}}
} \\
\displaystyle{
\vphantom{\Big(}
=
\sum_{\sigma\in S_{\ell+m+1,n}}\sum_{\tau\in S_{\ell+1,m}}
\big(f\circ_1(g\circ_1 h)\big)^{(\sigma\cdot(1_{n+1}\circ_1\tau))^{-1}}
} \\
\displaystyle{
\vphantom{\Big(}
=
\sum_{\sigma\in S_{\ell+1,m,n}}
\big(f\circ_1(g\circ_1 h)\big)^{\sigma^{-1}}
\,.}
\end{array}
\end{equation}
In the second equality we used the equivariance condition \eqref{eq:operad9},
in the third equality we used the fact that the map \eqref{eq:perm4} is a group homomorphism,
and in the fourth equality we used Proposition \ref{prop:perm1}(b).
On the other hand, we have
\begin{equation}\label{20170608:eq4}
(f\Box g)\Box h
=
\sum_{\sigma\in S_{m+1,n}}\sum_{\tau\in S_{\ell+1,m+n}}
\big((f\circ_1g)^{\sigma^{-1}}\circ_1 h\big)^{\tau^{-1}}
\,.
\end{equation}
By Proposition \ref{prop:perm2}(a),
the sum in the right-hand side of \eqref{20170608:eq4}
split as sum of two terms,
in the first one we sum over the shuffles $\sigma\in S_{m+1,n}$
such that $\sigma(1)=1$,
and in the second one we sum over the shuffles $\sigma\in S_{m+1,n}$
such that $\sigma(m+2)=1$.
The first term is
\begin{equation}\label{20170608:eq5}
\begin{array}{l}
\displaystyle{
\vphantom{\Big(}
\sum_{\substack{\sigma\in S_{m+1,n} \\ \text{s.t. } \sigma(1)=1}}\sum_{\tau\in S_{\ell+1,m+n}}
\big((f\circ_1g)^{\sigma^{-1}}\circ_1 h\big)^{\tau^{-1}}
} \\
\displaystyle{
\vphantom{\Big(}
=
\sum_{\substack{\sigma\in S_{m+1,n} \\ \text{s.t. } \sigma(1)=1}}\sum_{\tau\in S_{\ell+1,m+n}}
\big((f\circ_1g)\circ_1 h\big)^{(\sigma^{-1}\circ_11_{\ell+1})\cdot\tau^{-1}}
} \\
\displaystyle{
\vphantom{\Big(}
=
\sum_{\substack{\sigma\in S_{m+1,n} \\ \text{s.t. } \sigma(1)=1}}\sum_{\tau\in S_{\ell+1,m+n}}
\big((f\circ_1g)\circ_1 h\big)^{(\tau\cdot(\sigma\circ_11_{\ell+1}))^{-1}}
} \\
\displaystyle{
\vphantom{\Big(}
=
\sum_{\sigma\in S_{m,n}}\sum_{\tau\in S_{\ell+1,m+n}}
\big((f\circ_1g)\circ_1 h\big)^{(\tau\cdot((1_2\circ_2\sigma)\circ_11_{\ell+1}))^{-1}}
} \\
\displaystyle{
\vphantom{\Big(}
=
\sum_{\sigma\in S_{m,n}}\sum_{\tau\in S_{\ell+1,m+n}}
\big((f\circ_1g)\circ_1 h\big)^{(\tau\cdot(1_{\ell+2}\circ_{\ell+2}\sigma))^{-1}}

} \\
\displaystyle{
\vphantom{\Big(}
=
\sum_{\sigma\in S_{\ell+1,m,n}}
\big((f\circ_1g)\circ_1 h\big)^{\sigma^{-1}}
\,,}
\end{array}
\end{equation}
where we used the equivariance relation \eqref{eq:operad9} for the first equality,
the fact that the map \eqref{eq:perm5} is a homomorphism for the second equality,
Proposition \ref{prop:perm2}(c) for the third equality,
equation \eqref{eq:perm14} for the fourth equality,
and Proposition \ref{prop:perm1}(c) for the fifth equality.
Note that the right-hand side of \eqref{20170608:eq5} coincides with the right-hand side of \eqref{20170608:eq3}
since the $\circ_1$-product is associative.
The second term, where we take the sum over the shuffles $\sigma\in S_{m+1,n}$
such that $\sigma(m+2)=1$ in \eqref{20170608:eq4}, is
\begin{equation}\label{20170608:eq6}
\begin{array}{l}
\displaystyle{
\vphantom{\Big(}
\sum_{\substack{\sigma\in S_{m+1,n} \\ \text{s.t. }\sigma(m+2)=1}}\sum_{\tau\in S_{\ell+1,m+n}}
\big((f\circ_1g)^{\sigma^{-1}}\circ_1 h\big)^{\tau^{-1}}
} \\
\displaystyle{
\vphantom{\Big(}
=
\sum_{\substack{\sigma\in S_{m+1,n} \\ \text{s.t. }\sigma(m+2)=1}}\sum_{\tau\in S_{\ell+1,m+n}}
\big(
(f\circ_1g)\circ_{\sigma^{-1}(1)}h
\big)^{(\sigma^{-1}\circ_11_{\ell+1})\cdot\tau^{-1}}
} \\
\displaystyle{
\vphantom{\Big(}
=
\sum_{\substack{\sigma\in S_{m+1,n} \\ \text{s.t. }\sigma(m+2)=1}}\sum_{\tau\in S_{\ell+1,m+n}}
\big(
(f\circ_1g)\circ_{m+2}h
\big)^{(\sigma\circ_{m+2}1_{\ell+1})^{-1}\tau^{-1}}
} \\
\displaystyle{
\vphantom{\Big(}
=
\sum_{\substack{\sigma\in S_{m+1,n} \\ \text{s.t. }\sigma(m+2)=1}}\sum_{\tau\in S_{\ell+1,m+n}}
\big(
f(g\otimes h\otimes \overbrace{1\otimes\dots\otimes 1}^{n-1})
\big)^{(\tau\cdot(\sigma\circ_{m+2}1_{\ell+1}))^{-1}}
} \\
\displaystyle{
\vphantom{\Big(}
=\!\!\!
\sum_{\sigma\in S_{m\!+\!1,n\!-\!1}}
\sum_{\tau\in S_{\ell\!+\!1,m\!+\!n}}
\!\!\!
\big(
f(g\!\otimes\! h\!\otimes\! \overbrace{1\!\otimes\! \dots\!\otimes\! 1}^{n-1})
\big)^{(\tau\cdot(
(
(1_2\circ_2\sigma)\cdot(1_n\circ_1(1,2)\circ_11_{m+1})
)
\circ_{m+2}1_{\ell+1}))^{-1}}
\,,}
\end{array}
\end{equation}
where
we used the equivariance relation \eqref{eq:operad9} for the first equality,
equation \eqref{20170608:eq7a} for the second equality,
the definition \eqref{eq:operad8} of the $\circ_i$-products for the third equality,
and Proposition \ref{prop:perm2}(d) for the fourth equality.
Equation \eqref{eq:operad9b}, with $(1_2\circ_2\sigma)$ in place of $\beta$,
$(1_n\circ_1(1,2)\circ_11_{m+1})$ in place of $\sigma$, $\alpha=\tau=1_{\ell+1}$ and $i=m+2$,
gives
$$
\begin{array}{l}
\displaystyle{
\vphantom{\Big(}
((1_2\circ_2\sigma)\cdot(1_n\circ_1(1,2)\circ_11_{m+1}))
\circ_{m+2}1_{\ell+1}
} \\
\displaystyle{
\vphantom{\Big(}
=
((1_2\circ_2\sigma)\circ_11_{\ell+1})
\cdot
((1_n\circ_1(1,2)\circ_11_{m+1})\circ_{m+2}1_{\ell+1})
} \\
\displaystyle{
\vphantom{\Big(}
=
(1_{\ell+2}\circ_{\ell+2}\sigma)
\cdot
(1,2)(1_{m+1},1_{\ell+1},1_{n-1})
\,,}
\end{array}
$$
where, for the second equality, we used equations \eqref{eq:perm14} and \eqref{michelino:eq3}.
Hence, the right-hand side of \eqref{20170608:eq6} becomes
\begin{equation}\label{20170608:eq6b}
\begin{array}{l}
\displaystyle{
\vphantom{\Big(}
\sum_{\sigma\in S_{m+1,n-1}}
\sum_{\tau\in S_{\ell+1,m+n}}
\big(
f(g\otimes h\otimes 1\otimes\dots\otimes 1)
\big)^{(\tau\cdot
(1_{\ell+2}\circ_{\ell+2}\sigma)
\cdot
(1,2)(1_{m+1},1_{\ell+1},1_{n-1})
)^{-1}}
} \\
\displaystyle{
\vphantom{\Big(}
=
\sum_{\sigma\in S_{\ell+1,m+1,n-1}}
\big(
f(g\otimes h\otimes 1\otimes \dots\otimes 1)
\big)^{(
\sigma\cdot
(1,2)(1_{m+1},1_{\ell+1},1_{n-1})
)^{-1}}
} \\
\displaystyle{
\vphantom{\Big(}
=
\sum_{\sigma\in S_{m+1,\ell+1,n-1}}
\big(
f(g\otimes h\otimes 1\otimes \dots \otimes 1)
\big)^{\sigma^{-1}}
\,,}
\end{array}
\end{equation}
where we used Proposition \ref{prop:perm1}(c) in the first equality
and Proposition \ref{prop:perm0}(b) for the second equality.
Combining \eqref{20170608:eq3}, \eqref{20170608:eq4}, \eqref{20170608:eq5},
\eqref{20170608:eq6} and \eqref{20170608:eq6b},
we get
\begin{equation}\label{20170608:eq7b}
(f\Box g)\Box h-f\Box (g\Box h)
=
\sum_{\sigma\in S_{m+1,\ell+1,n-1}}
\big(
f(g\otimes h\otimes 1\otimes \dots \otimes 1)
\big)^{\sigma^{-1}}
\,.
\end{equation}
To conclude the proof of (b),
we observe that the right-hand side of \eqref{20170608:eq7b}
is manifestly supersymmetric with respect to the exchange of $g$ and $h$.
Indeed, since $f\in W_n$, we have $f=f^{(1,2)}$.
Hence, by the equivariance condition \eqref{eq:operad4}, we have
$$
\begin{array}{l}
\displaystyle{
\vphantom{\Big(}
\sum_{\sigma\in S_{m+1,\ell+1,n-1}}
\big(
f(g\otimes h\otimes1\otimes\dots\otimes1)
\big)^{\sigma^{-1}}
} \\
\displaystyle{
\vphantom{\Big(}
=
\sum_{\sigma\in S_{m+1,\ell+1,n-1}}
\big(
f^{(1,2)}(g^{1_{m+1}}\otimes h^{1_{\ell+1}}\otimes 1\otimes\dots\otimes1)
\big)^{\sigma^{-1}}
} \\
\displaystyle{
\vphantom{\Big(}
=
(-1)^{p(g)p(h)}
\sum_{\sigma\in S_{m+1,\ell+1,n-1}}
\big(
f(h\otimes g\otimes 1\otimes\dots\otimes1)
\big)^{
(1,2)(1_{m+1},1_{\ell+1},1_{n-1})\cdot\sigma^{-1}}
} \\
\displaystyle{
\vphantom{\Big(}
=
(-1)^{p(g)p(h)}
\sum_{\sigma\in S_{m+1,\ell+1,n-1}}
\big(
f(h\otimes g\otimes 1\otimes \dots \otimes 1)
\big)^{
(\sigma\cdot(1,2)(1_{\ell+1},1_{m+1},1_{n-1}))^{-1}}
} \\
\displaystyle{
\vphantom{\Big(}
=
(-1)^{p(g)p(h)}
\sum_{\sigma\in S_{\ell+1,m+1,n-1}}
\big(
f(h\otimes g\otimes 1\otimes \dots\otimes 1)
\big)^{
\sigma^{-1}}
\,,}
\end{array}
$$
again by Proposition \ref{prop:perm0}(b).
Claim (c) is an obvious consequence of (b).
\end{proof}
\begin{remark}\label{20170603:thm1}
For an arbitrary non-symmetric operad $\mc P$
(i.e. for which we do not require 
the action of the symmetric groups and
the equivariance axiom \eqref{eq:operad4}),
we can also construct a $\mb Z$-graded Lie superalgebra
\begin{equation}\label{20170603:eq1}
\mc G=\bigoplus_{n\geq-1}\mc G_n
\,\,,\quad
\mc G_{n-1}=\mc P(n)
\,,
\end{equation}
with Lie bracket
($f\in\mc P(n)$, $g\in\mc P(m)$)
\begin{equation}\label{20170603:eq2}
[f,g]
=
\sum_{i=1}^{n} 
f\circ_i g
-(-1)^{p(f)p(g)}
\sum_{i=1}^m
g\circ_i f
\,.
\end{equation}
Indeed, letting $f\circ g= \sum_{i=1}^{n} f\circ_i g$, we have,
by the associativity condition \eqref{eq:operad25},
$$
\begin{array}{l}
\displaystyle{
\vphantom{\Big(}
(f\circ g)\circ h-f\circ(g\circ h)
=
\sum_{i=1}^{n}\sum_{j=1}^{m+n-1}
(f\circ_i g)\circ_j h
-
\sum_{i=1}^{n}\sum_{j=1}^{m}
f\circ_i (g\circ_j h)
} \\
\displaystyle{
\vphantom{\Big(}
=
(-1)^{p(g)p(h)}
\sum_{1\leq j<i\leq n}
(f\circ_j h)\circ_{\ell+i-1} g
+
\sum_{1\leq i<j\leq n}
(f\circ_i g)\circ_{m+j-1} h
\,.}
\end{array}
$$
Since the expression in the right-hand side is supersymmetric with respect to the exchange of $g$ and $h$,
it follows that \eqref{20170603:eq2} is a Lie superalgebra bracket.
In the special case when each $\mc P(n)$ has the same parity as $n+1$,
the resulting bracket \eqref{20170603:eq2}
is known as the Gerstenhaber bracket \cite{Ger63}.
\end{remark}
\section{The operad governing Lie superalgebras}\label{sec:4}


Given the vector superspace $V$, with parity $p$,
we denote by $\Pi V$ the same vector space with reversed parity $\bar p=1-p$,
and we consider the corresponding operad $\mc Hom(\Pi V)$
from Example \ref{ex:operadH},
and the associated $\mb Z$-graded Lie superalgebra $W(\Pi V):=W(\mc Hom(\Pi V))$
given by Theorem \ref{20170603:thm2}.
\begin{proposition}[\cite{NR67,DSK13}]\label{20170612:prop1}
We have a bijective correspondence between 
the odd elements\/ $X\in W_1(\Pi V)$ such that\/ $X\Box X=0$
and the Lie superalgebra brackets\/
$[\cdot\,,\,\cdot]\colon V\times V\to V$ on\/ $V$,
given by 
\begin{equation}\label{20170612:eq3}
[a,b]=(-1)^{p(a)}X(a\otimes b)
\,.
\end{equation}
\end{proposition}
\begin{proof}
By definition, $X\in (\mc Hom(\Pi V))(2)_{\bar 1}$ is
an odd linear map $X\colon (\Pi V)^{\otimes 2}\to\Pi V$,
and it corresponds, via \eqref{20170612:eq3},
to a parity preserving bilinear map $[\cdot\,,\,\cdot]:\,V\times V\to V$.
Moreover, to say that $X$ lies in $W_1(\Pi V)=(\mc Hom(\Pi V))(2)^{S_2}$
is equivalent to say that the corresponding bracket $[\cdot\,,\,\cdot]$
satisfies skewsymmetry. 
Finally, by the definition \eqref{eq:box} of the $\Box$-product,
we have ($a,b,c\in V$)
$$
\begin{array}{l}
\displaystyle{
\vphantom{\Big(}
(X\Box X)(a\otimes b\otimes c)
=
\sum_{\sigma\in S_{2,1}}(X\circ_1 X)^{\sigma^{-1}}(a\otimes b\otimes c)
} \\
\displaystyle{
\vphantom{\Big(}
=
(-1)^{p(b)}\Big(
[[a,b],c]-[a,[b,c]]+(-1)^{p(a)p(b)}[b,[a,c]]
\Big)
\,.}
\end{array}
$$
Hence, $X\Box X=0$ if and only if the Jacobi identity 
holds.
\end{proof}
Note that, if $X\in W_1(\Pi V)_{\bar 1}$ satisfies $X\Box X=0$,
then it follows by the Jacobi identity for the Lie superalgebra $W(\Pi V)$
that $(\ad X)^2=0$,
i.e. we have a cohomology complex $(W(\Pi V),\ad X)$.
\begin{definition}
Let $V$ be a Lie superalgebra.
The corresponding \emph{Lie superalgebra cohomology complex} is defined as
$$
(W(\Pi V),\ad X)
\,,
$$
where $X\in W(\Pi V)_{\bar 1}$ is given by \eqref{20170612:eq3}.
\end{definition}
Obviously, the kernel of $\ad X$ is a subalgebra of $W(\Pi V)$ and the image of $\ad X$ is its ideal.
Hence, the cohomology $H(W(\Pi V),\ad X)$ has the structure of a Lie superalgebra.
\begin{remark}\label{rem:lie-operad}
One can define the $\mc Lie$ operad as follows:
$\mc Lie(1)=\mb F1$;
$\mc Lie(2)$ is the non-trivial $1$-dimensional representation of $S_2$,
with basis element denoted by $[\cdot\,,\,\cdot]$;
for every $n\geq2$,  all the elements of $\mc Lie(n)$ are 
obtained by composition of $[\cdot\,,\,\cdot]\in\mc Lie(2)$,
and they are subject to the relation in $\mc Lie(3)$ corresponding to the Jacobi identity:
$$
[\cdot\,,\,\cdot]\circ_2[\cdot\,,\,\cdot]-\sigma_{12}([\cdot\,,\,\cdot]\circ_2[\cdot\,,\,\cdot])
=
[\cdot\,,\,\cdot]\circ_1[\cdot\,,\,\cdot]
\,.
$$
Then, a Lie superalgebra structure on a vector superspace $V$
is the same as a morphism of (symmetric) operads
$\mc Lie\to\mc Hom(V)$.
Proposition \ref{20170612:prop1} gives such a morphism by sending $[\cdot\,,\,\cdot]$ to $X$.
\end{remark}

In the following sections, we will repeat the same line of reasoning 
as the one used in the present section
for the cohomology theories of Lie conformal algebras, 
Poisson algebras, Poisson vertex algebras and vertex algebras:
after reviewing their definition,
we will construct, for each of them, an operad $\mc P$,
and we will describe their algebraic structures as an element $X\in W_1\subset W(\mc P)$
such that $X\Box X=0$.
In this way, we automatically get,
for each algebraic structure of interest,
the corresponding cohomology complex $(W(\mc P),\ad X)$.

\section{The operad governing Lie conformal superalgebras}\label{sec:LCA}

\subsection{Lie conformal superalgebras}\label{sec:LCA.0}

Recall that a Lie conformal superalgebra is a vector superspace $V$,
endowed with an even endomorphism $\partial\in\End(V)$
and a bilinear (over $\mb F$) $\lambda$-bracket
$[\cdot\,_\lambda\,\cdot]\colon V\times V\to V[\lambda]$
satisfying 
sesquilinearity ($a,b\in V$):
\begin{equation}\label{20170612:eq6}
[\partial a_\lambda b]=-\lambda[a_\lambda b]
\,,\,\,
[a_\lambda \partial b]=(\lambda+\partial)[a_\lambda b]
\,,
\end{equation}
skewsymmetry ($a,b\in V$):
\begin{equation}\label{20170612:eq4}
[a_\lambda b]=-(-1)^{p(a)p(b)}[b_{-\lambda-\partial}a]
\,,
\end{equation}
and the Jacobi identity ($a,b,c\in V$):
\begin{equation}\label{20170612:eq5}
[a_{\lambda}[b_\mu c]]-(-1)^{p(a)p(b)}[b_\mu [a_\lambda ,b]]
=[[a_\lambda b]_{\lambda+\mu}c]
\,.
\end{equation}

\subsection{The $\mc Chom$ operad}\label{sec:LCA.1}

Let $V=V_{\bar 0}\oplus V_{\bar 1}$ be a vector superspace endowed
with an even endomorphism $\partial\in\End V$.
The operad $\mc{C}hom$ is defined as 
the collection of superspaces
\begin{equation}\label{eq:chom1}
\mc{C}hom(n)=\Hom_{\mb F[\partial]^{\otimes n}}(V^{\otimes n}
,\mb F_-[\lambda_1,\dots,\lambda_n]\otimes_{\mb F[\partial]}V)
\,\,,\,\,\,\,
n\geq0
\,.
\end{equation}
Here and further 
$\lambda_1,\dots,\lambda_k$ are commuting indeterminates of even parity and
$\mb F_-[\lambda_1,\dots,\lambda_k]$
denotes the space of polynomials in the variables $\lambda_1,\dots,\lambda_n$.
This space is endowed with a structure of a left $\mb F[\partial]^{\otimes n}$-module
by letting $P_1(\partial)\otimes\,\cdots\,\otimes P_n(\partial)$ act as
multiplication by $P_1(-\lambda_1)\cdots P_n(-\lambda_n)$,
and a structure of a right $\mb F[\partial]$-module
by letting $\partial$ act as multiplication by $-\lambda_1-\dots-\lambda_n$.

Note that $\mc Chom(0)=V/\dd V$.
Obviously, we can identify $\mb F_-[\lambda]\otimes_{\mb F[\partial]}V\simeq V$,
so that $\mc Chom(1)=\End_{\mb F[\partial]}(V)$.
For arbitrary $n\geq1$, $\mc Chom(n)$ consists of all linear maps
$$
f\colon V^{\otimes n}\,\longrightarrow\,\mb F_-[\lambda_1,\dots,\lambda_n]\otimes_{\mb F[\partial]}V
\,\,,\,\,\,\,
v_1\otimes\,\cdots\,\otimes v_n\mapsto f_{\lambda_1,\dots,\lambda_n}(v_1\otimes\,\cdots\,\otimes v_n)
\,,
$$
satisfying the sesquilinearity conditions:
\begin{equation}\label{20170613:eq1}
f_{\lambda_1,\dots,\lambda_n}(v_1\otimes\,\cdots\partial v_i\cdots\,\otimes v_n)
=-\lambda_if_{\lambda_1,\dots,\lambda_n}(v_1\otimes\,\cdots\,\otimes v_n)
\,\,\text{ for all } i=1,\dots,n
\,.
\end{equation}
In particular, $\mc Chom(2)$ is identified with the space of all $\la$-brackets on $V$,
satisfying \eqref{20170612:eq6}.

The $\mb Z/2\mb Z$-structure of $\mc Chom(n)$ is induced by that of $V$.
The composition of $f\in\mc Chom(n)$
and $g_1\in\mc Chom(m_1),\dots,g_n\in\mc Chom(m_n)$ is defined as follows:
\begin{equation}\label{20170613:eq2}
\begin{array}{l}
\displaystyle{
\vphantom{\Big(}
\big(f(g_1\otimes\dots\otimes g_n)\big)_{\lambda_1,\dots,\lambda_{M_n}}
(v_1\otimes\dots\otimes v_{M_n})
} \\
\displaystyle{
\vphantom{\Big(}
:=
f_{\Lambda_1,\dots,\Lambda_n}
\big(
((g_1)_{\lambda_{1},\dots,\lambda_{M_1}}
\otimes\dots\otimes
(g_n)_{\lambda_{M_{n-1}+1},\dots,\lambda_{M_n}})
(v_1\otimes\dots\otimes v_{M_n})
\big)
\,,}
\end{array}
\end{equation}
where we let (cf. \eqref{20170821:eq2a})
\begin{equation}\label{20170821:eq2}
M_i=\sum_{j=1}^im_j
\,,\,\,
i=0,\dots,n
\,\,,\,\,\text{ and }\,\,
\Lambda_i
=
\sum_{j=M_{i-1}+1}^{M_i}
\lambda_j
\,,\,\,
i=1,\dots,n
\,,
\end{equation}
and, recalling \eqref{20170804:eq1}, we have
\begin{equation}\label{20170821:eq5}
\begin{array}{l}
\displaystyle{
\vphantom{\Big(}
((g_1)_{\lambda_{1},\dots,\lambda_{M_1}}
\otimes\dots\otimes
(g_n)_{\lambda_{M_{n-1}+1},\dots,\lambda_{M_n}})
(v_1\otimes\dots\otimes v_{M_n})
} \\
\displaystyle{
\vphantom{\Big(}
=\pm\,
(g_1)_{\lambda_{1},\dots,\lambda_{M_1}}\!\!(v_1\otimes\dots\otimes v_{M_1})
\otimes\dots\otimes
(g_n)_{\lambda_{M_{\!n\!-\!1}\!+1},\dots,\lambda_{M_n}}\!\!(v_{M_{\!n\!-\!1}\!+1}\otimes\dots\otimes v_{M_n})
\,,}
\end{array}
\end{equation}
where
\begin{equation}\label{20170821:eq5b}
\pm=(-1)^{\sum_{i<j}p(g_j)(p(v_{M_{i-1}+1})+\dots+p(v_{M_i}))}
\,.
\end{equation}
The unity in the $\mc Chom$-operad is $1=\id_V\in\mc Chom(1)=\End_{\mb F[\partial]} V$,
and the right action of $S_n$ on $\mc Chom(n)$
is given by (cf. \eqref{eq:operad14}, \eqref{eq:operad7} and \eqref{20170615:eq2}):
\begin{equation}\label{20170613:eq3}
\begin{array}{l}
\displaystyle{
\vphantom{\Big(}
(f^\sigma)_{\lambda_1,\dots,\lambda_n}(v_1\otimes\dots\otimes v_n)
=
f_{\sigma(\lambda_1,\dots,\lambda_n)}
(\sigma(v_1\otimes\dots\otimes v_n))
} \\
\displaystyle{
\vphantom{\Big(}
\,\,\,\,\,\,\,\,\,\,\,\,\,\,\,\,\,\,
=
\epsilon_v(\sigma) 
f_{\lambda_{\sigma^{-1}(1)},\dots,\lambda_{\sigma^{-1}(n)}}
(v_{\sigma^{-1}(1)}\otimes\dots\otimes v_{\sigma^{-1}(n)})
\,,}
\end{array}
\end{equation}
for every $\sigma\in S_n$,
where $\epsilon_v(\sigma)$ is given by \eqref{eq:operad14}.

Let us first check the associativity axiom for the operad $\mc Chom$,
which reads
\begin{equation}\label{20170822:eq1}
\begin{array}{l}
\displaystyle{
\vphantom{\Big(}
\big(
(f(g_1\otimes\dots\otimes g_n))(h_1\otimes\dots\otimes h_{M_n})
\big)_{\lambda_1,\dots,\lambda_{L_{M_n}}}
(v_1\otimes\dots\otimes v_{L_{M_n}})
} \\
\displaystyle{
\vphantom{\Big(}
=
\big(
f((g_1\otimes\dots\otimes g_n)(h_1\otimes\dots\otimes h_{M_n}))
\big)_{\lambda_1,\dots,\lambda_{L_{M_n}}}
(v_1\otimes\dots\otimes v_{L_{M_n}})
\,,}
\end{array}
\end{equation}
for every $f\in\mc Chom(n)$,
$g_i\in\mc Chom(m_i)$, $i=1,\dots,n$,
$h_j\in\mc Chom(\ell_j)$, $j=1,\dots,m_1+\dots+m_n=:M_n$,
and $v_k\in V$, $k=1,\dots,\ell_1+\dots+\ell_{M_n}=:L_{M_n}$.
Let us denote, in accordance to \eqref{20170821:eq2},
\begin{equation}\label{20170822:eq2}
M_i=\sum_{j=1}^i m_j
\,,\,\,i=0,\dots,n
\,,\,\,\text{ and }\,\,
L_j=\sum_{k=1}^j \ell_k
\,.
\end{equation}
Then, using the definition \eqref{20170613:eq2} of the composition map in the operad $\mc Chom$,
one easily checks that both sides of \eqref{20170822:eq1} are equal to
$$
\begin{array}{l}
\displaystyle{
\vphantom{\Big(}
f_{\sum_{i=1}^{L_{M_1}}\!\!\!\lambda_i,\dots,\sum_{i=L_{M_{\!n\!-\!1}}\!\!\!+1}^{L_{M_n}}\!\lambda_i}
\Big(
\big(
(g_1)_{\sum_{j=1}^{L_1}\!\!\lambda_j,\dots,\sum_{j=L_{\!M_1\!-\!1}\!+1}^{L_{M_1}}\!\!\lambda_j}
\!\!\!\otimes\cdots
} \\
\displaystyle{
\vphantom{\Big(}
\otimes\!
(g_n)_{\sum_{j=L_{M_{\!n\!-\!1}}\!\!+1}^{L_{M_{\!n\!-\!1}\!\!+2}}\!\!\lambda_j,\dots,\sum_{j=L_{M_n\!-\!1}\!+1}^{L_{M_n}}\!\!\lambda_j}
\big)
\Big(
\big(
(h_1)_{\lambda_1,\dots,\lambda_{L_1}}
\otimes\cdots
} \\
\displaystyle{
\vphantom{\Big(}
\otimes
(h_{M_n})_{\lambda_{L_{M_n-1}+1},\dots,\lambda_{L_{M_n}}}
\big)
(v_1\otimes\dots\otimes v_{L_{M_n}})
\Big)
\Big)
\,,}
\end{array}
$$
proving associativity.
The unity axiom is immediate to check.
Next, we prove the equivariance axiom \eqref{eq:operad4}.
If we apply the left-hand side of \eqref{eq:operad4}
to a monomial $v_1\otimes\dots\otimes v_{M_n}$
and we evaluate it on the variables $\lambda_1,\dots,\lambda_{M_n}$,
we get
$$
\begin{array}{l}
\displaystyle{
\vphantom{\Big(}
\big(f^\sigma(g_1^{\tau_1}\otimes \dots\otimes g_n^{\tau_n})\big)_{\lambda_1,\dots,\lambda_{M_n}}
(v_1\otimes\dots\otimes v_{M_n})
} \\
\displaystyle{
\vphantom{\Big(}
=
(f^\sigma)_{\Lambda_1,\dots,\Lambda_n}
\Big(
\big(
(g_1^{\tau_1})_{\lambda_1,\dots,\lambda_{M_1}}
\otimes \dots\otimes 
(g_n^{\tau_n})_{\lambda_{M_{n-1}+1},\dots,\lambda_{M_n}}
\big)
(v_1\otimes\dots\otimes v_{M_n})
\Big)
} \\
\displaystyle{
\vphantom{\Big(}
=
f_{\sigma(\Lambda_1,\dots,\Lambda_n)}
\Big(
\sigma\Big(
\big(
(g_1)_{\tau_1(\lambda_1,\dots,\lambda_{M_1})}
\otimes \dots\otimes 
(g_n)_{\tau_n(\lambda_{M_{n-1}+1},\dots,\lambda_{M_n})}
\big)
\big(
} \\
\displaystyle{
\vphantom{\Big(}
\,\,\,\,\,\,\,\,\,\,\,\,\,\,\,\,\,\,\,\,\,\,\,\,\,\,\,
\,\,\,\,\,\,\,\,\,\,\,\,\,\,\,\,\,\,\,\,\,\,\,\,\,\,\,
\tau_1(v_1\otimes\dots\otimes v_{M_1})
\otimes\dots\otimes
\tau_n(v_{M_{n-1}+1}\otimes\dots\otimes v_{M_n})
\big)
\Big)\Big)
} \\
\displaystyle{
\vphantom{\Big(}
=
f_{\sigma(\Lambda_1,\dots,\Lambda_n)}
\Big(
\big(
\sigma
\big(
(g_1)_{\tau_1(\lambda_1,\dots,\lambda_{M_1})}
\otimes \dots\otimes 
(g_n)_{\tau_n(\lambda_{M_{n-1}+1},\dots,\lambda_{M_n})}
\big)
\big)
\big(
} \\
\displaystyle{
\vphantom{\Big(}
\,\,\,\,\,\,\,\,\,\,\,\,\,\,\,\,\,\,\,\,\,\,\,\,\,\,\,
\,\,\,\,\,\,\,\,\,\,\,\,\,\,\,\,\,\,\,\,\,\,\,\,\,\,\,
\sigma(\tau_1,\dots,\tau_n)
(v_1\otimes\dots\otimes v_{M_n})
\big)
\Big)
} \\
\displaystyle{
\vphantom{\Big(}
=
\big(f(\sigma(g_1\otimes\dots\otimes g_n))\big)_{(\sigma(\tau_1,\dots,\tau_n))(\lambda_1,\dots,\lambda_{M_n})}
\big(
\sigma(\tau_1,\dots,\tau_n)
(v_1\otimes\dots\otimes v_{M_n})
\big)
} \\
\displaystyle{
\vphantom{\Big(}
=
\big(
\big(f(\sigma(g_1\otimes\dots\otimes g_n))\big)^{\sigma(\tau_1,\dots,\tau_n)}
\big)_{\lambda_1,\dots,\lambda_{M_n}}
(v_1\otimes\dots\otimes v_{M_n})
\,.}
\end{array}
$$
For the first equality we used the definition \eqref{20170613:eq2} of the composition maps in $\mc Chom$,
for the second equality we used the definition \eqref{20170613:eq3}
of the action of the symmetric group on $\mc Chom$,
for the third equality we used Lemma \ref{20170821:lem}
and the definition \eqref{eq:operad19} of the composition map of permutations,
for the fourth equality we used again \eqref{eq:operad19} and \eqref{20170613:eq2},
and for the last equality we used again \eqref{20170613:eq3}.
This proves the equivariance condition \eqref{eq:operad4}.

\subsection{Lie conformal superalgebras and the operad $\mc Chom$}\label{sec:LCA.2}

Given the vector superspace $V$, with parity $p$,
and the even endomorphism $\partial\in\End(V)$,
we denote by $\Pi V$ the same vector space with reversed parity $\bar p=1-p$.
Obviously, $\partial$ is also an even endomorphism of $\Pi V$.
We consider the corresponding operad $\mc Chom(\Pi V)$
from Section \ref{sec:LCA.1}
and the associated $\mb Z$-graded Lie superalgebra 
$W^\partial(\Pi V):=W(\mc Chom(\Pi V))$
given by Theorem \ref{20170603:thm2}.
\begin{proposition}[\cite{DSK13}]\label{20170612:prop2}
We have a bijective correspondence between 
the odd elements\/ $X\in W^\partial_1(\Pi V)$ such that\/ $X\Box X=0$
and the Lie conformal superalgebra\/ $\lambda$-brackets\/
$[\cdot\,_\lambda\,\cdot]\colon V\times V\to V[\lambda]$ on\/ $V$,
given by 
\begin{equation}\label{20170612:eq3a}
[a_\lambda b]=(-1)^{p(a)}X_{\lambda,-\lambda-\partial}(a\otimes b)
\,.
\end{equation}
\end{proposition}
\begin{proof}
First, $X\in (\mc Chom(\Pi V))(2)$ is, by definition,
an odd $\mb F[\partial]^{\otimes2}$-module homomorphism
$X_{\lambda,\mu}\colon (\Pi V)^{\otimes 2}\to\mb F_-[\lambda,\mu]\otimes_{\mb F[\partial]}\Pi V\simeq V[\lambda]$
(the last isomorphism being obtained by letting $\mu=-\lambda-\partial$),
and it corresponds, via \eqref{20170612:eq3},
to $\lambda$-bracket $[\cdot\,,\,\cdot]\colon V\times V\to V[\lambda]$
satisfying the sesquilinearity conditions \eqref{20170612:eq6}.
The condition that $X\in (\mc Chom(\Pi V))(2)$ is odd
(with respect to the parity $\bar p$ induced by that $\Pi V$),
translates into saying that the corresponding $\lambda$-bracket $[\cdot\,,\,\cdot]$
is parity preserving.
Moreover, the condition that $X$ is fixed by the action \eqref{20170613:eq3} of the symmetric group $S_2$
translates into saying that 
the corresponding $\lambda$-bracket $[\cdot\,_\lambda\,\cdot]$
satisfies the skewsymmetry axiom \eqref{20170612:eq4}.
To complete the proof, we need to check that the equation $X\Box X=0$
translates to the Jacobi identity for the $\lambda$-bracket $[\cdot\,_\lambda\,\cdot]$.
By equation \eqref{eq:box1},
we have
$$
\begin{array}{l}
\displaystyle{
\vphantom{\Big(}
(X\Box X)_{\lambda,\mu,\nu}(a\otimes b\otimes c)
=
\sum_{\sigma\in S_{2,1}}\big((X\circ_1 X)^{\sigma^{-1}}\big)_{\lambda,\mu,\nu}(a\otimes b\otimes c)
} \\
\displaystyle{
\vphantom{\Big(}
=
X_{\lambda+\mu,\nu}(X_{\lambda,\mu}(a\otimes b)\otimes c)
+
(-1)^{\bar p(b)\bar p(c)}
X_{\lambda+\nu,\mu}(X_{\lambda,\nu}(a\otimes c)\otimes b)
} \\
\displaystyle{
\vphantom{\Big(}
\,\,\,\,\,\,\,\,\,\,\,\,\,\,\,\,\,\,
+
(-1)^{\bar p(a)(\bar p(b)+\bar p(c))}
X_{\mu+\nu,\lambda}(X_{\mu,\nu}(b\otimes c)\otimes a)
} \\
\displaystyle{
\vphantom{\Big(}
=
X_{\lambda+\mu,\nu}(X_{\lambda,\mu}(a\otimes b)\otimes c)
+
(-1)^{\bar p(b)(1+\bar p(a))}
X_{\mu,\lambda+\nu}(b\otimes X_{\lambda,\nu}(a\otimes c))
} \\
\displaystyle{
\vphantom{\Big(}
\,\,\,\,\,\,\,\,\,\,\,\,\,\,\,\,\,\,
+
(-1)^{\bar p(a)}
X_{\lambda,\mu+\nu}(a\otimes X_{\mu,\nu}(b\otimes c))
} \\
\displaystyle{
\vphantom{\Big(}
=
(-1)^{p(b)}\Big(
[[a_\lambda b]_{\lambda+\mu} c]-[a_\lambda [b_\mu c]]+(-1)^{p(a)p(b)}[b_\mu [a_\lambda c]]
\Big)
\,.}
\end{array}
$$
Hence, $X\Box X=0$ if and only if the Jacobi identity \eqref{20170612:eq5} holds.
\end{proof}

%
\begin{definition}[\cite{BKV99,DSK09}]
Let $V$ be a Lie conformal superalgebra.
The corresponding \emph{Lie conformal superalgebra cohomology complex} is defined as
$$
(W^\partial(\Pi V),\ad X)
\,,
$$
where $X\in W^\dd(\Pi V)_{\bar 1}$ is given by \eqref{20170612:eq3a}.
\end{definition}

\begin{remark}
One can introduce a conformal version of the operad $\mc Chom$, which is associated to the \emph{basic} Lie conformal algebra complex (see \cite{DSK13}). This leads to the notion of a \emph{conformal operad}, which will be developed in a forthcoming publication. In the geometric context of chiral algebras, the corresponding object was constructed by Tamarkin in \cite{Tam02}.
\end{remark}

\section{The chiral operad}\label{sec:chop}

\subsection{Vertex algebras}\label{sec:vert}

In this subsection, we recall the ``fifth definition'' of a vertex algebra, given in \cite{DSK06}.
In a nutshell, this definition says that a vertex algebra is a Lie conformal algebra
in which the $\lambda$-bracket can be ``integrated''.
More precisely we have
\begin{definition}\label{def5}
A vertex algebra is a  $\mb Z /2\mb Z$-graded $\mb F[\partial]$-module $V$,
endowed with an even element $\vac\in V_{\bar{0}}$ 
and an {\itshape integral} $\lambda$-{\itshape bracket},
namely a linear map $V \otimes V \to \mb F [\lambda] \otimes V$, 
denoted by 
$$
\int^\lambda d\sigma[u_\sigma v]
=
{:}uv{:}+\int_0^\lambda\!d\sigma\,[u_\sigma v]
\,,
$$
such that the following axioms hold:
\begin{enumerate}[(i)]
\item
$\int^\lambda d\sigma[\vac_\sigma v]=\int^\lambda d\sigma[v_\sigma \vac]=v$,
\item
$\int^\lambda\!\! d\sigma\, [\partial u_\sigma v] = -\int^\lambda\!\! d\sigma\,\sigma[u_\sigma v]$,
$\int^\lambda\!\! d\sigma\,[u_\sigma \partial v] = \int^\lambda\!\!d\sigma\, (\partial+\sigma)[u_\sigma v]$,
\item 
$\int^\lambda\!\! d\sigma\,[v_\sigma u] =(-1)^{p(u)p(v)}\int^{-\lambda-\partial}\!\! d\sigma\,[u_\sigma v]$,
\item
$\int^\lambda\!\! d\sigma\!\! \int^\mu\!\! d\tau \Big(
[u_\sigma[v_\tau w]] - (-1)^{p(u)p(v)} [v_\tau[u_\sigma w]] - [[u_\sigma v]_{\sigma+\tau}w] 
\Big) = 0$.
\end{enumerate}
If we do not assume the existence of the unit element $\vac\in V$ and we drop axiom (i),
we call $V$ a \emph{non-unital vertex algebra}.
\end{definition}
Paper \cite{DSK06} contains a detailed discussion of this definition
and the proof of its equivalence to other definitions of vertex algebras.
We shall call $[u_\lambda v]$,
defined as the derivative by $\lambda$ of the polynomial $\int^\lambda\!d\sigma\,[u_\sigma v]$,
the $\lambda$-\emph{bracket} of $u$ and $v$. 
Their \emph{normally ordered product}
${:}uv{:}$ 
is defined as the constant term of the polynomial $\int^\lambda\!d\sigma\,[u_\sigma v]$.
The polynomial $\int^\lambda\!d\sigma\,[u_\sigma v]$ will be called
the \emph{integral of the} $\lambda$-\emph{bracket} of $u$ and $v$.

Axioms (i)-(iv) are a concise way to write more complicated relations
involving the normally ordered products ${:}uv{:}$ and the $\lambda$-brackets $[u_\lambda v]$.
To explain this, let us describe the meaning of axiom (ii).
Taking the derivative with respect to $\lambda$ of both equations we get the
sesquilinearity conditions of the $\lambda$-bracket:
$[\partial u_\lambda v]=-\lambda[u_\lambda v]$ and 
$[u_\lambda \partial v]=(\partial+\lambda)[u_\lambda v]$,
while putting $\lambda=0$ in axiom (ii) we get that 
$\partial$ is a derivation of the normally ordered product,
plus a new piece of notation:
\begin{equation}\label{20160719:eq4}
\int^0\!d\sigma\,\sigma[u_\sigma v]=-{:}(\partial u)v{:}
\,.
\end{equation}
Similarly, 
axiom (iii) gives two conditions:
taking the derivative of both sides with respect to $\lambda$ we get 
the skewsymmetry of the $\lambda$-bracket: 
$$[v_\lambda u]=-(-1)^{p(u)p(v)}[u_{-\lambda-\partial}v],$$
while taking the constant term in $\lambda$ we get the quasi-commutativity of the normally ordered product:
$$
{:}uv{:}-(-1)^{p(u)p(v)}{:}vu{:}\,=\,\int_{-\partial}^0\!d\lambda\,[u_\lambda v]
\,.
$$
Finally, to explain of axiom (iv) we expand
all three summand in terms of normally ordered product and $\lambda$-brackets.
The first term is immediate to understand:
$$
\begin{array}{l}
\displaystyle{
\int^\lambda\!\! d\sigma\!\! \int^\mu\!\! d\tau [u_\sigma[v_\tau w]]
=
{:}u({:}vw{:}){:}
+\int_0^\mu\!\! d\tau {:}u[v_\tau w]{:}
} \\
\displaystyle{
+\int_0^\lambda\!\! d\sigma [u_\sigma {:}vw{:}]
+\int_0^\lambda\!\! d\sigma\!\! \int_0^\mu\!\! d\tau [u_\sigma[v_\tau w]]
\,.}
\end{array}
$$
Similarly for the second term.
To correctly expand the third term,
we first perform the change of variable $\sigma+\tau\mapsto\tau$,
we exchange the order of integration in $d\sigma$ and $d\tau$,
and we use the notation \eqref{20160719:eq4}.
As a result, we get
$$
\begin{array}{l}
\displaystyle{
\int^\lambda\!\! d\sigma\!\! \int^\mu\!\! d\tau [[u_\sigma v]_{\sigma+\tau}w] 
=
\int^{\lambda+\mu}\!\! d\tau
\Big[\Big(\int_{-\mu-\partial}^\lambda\!\! d\sigma [u_\sigma v]\Big)_{\tau}w\Big] 
} \\
\displaystyle{
={:}\Big(\int_{-\mu-\partial}^\lambda\!\! d\sigma [u_\sigma v]\Big)w{:}
+\int_0^{\lambda+\mu}\!\! d\tau
\Big[\big(\int_{-\mu-\partial}^\lambda\!\! d\sigma [u_\sigma v]\Big)_{\tau}w\Big] 
\,.}
\end{array}
$$
Hence, by taking constant term or derivative with respect to $\lambda$ and/or $\mu$,
axiom (iv) produces four different axioms on the normally ordered product and the $\lambda$-bracket.

Axiom (iv), i.e. Jacobi identity under integration, could be written in the seemingly equivalent form:
\begin{enumerate}
\item[(iv')]
$\displaystyle\int^\lambda\!\! d\sigma\!\! \int^\mu\!\! d\tau \Big(
[u_\sigma[v_{\tau-\sigma} w]] - (-1)^{p(u)p(v)} [v_{\tau-\sigma}[u_\sigma w]] - [[u_\sigma v]_{\tau}w] 
\Big) = 0$
\,.
\end{enumerate}
Not surprisingly, (iv) and (iv') are equivalent, as a consequence of the following:
\begin{lemma}\label{20160721:lem}
Let\/ $\int^\lambda\!d\sigma\,[\cdot\,_\sigma\,\cdot]\colon V\otimes V\to V[\lambda]$
be a linear map satisfying the sesquilinearity and skewsymmetry conditions under integration,
i.e. axioms (ii) and (iii) of Definition \ref{def5}.
Let
\begin{equation}\label{20160721:eq2}
J_{\lambda,\mu}(u,v,w)
=
\int^\lambda\!\! d\sigma\!\! \int^\mu\!\! d\tau \Big(
[u_\sigma[v_{\tau} w]] - (-1)^{p(u)p(v)} [v_{\tau}[u_\sigma w]] - [[u_\sigma v]_{\sigma+\tau}w] 
\Big)
\,,
\end{equation}
and let 
\begin{equation}\label{20160721:eq3}
\widetilde{J}_{\lambda,\mu}(u,v,w)
=
\int^\lambda\!\! d\sigma\!\! \int^\mu\!\! d\tau \Big(
[u_\sigma[v_{\tau-\sigma} w]] - (-1)^{p(u)p(v)} [v_{\tau-\sigma}[u_\sigma w]] - [[u_\sigma v]_{\tau}w] 
\Big)
\,.
\end{equation}
Then, we have
\begin{equation}\label{20160721:eq1}
\widetilde{J}_{\lambda,\mu}(u,v,w)
=
(-1)^{p(v)p(w)}
J_{\lambda,-\mu-\partial}(u,w,v)
\,.
\end{equation}
\end{lemma}
\begin{proof}
Applying the skewsymmetry condition under the sign of integral, we have
$$
\begin{array}{l}
\displaystyle{
\vphantom{\Big(}
\widetilde{J}_{\lambda,\mu}(u,v,w)
=
\int^\lambda\!\! d\sigma\!\! \int^\mu\!\! d\tau \Big(
[u_\sigma[v_{\tau-\sigma} w]] - (-1)^{p(u)p(v)} [v_{\tau-\sigma}[u_\sigma w]] - [[u_\sigma v]_{\tau}w] 
\Big)
} \\
\displaystyle{
\vphantom{\Big(}
=
\int^\lambda\!\!\! d\sigma\!\!\! \int^\mu\!\!\! d\tau \Big(\!
-(-1)^{p(v)p(w)} [u_\sigma[w_{-\tau+\sigma-\partial} v]] 
+(-1)^{p(v)p(w)} [[u_\sigma w]_{-\tau+\sigma-\partial}v] 
} \\
\displaystyle{
\vphantom{\Big(}
+(-1)^{(p(u)+p(v))p(w)} [w_{-\tau-\partial}[u_\sigma v]] 
\Big)
\,.}
\end{array}
$$
Note that, by the sesquilinearity condition, $\sigma$ in the first term of the right-hand side 
is the same as $-\partial$ acting on
the first factor $u$.
If we then perform the change of variable $-\tau-\partial=\rho$, the right-hand side becomes
$$
\begin{array}{l}
\displaystyle{
\vphantom{\Big(}
(-1)^{p(v)p(w)}
\int^\lambda\!\!\! d\sigma\!\! \int^{-\mu-\partial}\!\!\!\!\! d\rho \Big(
[u_\sigma[w_{\rho} v]] 
- [[u_\sigma w]_{\rho+\sigma}v] 
- (-1)^{p(u)p(w)} [w_{\rho}[u_\sigma v]] 
\Big)
} \\
\displaystyle{
\vphantom{\Big(}
=
(-1)^{p(v)p(w)}
J_{\lambda,-\mu-\partial}(u,w,v)
\,.}
\end{array}
$$
This completes the proof.
\end{proof}
\begin{definition}\label{def-mod}
A left \emph{module} $M$ over a non-unital vertex algebra $V$ is a $\mb Z /2\mb Z$-graded $\mb F[\partial]$-module
endowed with an {\itshape integral} $\lambda$-{\itshape action},
$V \otimes M \to \mb F [\lambda] \otimes M$, 
denoted $v\otimes m\mapsto \int^\lambda\!d\sigma\,(v_\sigma m)$,
preserving the $\mb Z/2\mb Z$-grading,
such that the following axioms hold:
\begin{enumerate}[(i)]
\item
$\int^\lambda\!\! d\sigma\, (\partial v_\sigma m) = -\int^\lambda\!\! d\sigma\,\sigma(v_\sigma m)$,
$\int^\lambda\!\! d\sigma\,(v_\sigma \partial m) = \int^\lambda\!\!d\sigma\, (\partial+\sigma)(v_\sigma m)$,
\item
$\int^\lambda\!\! d\sigma\!\! \int^\mu\!\! d\tau \Big(
u_\sigma(v_\tau m) - (-1)^{p(u)p(v)} v_\tau(u_\sigma m) - [u_\sigma v]_{\sigma+\tau}m 
\Big) = 0$.
\end{enumerate}
This is equivalent to say that the $\mb F[\partial]$-module $V\oplus M$
has a vertex algebra structure $\int^\lambda\!d\sigma\,[\cdot\,_\sigma\,\cdot]^\sim$
such that 
$\int^\lambda\!d\sigma\,[u_\sigma v]^\sim\in V[\lambda]$ for all $u,v\in V$,
making $V$ a vertex subalgebra,
and such that 
$\int^\lambda\!d\sigma\,[v_\sigma m]^\sim\in M[\lambda]$ for all $v\in V$ and $m\in M$,
and 
$\int^\lambda\!d\sigma\,[m_\sigma n]^\sim=0$ for all $m,n\in M$,
making $M$ an abelian ideal.

The left $V$-module structure on $M$
induces a right $V$-module structure on $M$ given by
\begin{equation}\label{eq:right}
\int^\lambda d\sigma\, m_\sigma v
=
(-1)^{p(v)p(m)}\int^{-\lambda-\partial}d\sigma\, v_{\sigma}m
\,.
\end{equation}
\end{definition}

\subsection{The spaces $\mc O_{k+1}^{\star T}$}\label{sec:ostart}
In the following subsection, we will introduce the chiral operad, which governs vertex algebras. In order to do so, we need to define certain spaces of rational functions.
For $k\geq-1$, let $\mc O_{k+1}^{\star}$ be the algebra of 
polynomials in the variables $z_0,\dots,z_k$ localized on the diagonals $z_i=z_j$ for $i\neq j$.
In other words, $\mc O_{0}^{\star}=\mb F$ and
$$
\mc O_{k+1}^{\star}
=
\mb F[z_0,\dots,z_k][z_{ij}^{-1}]_{0\leq i<j\leq k}
\,\,,\,\,\text{ where }\,\, \qquad z_{ij}=z_i-z_j, \;\; k\geq0.
$$
We will denote by $\mc O_{k+1}^{\star T}$ the subalgebra 
of translation invariant elements, i.e.,
$$
\mc O_{k+1}^{\star T}
= \Ker \Bigl( \sum_{i=0}^k \dd_{z_i} \Bigr)
=\mb F[z_{ij}^{\pm 1}]_{0\leq i<j\leq k}.
$$
Note that $\mc O_{0}^{\star T}=\mc O_{1}^{\star T}=\mb F$.

Let $\mc D_{k+1}$ be the algebra of regular differential operators 
in $z_0,\dots,z_k$,
i.e., $\mc D_{0}=\mb F$ and
$$
\mc D_{k+1}
=
\mb F[z_0,\dots,z_k][\partial_{z_0},\dots,\partial_{z_k}], \qquad k\geq0.
$$
Let $\mc D_{k+1}^T$ be the subalgebra of translation invariant elements:
$\mc D_{0}^T=\mb F$ and
$$
\mc D_{k+1}^T
= \Ker \ad\Bigl( \sum_{i=0}^k \dd_{z_i} \Bigr)
= \bigl(\mb F[z_{ij}]_{0\leq i<j\leq k} \bigr)[\partial_{z_0},\dots,\partial_{z_k}] , \qquad k\geq0.
$$
%
\begin{lemma}\label{20160719:lem}
The function\/ $f(z_0,\dots,z_k)=\prod_{0\leq i<j\leq k}z_{ij}^{-1}$
is a cyclic element of the\/ $\mc D_{k+1}$-module\/ $\mc O^\star_{k+1}$.
Consequently, $f$ is a cyclic element of the\/ $\mc D_{k+1}^T$-module\/ $\mc O^{\star T}_{k+1}$.
\end{lemma}
\begin{proof}[Proof (P.\ Etingof)]
Consider the Bernstein--Sato polynomial $b(s)$ associated to 
$f^{-1}$,
which admits a differential operator $L(s)$ (regular is $z_i$ and $s$) 
such that $L(s)f^{-s-1}=b(s)f^{-s}$.
It is known that the roots of the function $b(s)$ 
are negative, and by \cite[Cor.1.3]{BW15} we have that $b(s)=\pm b(-s-2)$. 
So $-2,-3,\dots$ are not roots of $b(s)$.
Hence, 
$f^2=\frac1{b(-2)}L(-2)f,\,f^3=\frac1{b(-3)}L(-3)f^2,\dots$,
all lie in the $\mc D_{k+1}$-submodule of $\mc O^\star_{k+1}$ generated by $f$.
The claim follows.
\end{proof}
\subsection{The operad $P^\ch$}\label{sec:wch}

Let $V=V_{\bar 0}\oplus V_{\bar 1}$ be a vector superspace endowed
with an even endomorphism $\partial$.
For $k\geq0$, the space $V^{\otimes(k+1)}\otimes\mc O_{k+1}^{\star}$ carries the structure 
of a right $\mc D_{k+1}$-module defined by
letting $z_i$ act by multiplication on $\mc O_{k+1}^{\star}$,
and letting $\partial_{z_i}$ act by
\begin{equation}\label{20160629:eq1}
\begin{split}
\big(v_0\otimes&\dots\otimes v_k\otimes f(z_0,\dots,z_k)\big)\cdot \partial_{z_i}
\\
&= v_0\otimes\cdots\otimes(\partial v_i)\otimes\cdots\otimes v_k\otimes f(z_0,\dots,z_k)
\\
&-v_0\otimes\dots\otimes v_k\otimes \frac{\partial f}{\partial z_i}(z_0,\dots,z_k).
\end{split}
\end{equation}
By restriction, $V^{\otimes(k+1)}\otimes\mc O_{k+1}^{\star T}$ is a right $\mc D_{k+1}^T$-module,
where $z_{ij}\in \mc D_{k+1}^T$ act by multiplication on $\mc O_{k+1}^{\star T}$.
For $k=-1$, $V^{\otimes 0}\otimes\mc O_{0}^{\star} \cong \mb F$ is also a module over $\mc D_0=\mc D_0^T=\mb F$.

Consider the space 
\begin{equation}\label{20160722:eq1}
V[\lambda_0,\dots,\lambda_k]\big/\big\langle\partial+\lambda_0+\dots+\lambda_k\big\rangle
\,.
\end{equation}
Here and further, $\langle\Phi\rangle$ denotes the image of an endomorphim $\Phi$.
The space \eqref{20160722:eq1} carries the structure of a right $\mc D_{k+1}$-module defined 
by letting $z_i$ act by
\begin{equation}\label{20160629:eq2}
A(\lambda_0,\dots,\lambda_k)\cdot z_i
=
-\frac{\partial}{\partial\lambda_i}A(\lambda_0,\dots,\lambda_k)
\,,
\end{equation}
and $\partial_{z_i}$ act by
\begin{equation}\label{20160629:eq3}
A(\lambda_0,\dots,\lambda_k)\cdot\partial_{z_i}
=
-\lambda_i A(\lambda_0,\dots,\lambda_k)
\,.
\end{equation}
Indeed, it is straightforward to check that both the actions \eqref{20160629:eq1}
and \eqref{20160629:eq2}-\eqref{20160629:eq3}
satisfy the defining relations 
$A\cdot\partial_{z_i}\cdot z_j=A\cdot z_j\cdot\partial_{z_i}+\delta_{ij}A$,
and that the actions \eqref{20160629:eq2}-\eqref{20160629:eq3}
commute with the operator $\partial+\lambda_0+\dots+\lambda_k$.
Note that formula \eqref{20160629:eq2} can be generalized to an arbitrary polynomial
$P(z_0,\dots,z_k)$ as follows:
\begin{equation}\label{20160629:eq2-b}
A(\lambda_0,\dots,\lambda_k)\cdot P(z_0,\dots z_k)
=
P\Big(
-\frac{\partial}{\partial\lambda_0},\dots,-\frac{\partial}{\partial\lambda_k}\Big)
A(\lambda_0,\dots,\lambda_k)
\,.
\end{equation}

A right $\mc D_{k+1}^T$-module homomorphism
from $V^{\otimes(k+1)}\otimes\mc O_{k+1}^{\star T}$
to $V[\lambda_0,\dots,\lambda_k]\big/\big\langle\partial+\lambda_0+\dots+\lambda_k\big\rangle$
is then a linear map
\begin{equation}\label{20160629:eq2-c}
\begin{split}
X\colon
V^{\otimes(k+1)}\otimes\mc O_{k+1}^{\star T}
&\to
V[\lambda_0,\dots,\lambda_k]\big/\big\langle\partial+\lambda_0+\dots+\lambda_k\big\rangle
\,,\\
\vphantom{\Big(}
v_0\otimes\dots\otimes v_k\otimes f(z_0,\dots,z_k)
&\mapsto
X_{\lambda_0,\dots,\lambda_k}(v_0,\dots,v_k;f)\,,
\end{split}
\end{equation}
satisfying the following two \emph{sesquilinearity} conditions:
\begin{equation}\label{20160629:eq4}
\begin{array}{l}
\displaystyle{
\vphantom{\Big(}
X_{\lambda_0,\dots,\lambda_k}(v_0,\dots,(\partial+\la_i) v_i,\dots,v_k;f)
=X_{\lambda_0,\dots,\lambda_k}\Bigl(v_0,\dots,v_k;\frac{\partial f}{\partial z_i}\Bigr)
\,,} \\
\displaystyle{
\vphantom{\Big(}
X_{\lambda_0,\dots,\lambda_k}(v_0,\dots,v_k;z_{ij}f)
=\Bigl(\frac{\partial}{\partial\lambda_j}-\frac{\partial}{\partial\lambda_i}\Bigr)
X_{\lambda_0,\dots,\lambda_k}(v_0,\dots,v_k;f)
\,.}
\end{array}
\end{equation}

\begin{remark}\label{ref:trcov}
Consider the usual action of $\dd$ on $V^{\otimes(k+1)}$ as
$\sum_{i=0}^k \dd_i$,
where $\dd_i$ 
denotes the action of $\dd$ on the $i$-th factor.
Then since $\sum_{i=0}^k \frac{\partial f}{\partial z_i} = 0$ for every $f\in\mc O_{k+1}^{\star T}$, the first sesquilinearity implies
\begin{equation}\label{trcov1}
X(\dd v \otimes f) = -\sum_{i=0}^k \la_i X(v \otimes f) = \dd (X(v \otimes f)), \qquad v\in V^{\otimes(k+1)}.
\end{equation}
\end{remark}

We let $P^\ch(k+1)$ be the space of all 
right $\mc D_{k+1}^T$-homomorphisms \eqref{20160629:eq2-c},
i.e. all linear maps \eqref{20160629:eq2-c} 
satisfying the sesquilinearity conditions \eqref{20160629:eq4}.
Sometimes, in order to specify the variables of the function $f\in\mc O_{k+1}^{\star T}$,
we will denote the image of the map $X$ as
\begin{equation}\label{20160709:eq1}
X^{z_0,\dots,z_k}_{\lambda_0,\dots,\lambda_k}(v_0,\dots,v_k;f(z_0,\dots,z_k))
\,.
\end{equation}
Note that, by definition,
\begin{equation}\label{pch0}
P^\ch(0) = \Hom_{\mb F} (\mb F, V/\langle\dd\rangle) \cong V/\dd V
\end{equation}
and
\begin{equation}\label{pch1}
P^\ch(1) = \Hom_{\mb F[\dd]} (V, V[\la_0]/\langle\dd+\la_0\rangle) \cong \End_{\mb F[\dd]} (V).
\end{equation}
We will denote by $1\in P^\ch(1)$ the identity endomorphism, so that
\begin{equation}\label{pch2}
1_{\la_0}(v_0;f) = f v_0 + \langle \dd+\la_0 \rangle, \qquad v_0\in V, \;\; f\in \mc O_{1}^{\star T} = \mb F.
\end{equation}

The symmetric group $S_{k+1}$ has a right action on $P^\ch(k+1)$ by permuting simultaneously the inputs 
$v_0,\dots,v_k$ of $X$ and the corresponding variables $z_0,\dots,z_k$ in $f$. 
Explicitly, for $X\in P^\ch(k+1)$ and $\sigma \in S_{k+1}$, we have
\begin{equation}\label{20160629:eq5}
\begin{split}
(X^\sigma &)^{z_0,\dots,z_k}_{\lambda_0,\dots,\lambda_k}(v_0,\dots,v_k;f(z_0,\dots,z_k))
\\
&=
\epsilon_v(\sigma)
X^{z_0,\dots,z_k}_{\lambda_{i_0},\dots,\lambda_{i_k}}(v_{i_0},\dots,v_{i_k};f(z_{i_0},\dots,z_{i_k})),
\end{split}
\end{equation}
where $i_s=\sigma^{-1}(s)$ and $\epsilon_v(\sigma)$
is given by \eqref{eq:operad14}.

To define the structure of an operad, we need to specify how the maps from $P^\ch$ are composed. 
Let $X\in P^\ch(k+1)$, $Y\in P^\ch(m+1)$, and $h\in\mc O_{k+m+1}^{\star T}$.
We can write $h$ in the form
\begin{equation}\label{circ2}
h(z_0,\dots,z_{k+m}) = f(z_0,\dots,z_{k}) g(z_0,\dots,z_{k+m}), 
\end{equation}
where $f\in\mc O_{k+1}^{\star T}$, $g\in\mc O_{k+m+1}^{\star T}$, and $g$ has no poles at $z_i=z_j$ for $0\le i<j \le k$. Then we define
\begin{equation}\label{circ1}
\begin{split}
(Y &\circ_1 X)_{\la_0,\la_1,\dots,\la_{k+m}}^{z_0,z_1,\dots,z_{k+m}} \bigl(v_0,v_1,\dots,v_{k+m}; h(z_0,\dots,z_{k+m})\bigr) \\
&= Y_{\la_0',\la_{k+1},\dots,\la_{k+m}}^{z_0,z_{k+1},\dots,z_{k+m}} 
\bigl( X_{\la_0-\dd_{z_0},\dots,\la_k-\dd_{z_k}}^{z_0,\dots,z_k} 
( v_0,\dots,v_k; f(z_0,\dots,z_{k}) )_{\to} ,\\
& \qquad v_{k+1},\dots,v_{k+m}; g(z_0,\dots,z_{k+m}) |_{z_1=\cdots=z_k=z_0} \bigr),
\end{split}
\end{equation}
where $\la_0' = \la_0+\la_1+\cdots+\la_k$ and the arrow $\to$ means that we apply the derivatives
$\dd_{z_i}$ to $g$ before setting $z_i=z_0$ $(1\le i \le k)$.
\begin{lemma}\label{lem:pch}
The product \eqref{circ1} is a well defined map from\/ $P^\ch(k+1) \times P^\ch(m+1)$ to\/ $P^\ch(k+m+1)$.
\end{lemma}
\begin{proof}
First, we will check that $Y \circ_1 X$ is independent of the choice of factorization \eqref{circ2}. Let us denote the right-hand side of  \eqref{circ1} 
by $R(f,g)$, and consider $R(f,z_{ij}g)$ where $0\le i<j\le k$. Notice that for any polynomial $P$ and $0\le i\le k$, we have
\begin{equation}\label{pch11}
\begin{split}
P(\la_0&-\dd_{z_0},\dots,\la_k-\dd_{z_k}) (z_i g) |_{z_1=\cdots=z_k=z_0} \\
&= (z_0- \dd_{\la_i}) P(\la_0-\dd_{z_0},\dots,\la_k-\dd_{z_k}) g |_{z_1=\cdots=z_k=z_0} .
\end{split}
\end{equation}
In particular,
\begin{equation}\label{pch12}
\begin{split}
P(\la_0&-\dd_{z_0},\dots,\la_k-\dd_{z_k}) (z_{ij}g) |_{z_1=\cdots=z_k=z_0} \\
&= (\dd_{\la_j} - \dd_{\la_i}) P(\la_0-\dd_{z_0},\dots,\la_k-\dd_{z_k}) g |_{z_1=\cdots=z_k=z_0}.
\end{split}
\end{equation}
Hence, the sesquilinearity of $X$ implies that $R(z_{ij}f,g)=R(f,z_{ij}g)$.
This proves that $Y \circ_1 X$ is well defined.

We will show that $Y \circ_1 X$ satisfies the second sesquilinearity in \eqref{20160629:eq4}.
First, if again $0\le i<j\le k$, then $(\dd_{\la_j} - \dd_{\la_i}) \la_0' = 0$, and \eqref{pch12} 
implies $R(f,z_{ij}g) = (\dd_{\la_j} - \dd_{\la_i}) R(f,g)$ as desired.
On the other hand, if $k+1\le j\le k+m$, then using \eqref{pch11} and the sesquilinearity of $Y$, we obtain:
\begin{align*}
R&(f,z_{0j}g)
= Y_{\la_0',\la_{k+1},\dots,\la_{k+m}}^{z_0,z_{k+1},\dots,z_{k+m}} 
\bigl( (z_{0j}- \dd_{\la_0}) X_{\la_0-\dd_{z_0},\dots,\la_k-\dd_{z_k}}^{z_0,\dots,z_k} 
( v_0,\dots,v_k; f )_{\to} ,\\
& \qquad v_{k+1},\dots,v_{k+m}; g |_{z_1=\cdots=z_k=z_0} \bigr) \\
&= (\dd_{\la_j} - \dd_{\la_0'}) Y_{\la_0',\la_{k+1},\dots,\la_{k+m}}^{z_0,z_{k+1},\dots,z_{k+m}} 
\bigl( X_{\la_0-\dd_{z_0},\dots,\la_k-\dd_{z_k}}^{z_0,\dots,z_k} 
( v_0,\dots,v_k; f )_{\to} ,\\
& \qquad v_{k+1},\dots,v_{k+m}; g |_{z_1=\cdots=z_k=z_0} \bigr) \\
&- Y_{\la_0',\la_{k+1},\dots,\la_{k+m}}^{z_0,z_{k+1},\dots,z_{k+m}} 
\bigl( \dd_{\la_0} X_{\la_0-\dd_{z_0},\dots,\la_k-\dd_{z_k}}^{z_0,\dots,z_k} 
( v_0,\dots,v_k; f )_{\to} ,\\
& \qquad v_{k+1},\dots,v_{k+m}; g |_{z_1=\cdots=z_k=z_0} \bigr) \\
&= (\dd_{\la_j} - \dd_{\la_0}) R(f,g).
\end{align*}
All other cases for $i,j$ can be obtained from the above, by using the identity $z_{lm}=z_{lp}+z_{pm}$.
This proves that $Y\circ_1 X$ satisfies the second sesquilinearity in \eqref{20160629:eq4}.

To prove the first sesquilinearity, consider $\dd h/\dd z_i$ instead of $h$. Then from \eqref{circ2}, we get
\begin{align*}
\frac{\dd h}{\dd z_i}(z_0,\dots,z_{k+m}) &= \frac{\dd f}{\dd z_i}(z_0,\dots,z_{k}) g(z_0,\dots,z_{k+m}) \\
&+ f(z_0,\dots,z_{k}) \frac{\dd g}{\dd z_i}(z_0,\dots,z_{k+m}) \,.
\end{align*}
In the right-hand side, each of the two summands is factored as in \eqref{circ2}. Thus, 
$$
(Y \circ_1 X)_{\la_0,\la_1,\dots,\la_{k+m}}^{z_0,z_1,\dots,z_{k+m}} \Bigl(v_0,v_1,\dots,v_{k+m}; \frac{\dd h}{\dd z_i}\Bigr) 
=R\Bigl(\frac{\dd f}{\dd z_i},g\Bigr) + R\Bigl(f,\frac{\dd g}{\dd z_i}\Bigr).
$$
We consider two cases: $0\le i\le k$ and $k+1\le i\le k+m$. In the first case, by the sesquilinearity of $X$, \begin{align*}
R\Bigl(\frac{\dd f}{\dd z_i},g\Bigr) &= Y_{\la_0',\la_{k+1},\dots,\la_{k+m}}^{z_0,z_{k+1},\dots,z_{k+m}} 
\bigl( X_{\la_0-\dd_{z_0},\dots,\la_k-\dd_{z_k}}^{z_0,\dots,z_k} ( v_0,\dots, \dd v_i, \dots, v_k;f )_\to ,  \\
& \qquad v_{k+1},\dots,v_{k+m}; g |_{z_1=\cdots=z_k=z_0} \bigr) \\
&+Y_{\la_0',\la_{k+1},\dots,\la_{k+m}}^{z_0,z_{k+1},\dots,z_{k+m}} 
\bigl( X_{\la_0-\dd_{z_0},\dots,\la_k-\dd_{z_k}}^{z_0,\dots,z_k} ( v_0,\dots,v_k;f )_\to , \\
&\qquad v_{k+1},\dots,v_{k+m}; \bigl((\la_i-\dd_{z_i}) g\bigr) \big|_{z_1=\cdots=z_k=z_0} \bigr).
\end{align*}
Combined with $R(f,\dd g/\dd z_i)$, this gives exactly the first sesquilinearity \eqref{20160629:eq4} for $Y\circ_1 X$.
In the case $k+1\le i\le k+m$, we have $\dd f/\dd z_i=0$ and $R(f,\dd g/\dd z_i)$ gives the sesquilinearity of $Y\circ_1 X$
after applying the sesquilinearity of $Y$.
\end{proof}

We will extend the definition of the $\circ_1$ product as follows.
Fix $n\geq1$ and $m_1,\dots,m_n\geq0$,
and again use the notation \eqref{20170821:eq2}.
Consider $Y\in P^\ch(n)$, $X_i\in P^\ch(m_i)$, and $v_k\in V$ for $1\le i\le n$, $1\le k\le M=M_n$.
Let 
\begin{equation}\label{circ3}
w_i=v_{M_{i-1}+1}\otimes\dots\otimes v_{M_i}\,\in V^{\otimes m_i}, \qquad i=1,\dots,n,
\end{equation}
where $M_0=0$.
For $h\in\mc O_{M}^{\star T}$, we can write
\begin{equation}\label{circ4}
h(z_1,\dots,z_M) = g(z_1,\dots,z_M) \prod_{i=1}^n f_i(z_{M_{i-1}+1},\dots, z_{M_i}), 
\end{equation}
so that $f_i \in \mc O_{m_i}^{\star T}$ and $g\in\mc O_{M}^{\star T}$ has no poles at
$z_k=z_l$ for $M_{i-1}+1 \le k<l \le M_i$ ($1\le i\le n$).
Then the \emph{composition} $Y(X_1\otimes\dots\otimes X_n)\in P^\ch(M)$ is defined as follows:
\begin{equation}\label{circ5}
\begin{split}
\bigl(Y&(X_1\otimes\dots\otimes X_n)\bigr)_{\la_1,\dots,\la_{M}}^{z_1,\dots,z_{M}} \bigl(v_1\otimes\dots\otimes v_{M}; h(z_1,\dots,z_{M})\bigr) \\
&= \pm Y_{\La_1,\dots,\La_n}^{z_{M_1},\dots,z_{M_n}} 
\Bigl( \bigotimes_{i=1}^n \,
(X_i)_{\la_{M_{i-1}+1}-\dd_{z_{M_{i-1}+1}},\dots,\la_{M_i}-\dd_{z_{M_i}}}^{z_{M_{i-1}+1},\dots,z_{M_i}} 
( w_i; f_i )_{\to}; \\
& \qquad g(z_1,\dots,z_M) \big|_{z_k=z_{M_i} \, (M_{i-1}+1 \le k \le M_i, \,  1\le i\le n)} \Bigr),
\end{split}
\end{equation}
where $\La_i$ are given by \eqref{20170821:eq2} and the sign $\pm$ is 
\begin{equation}\label{circ6}
\pm=(-1)^{\sum_{i<j} p(X_j)(p(v_{M_{i-1}+1})+\dots+p(v_{M_i}))}
\end{equation}
(cf.\ \eqref{20170821:eq5b}).
As in \eqref{circ1}, we first take the partial derivatives of $g$ indicated by the arrows $\to$ and then we make the substitutions $z_k=z_{M_i}$.

It is clear that the $\circ_1$-product is a special case of the composition \eqref{circ5}, namely
$Y \circ_1 X = Y(X\otimes 1\otimes\cdots\otimes 1)$, where $1\in P^\ch(1)$ is the identity operator \eqref{pch2}.

\begin{proposition}\label{prop:pch}
The collection of vector superspaces\/ $P^\ch(n)$ $(n\ge 0)$, with the action of\/ $S_n$ described above, the compositions \eqref{circ5},
and unit\/ $1\in P^\ch(1)$, is an operad.
\end{proposition}
\begin{proof}
First, it is straightforward to generalize Lemma \ref{lem:pch} for the composition \eqref{circ5}: the left-hand side of \eqref{circ5}
is independent of the choice of factorization \eqref{circ4}, and it satisfies the sesquilinearity \eqref{20160629:eq4}. 
Hence, \eqref{circ5} are well-defined compositions in $P^\ch$.

The properties \eqref{eq:operad3} of the unit $1\in P^\ch(1)$ are obvious. The equivariance of the compositions \eqref{circ5}
under the action of the symmetric group is also easy to see. Its proof is identical to the proof for the $\mc Chom$ operad
from Section \ref{sec:LCA.1}.

Finally, the associativity of the compositions is also similar to the case of $\mc Chom$. The only additional ingredient is that we have
to take derivatives of functions and make substitutions in them. We use that, by the chain rule from calculus, 
\begin{equation}\label{circ7}
\dd_{z_0} \bigl( f(z_0,\dots,z_k) \big|_{z_1=\cdots=z_k=z_0} \bigr) = \sum_{i=0}^k \dd_{z_i} f(z_0,\dots,z_k) \big|_{z_1=\cdots=z_k=z_0}.
\end{equation}
Then in both sides of the associativity axiom
\begin{equation*}
(Y(X_1\otimes\dots\otimes X_n))(Z_1\otimes\dots\otimes Z_{M_n}) =
Y((X_1\otimes\dots\otimes X_n)(Z_1\otimes\dots\otimes Z_{M_n}))
\end{equation*}
the derivatives get spread in the same way over the different variables.
\end{proof}

\subsection{Vertex (super)algebra structures}\label{sec:4.4-ch}

As before, let $V$ be an $\mb F[\partial]$-module with parity $p$,
and $\Pi V$ be the same $\mb F[\partial]$-module with reversed parity $\bar p=1-p$.
Consider the $\mb Z$-graded Lie superalgebra 
$$W^{\ch}(\Pi V)=\bigoplus_{k=-1}^\infty W^{\ch}_k(\Pi V) = W(P^\ch(\Pi V)),$$
defined in Section \ref{sec:3.2} for an arbitrary operad.
Note that $W_{-1}$ and $W_0$ are given by \eqref{pch0} and \eqref{pch1} with $V$ replaced by $\Pi V$.
Hence, they are the same as for the $\mc Chom$ operad.

By definition, an odd element $X\in W^{\ch}_1(\Pi V)$ is
an odd $\mc D_2^T$-module homomorphism:
\begin{equation}\label{20160719:eq1}
X_{\lambda_0,\lambda_1}\,:\,\,
\Pi V\otimes\Pi V\otimes\mc O^{\star T}_2
\,\to\,\Pi V[\lambda_0,\lambda_1]/\langle\partial+\lambda_0+\lambda_1\rangle
\,,
\end{equation}
satisfying the sesquilinearity axioms \eqref{20160629:eq4}
and the symmetry condition \eqref{20160629:eq5}.
%
Since $\mc O_2^{\star T}=\mb F[z_{01}^{\pm1}]$,
the $\mc D^T_2$-module homomorphism \eqref{20160719:eq1}
is uniquely determined, via the sesquilinearity axioms \eqref{20160629:eq4},
by its values on $(\Pi V)^{\otimes2}\otimes z_{01}^{-1}$.
We have
$$
\Pi V[\lambda_0,\lambda_1]\big/\big\langle \partial+\lambda_0+\lambda_1\big\rangle
\simeq
\Pi V[\lambda]
$$
by equating $\lambda_0=\lambda,\,\lambda_1=-\lambda-\partial$.
Hence, an odd $X\in W^{\ch}_1(\Pi V)$
corresponds bijectively to an even linear map 
$V\otimes V\,\to\, V[\lambda]$, which we shall denote as follows
\begin{equation}\label{20160719:eq2}
u\otimes v\,\mapsto\,
\int^{\lambda}\!d\sigma[u_\sigma v]
=\,\,:\!uv\!:+\int_0^\lambda d\sigma[u_\sigma v]
\,.
\end{equation}
Here and further, when passing from the ``$X$''-notation 
to the ``$\int^\lambda\!d\sigma\,[\cdot\,_\sigma\,\cdot]$''-notation,
we identify the vector spaces $\Pi V$ and $V$.
The correspondence between $X\in W^{\ch}(\Pi V)$ and the map \eqref{20160719:eq2}
is as follows:
the corresponding to $X$ integral of $\lambda$-bracket is
\begin{equation}\label{20160719:eq3}
\int^{\lambda}\!d\sigma[u_\sigma v]
=
(-1)^{p(u)}X^{z_0,z_1}_{\lambda,-\lambda-\partial}\bigl(u,v;z_{10}^{-1}\bigr)
\,.
\end{equation}
Conversely,
given the integral of $\lambda$-bracket \eqref{20160719:eq2},
we associate to it the map $X$ as in \eqref{20160719:eq1} by letting
\begin{equation}\label{20160719:eq3b}
X^{z_0,z_1}_{\lambda_0,\lambda_1}\bigl(v_0,v_1;z_{10}^{-1}\bigr)
=
(-1)^{1+\bar p(v_0)}\int^{\lambda_0}\!d\sigma[{v_0}_{\sigma} v_1]
\,,
\end{equation}
and extending it to $(\Pi V)^{\otimes2}\otimes\mc O^{\star T}_2$
via the sesquilinearity axioms \eqref{20160629:eq4}.
In particular, by sesquilinearity, we have
\begin{equation}\label{20160719:eq3c}
X^{z_0,z_1}_{\lambda_0,\lambda_1}\bigl(v_0,v_1;1\bigr)
=
(-1)^{p(v_0)}
[{v_0}_{\lambda_0} v_1]
\,.
\end{equation}

We can translate the sesquilinearity and symmetry conditions for $X$
to axioms on the corresponding integral of $\lambda$-bracket \eqref{20160719:eq2}.
All the sesquilinearity conditions \eqref{20160629:eq4} translate
to 
\begin{equation}\label{20160707:eq1b}
\int^\lambda\!d\sigma\,[\partial u_\sigma v]=-\int^\lambda\!d\sigma\,\sigma[u_\sigma v]
\,\,,\,\,\,\,
\int^\lambda\!d\sigma\,[u_\sigma \partial v]=\int^\lambda\!d\sigma\,(\partial+\sigma)[u_\sigma v]
\,.
\end{equation}
while the symmetry conditions \eqref{20160629:eq5} on $X$
translate, in the notation \eqref{20160719:eq2}, to
\begin{equation}\label{20160707:eq2}
\int^\lambda d\sigma[u_\sigma v]
=(-1)^{p(u)p(v)}\int^{-\lambda-\partial} d\sigma[v_\sigma u]
\,.
\end{equation}
As a result, we get the following:
\begin{proposition}\label{20160719:prop}
The space\/ $W_1^{\ch}(\Pi V)$
is identified via \eqref{20160719:eq3} 
with the space of integrals of\/ $\lambda$-brackets
$$
\int^\lambda d\sigma[\cdot\,_\sigma\,\cdot]:\,V\otimes V\to V[\lambda]
\,,
$$
satisfying axioms (ii) and (iii) in the Definition \ref{def5} of a vertex algebra.
\end{proposition}

Next, let $X,Y\in W^\ch_1(\Pi V)_{\bar1}$.
We can rewrite their box-product \eqref{eq:box} 
in terms of the notation \eqref{20160719:eq3} with the integrals of $\lambda$-brackets
corresponding to $X$ and $Y$.
By Lemma \ref{20160719:lem},
the ring $\mc O^{\star T}_3=\mb F[z_{21}^{\pm1},z_{20}^{\pm1},z_{10}^{\pm1}]$
is generated as a $\mc D^T_3$-module by the cyclic element
$f=z_{21}^{-1}z_{20}^{-1}z_{10}^{-1}$.
Hence, to determine $X\Box Y$ (and to prove, for example, that $X\Box Y=0$)
it suffices to compute it for this function.

In the following three lemmas, we will compute the three summands contributing to $X\Box Y$ from \eqref{eq:box1}.
We will express them in terms of the notation \eqref{20160719:eq3}, with the above choice of $f$.

\begin{lemma}\label{lem:box1}
For\/ $X,Y\in W^\ch_1(\Pi V)_{\bar1}$, we have:
\begin{align*}
(X&\circ_1 Y)^{z_0,z_1,z_2}_{\la_0,\la_1,\la_2} \Bigl( v_0,v_1,v_2;\frac1{z_{21}z_{20}z_{10}} \Bigr) \\
&=(-1)^{p(v_1)}
\int^{\lambda_0}\!d\sigma_0\,
\int_{\lambda_1}^{\lambda_0+\lambda_1-\sigma_0}\!d\sigma_1\,
(\lambda_0+\lambda_1-\sigma_0-\sigma_1)
\big[[{v_0}_{\si_0} v_1]^Y\, {}_{\sigma_0+\sigma_1} v_2\big]^X
 \\
&
+(-1)^{p(v_1)}
\int^{\lambda_0}\!d\sigma_0\,
\int^{\lambda_1}\!d\sigma_1\,
(\lambda_0-\sigma_0)
\big[[{v_0}_{\sigma_0} v_1]^Y\, {}_{\sigma_0+\sigma_1} v_2\big]^X.
\end{align*}
\end{lemma}

\begin{proof}
We have by \eqref{circ1}:
\begin{align*}
(X&\circ_1 Y)^{z_0,z_1,z_2}_{\la_0,\la_1,\la_2} \Bigl( v_0,v_1,v_2;\frac1{z_{21}z_{20}z_{10}} \Bigr) \\
&=X^{z_0,z_2}_{\la_0+\la_1,\la_2} \Bigl( Y^{z_0,z_1}_{\la_0-\dd_{z_0},\la_1-\dd_{z_1}} \Bigl( v_0,v_1;\frac1{z_{10}} \Bigr)_{\to},
v_2; \frac1{z_{21}z_{20}}\Big|_{z_1=z_0} \Bigr) \\
&=(-1)^{p(v_0)} X^{z_0,z_2}_{\la_0+\la_1,\la_2} \Bigl( \int^{\la_0-\dd_{z_0}} d\sigma [{v_0}_\sigma v_1]^Y, 
v_2; \frac1{z_{21}z_{20}}\Big|_{z_1=z_0} \Bigr).
\end{align*}
Now we use Taylor's formula for a polynomial $F$:
\begin{align*}
F(&\la_0-\dd_{z_0}) \frac1{z_{21}z_{20}}\Big|_{z_1=z_0} 
= e^{-\dd_{z_0}\dd_{\la_0}} \frac{F(\la_0)}{z_{21}z_{20}}\Big|_{z_1=z_0} 
= \frac1{z_{20}(z_{20}+\dd_{\la_0})} F(\la_0) \\
&= \frac{1-e^{-\dd_{z_0}\dd_{\la_0}}}{\dd_{\la_0}}  \frac{F(\la_0)}{z_{20}}
=\int_0^{\dd_{z_0}} d\tau \, e^{-\tau\dd_{\la_0}} \frac{F(\la_0)}{z_{20}} \\
&= \int_0^{\dd_{z_0}} d\tau F(\la_0-\tau) \frac1{z_{20}} \,.
\end{align*}
Applying this to the previous expression,
we obtain:
\begin{align*}
(X&\circ_1 Y)^{z_0,z_1,z_2}_{\la_0,\la_1,\la_2} \Bigl( v_0,v_1,v_2;\frac1{z_{21}z_{20}z_{10}} \Bigr) \\
&=(-1)^{p(v_0)} X^{z_0,z_2}_{\la_0+\la_1,\la_2} \Bigl( \int_0^{\dd_{z_0}} d\tau \int^{\la_0-\tau} d\sigma [{v_0}_\sigma v_1]^Y, 
v_2; \frac1{z_{20}} \Bigr) \\
&=(-1)^{p(v_0)} X^{z_0,z_2}_{\la_0+\la_1,\la_2} \Bigl( \int_0^{\dd+\la_0+\la_1} d\tau \int^{\la_0-\tau} d\sigma [{v_0}_\sigma v_1]^Y, 
v_2; \frac1{z_{20}} \Bigr),
\end{align*}
where for the second equality we used the sesquilinearity \eqref{20160629:eq4}. After that, we can write the result in terms of notation \eqref{20160719:eq3}:
\begin{align*}
(-1)^{p(v_1)} \int^{\la_0+\la_1} d\sigma_1 \Bigl[ \Bigl( \int_0^{\dd+\la_0+\la_1} d\tau \int^{\la_0-\tau} d\sigma [{v_0}_\sigma v_1]^Y, 
\Bigr)_{\sigma_1} v_2 \Bigr]^X.
\end{align*}
Then using sesquilinearity (ii) from Definition \ref{def5}, we replace $\dd$ by $-\sigma_1$ in the second integral.
Now we change the order of integration with respect to $\tau$ and $\sigma$:
\begin{align*}
&\int_0^{\la_0+\la_1-\sigma_1} d\tau \int^{\la_0-\tau} d\sigma F(\si,\si_1) \\
&= \int^{\si_1-\la_1} d\sigma \int_0^{\la_0+\la_1-\sigma_1} d\tau F(\si,\si_1)
+\int_{\si_1-\la_1}^{\la_0} d\sigma \int_0^{\la_0-\sigma} d\tau F(\si,\si_1) \\
&= \int^{\si_1-\la_1} d\sigma (\la_0+\la_1-\sigma_1) F(\si,\si_1)
+\int_{\si_1-\la_1}^{\la_0} d\sigma (\la_0-\sigma) F(\si,\si_1).
\end{align*}
After that, we change the order of integration with respect to $\si_1$ and $\si_0=\sigma$
and make the change of variables $\si_1\mapsto\si_0+\si_1$:
\begin{align*}
& \int^{\la_0+\la_1} d\sigma_1 \int^{\si_1-\la_1} d\sigma (\la_0+\la_1-\sigma_1) F(\si,\si_1) \\
&+ \int^{\la_0+\la_1} d\sigma_1 \int_{\si_1-\la_1}^{\la_0} d\sigma (\la_0-\sigma) F(\si,\si_1) \\
&= \int^{\la_0} d\sigma_0 \int_{\la_1}^{\la_0+\la_1-\si_0} d\sigma_1 (\la_0+\la_1-\si_0-\sigma_1) F(\si_0,\si_0+\si_1) \\
&+ \int^{\la_0} d\sigma_0 \int^{\la_1} d\sigma_1 (\la_0-\sigma_0) F(\si_0,\si_0+\si_1).
\end{align*}
Combining these equations completes the proof of Lemma \ref{lem:box1}.
\end{proof}

\begin{lemma}\label{lem:box2}
For\/ $X,Y\in W^\ch_1(\Pi V)_{\bar1}$, we have:
\begin{align*}
(X&\circ_2 Y)^{z_0,z_1,z_2}_{\la_0,\la_1,\la_2} \Bigl( v_0,v_1,v_2;\frac1{z_{21}z_{20}z_{10}} \Bigr) \\
&=
(-1)^{1+p(v_1)}
\int^{\lambda_0}\!d\sigma_0\,
\int^{\lambda_1}\!d\sigma_1\,
(\lambda_0-\sigma_0)
\big[{v_0}_{\sigma_0} [{v_1}{}_{\sigma_1} v_2]^Y\big]^X \\
&+
(-1)^{1+p(v_1)}
\int^{\lambda_0}\!d\sigma_0\,
\int_{\lambda_1}^{\lambda_0+\lambda_1-\sigma_0}\!d\sigma_1\,
(\lambda_0+\lambda_1-\sigma_0-\sigma_1)
\big[{v_0}_{\sigma_0} [{v_1}{}_{\sigma_1} v_2]^Y\big]^X.
\end{align*}
\end{lemma}
\begin{proof}
Since $X\circ_2 Y = X(1\otimes Y)$, we have by \eqref{circ5}:
\begin{align*}
(X&\circ_2 Y)^{z_0,z_1,z_2}_{\la_0,\la_1,\la_2} \Bigl( v_0,v_1,v_2;\frac1{z_{21}z_{20}z_{10}} \Bigr) \\
&= (-1)^{\bar p(v_0)} X^{z_0,z_2}_{\lambda_0,\lambda_1+\lambda_2}\Bigl(v_0,
Y^{z_1,z_2}_{\lambda_1-\dd_{z_1},\lambda_2-\dd_{z_2}}
\Bigl(v_1,v_2; \frac1{z_{21}} \Bigr)_\to ; \frac1{z_{20}z_{10}} \Big|_{z_1=z_2} \Bigr).
\end{align*}
The rest of the proof is similar to that of Lemma \ref{lem:box1}.
\end{proof}

\begin{lemma}\label{lem:box3}
For\/ $X,Y\in W^\ch_1(\Pi V)_{\bar1}$, we have:
\begin{align*}
&\bigl((X\circ_2 Y)^{(12)} \bigr)^{z_0,z_1,z_2}_{\la_0,\la_1,\la_2} \Bigl( v_0,v_1,v_2;\frac1{z_{21}z_{20}z_{10}} \Bigr) \\
&=
(-1)^{p(v_1)+p(v_0)p(v_1)}
\int^{\lambda_0}\!d\sigma_0\,
\int^{\lambda_1}\!d\sigma_1\,
(\lambda_0-\sigma_0)
\big[{v_1}_{\sigma_1} [{v_0}{}_{\sigma_0} v_2]^Y\big]^X
\\
&+(-1)^{p(v_1)+p(v_0)p(v_1)}
\int^{\lambda_0}\!\!d\sigma_0\,
\int_{\lambda_1}^{\lambda_0+\lambda_1-\sigma_0}\!\!d\sigma_1\,
(\lambda_0+\lambda_1-\sigma_0-\sigma_1)
\big[{v_1}_{\sigma_1} [{v_0}{}_{\sigma_0} v_2]^Y\big]^X.
\end{align*}
\end{lemma}
\begin{proof}
Recall that $(X\circ_2 Y)^{(12)}$ is obtained from $X\circ_2 Y$
by switching the roles of $z_0$ and $z_1$,
$v_0$ and $v_1$, $\lambda_0$ and $\lambda_1$, and $\si_0$ and $\si_1$.
Then we perform a change of order of integration.
Note that 
there is a double sign change:
one is coming from the change of sign of the function $f=z_{21}^{-1}z_{20}^{-1}z_{10}^{-1}$
when we exchange $z_0$ and $z_1$,
and the other sign change pops out when we change 
the order of integration in $d\si_0$ and $d\si_1$.
\end{proof}

As a result of the above three lemmas, the box-product $X\Box Y$ can be written as follows
\begin{equation}\label{20160718:eq5}
\begin{split}
(-1&)^{p(v_1)+1}
(X\Box Y)^{z_0,z_1,z_2}_{\lambda_0,\lambda_1,\lambda_2}
\Big(v_0,v_1,v_2;\frac1{z_{21}z_{20}z_{10}}\Big) 
\\
&= 
\int^{\lambda_0}\!\!\!d\sigma_0\,
\int^{\lambda_1}\!\!\!d\sigma_1\,
(\lambda_0-\sigma_0)
j^{X,Y}_{\si_0,\si_1}(v_0,v_1,v_2) \\
&+
\int^{\lambda_0}\!d\sigma_0\,
\int_{\lambda_1}^{\lambda_0+\lambda_1-\sigma_0}
\!\!\!\!\!\!\!\!\!\!\!
d\sigma_1\,
(\lambda_0+\lambda_1-\sigma_0-\sigma_1)
j^{X,Y}_{\si_0,\si_1}(v_0,v_1,v_2)\,,
\end{split}
\end{equation}
where
\begin{equation}\label{20160718:eq5a}
\begin{split}
j^{X,Y}_{\si_0,\si_1}(v_0&,v_1,v_2)
=
\big[{v_0}_{\sigma_0} [{v_1}{}_{\sigma_1} v_2]^Y\big]^X \\
&-(-1)^{p(v_0)p(v_1)}
\big[{v_1}_{\sigma_1} [{v_0}{}_{\sigma_0} v_2]^Y\big]^X
-\big[[{v_0}_{\sigma_0} v_1]^Y\, {}_{\sigma_0+\sigma_1} v_2\big]^X.
\end{split}
\end{equation}

From this we can derive the main result of the present section.
\begin{theorem}\label{20160719:thm}
An odd element\/ $X\in W^{\ch}_1(\Pi V)$ satisfies\/ $X\Box X=0$
if and only if the corresponding integral of\/ $\lambda$-bracket\/ \eqref{20160719:eq3}
satisfies axiom (iv) of Definition \ref{def5}.
Consequently, such elements\/ $X$
are in bijective correspondence with the structures
of non-unital vertex algebras on the\/ $\mb F[\partial]$-module\/ $V$.
\end{theorem}
\begin{proof}
The symmetry condition $X=X^{(12)}$ on the element $X\in W^\ch_1(\Pi V)$
translates to the symmetry axiom (iii) in Definition \ref{def5}.
In the notation \eqref{20160721:eq2}, axiom (iv) of the Definition \ref{def5} of a vertex algebra 
reads
\begin{equation}\label{20160721:eq4}
J_{\lambda_0,\lambda_1}(v_0,v_1,v_2)=0
\,,
\end{equation}
while, by \eqref{20160718:eq5}, the condition that $X\Box X=0$ can be written as follows:
\begin{equation}\label{20160721:eq5}
\begin{array}{l}
\displaystyle{
\vphantom{\Big(}
\int^{\lambda_0}\!\!\!d\sigma_0\,
\int^{\lambda_1}\!\!\!d\sigma_1\,
(\lambda_0-\sigma_0)
j_{\sigma_0,\sigma_1}(v_0,v_1,v_2)
} \\
\displaystyle{
\vphantom{\Bigg(}
+
\int^{\lambda_0}\!d\sigma_0\,
\int_{\lambda_1}^{\lambda_0+\lambda_1-\sigma_0}
\!\!\!\!\!\!\!\!\!\!\!
d\sigma_1\,
(\lambda_0+\lambda_1-\sigma_0-\sigma_1)
j_{\sigma_0,\sigma_1}(v_0,v_1,v_2)
=0
\,,}
\end{array}
\end{equation}
where 
$$
\begin{array}{l}
\displaystyle{
\vphantom{\Big(}
j_{\lambda_0,\lambda_1}(v_0,v_1,v_2)
=
\frac{\partial^2}{\partial\lambda_0\partial\lambda_1}J_{\lambda_0,\lambda_1}(v_0,v_1,v_2)
} \\
\displaystyle{
\vphantom{\Big(}
=
[{v_0}_{\lambda_0}[{v_1}_{\lambda_1}v_2]]
-(-1)^{p(v_0)p(v_1)}[{v_1}_{\lambda_1}[{v_0}_{\lambda_0}v_2]]
-[[{v_0}_{\lambda_0}v_1]_{\lambda_0+\lambda_1}v_2]
,}
\end{array}
$$
which is the same as \eqref{20160718:eq5a} for $Y=X$.

By the sesquilinearity axiom (ii) and the change of variable of integration 
$\tilde{\sigma}_1=\sigma_0+\sigma_1$,
we can rewrite the left-hand side of \eqref{20160721:eq5},
using the notation \eqref{20160721:eq3}, as
$$
\widetilde{J}_{\lambda_0,\lambda_0+\lambda_1}((\lambda_0+\partial)v_0,v_1,v_2)
+\widetilde{J}_{\lambda_0,\lambda_0+\lambda_1}(v_0,(\lambda_1+\partial)v_1,v_2)
-J_{\lambda_0,\lambda_1}(v_0,(\lambda_1+\partial)v_1,v_2)
\,.
$$
Hence, due to Lemma \ref{20160721:lem}, we get that \eqref{20160721:eq4}
implies \eqref{20160721:eq5}.
Conversely, if we take the derivative with respect to $\lambda_0$ of the left-hand side of \eqref{20160721:eq5},
we get
$$
\widetilde{J}_{\lambda_0,\lambda_0+\lambda_1}(v_0,v_1,v_2)\,.
$$
Hence, \eqref{20160721:eq5} implies \eqref{20160721:eq4},
again due to Lemma \ref{20160721:lem}. 
\end{proof}

\section{Vertex algebra modules and cohomology complexes}\label{sec:5}

\subsection{Cohomology of vertex algebras}\label{sec:vacoh}

As a consequence of Theorems \ref{20170603:thm2} and \ref{20160719:thm},
we obtain a cohomology complex associated to a vertex algebra $V$.

\begin{definition}\label{def:vacoho}
Let $V$ be a (non-unital) vertex algebra.
The corresponding \emph{vertex algebra cohomology complex} is defined as
$$
(W^\ch(\Pi V),\ad X)
\,,
$$
where $X\in W^\ch(\Pi V)_{\bar 1}$ is associated to the vertex algebra structure of $V$
via \eqref{20160719:eq3b}.
\end{definition}

As in Section \ref{sec:4}, the cohomology 
$$
H^\ch(V) 
=
\ker(\ad X)/\im(\ad X)
$$
is a $\mb Z$-graded Lie superalgebra.
However, in order to stick to the tradition,
we shift the index by 1, namely for $k\geq0$ we let
$$
H^k(V) 
=
\ker\big(\ad X\big|_{W^{\ch}_{k+1}(\Pi V)}\big)\big/[X,W^{\ch}_k(\Pi V)
\,.
$$

We will generalize the above cohomological construction 
for an arbitrary module $M$ over a vertex algebra $V$.
To this end, we first need to generalize the construction 
of the Lie superalgebra $W^\ch(\Pi V)$.

\subsection{The space $W^{\ch}(\Pi V,\Pi M)$
}
\label{sec:4.2-ch}

Let $V$ and $M$ be vector superspaces with parity $p$, endowed 
with a structure of $\mb F[\partial]$-modules.
As usual we denote by $\Pi V$ and $\Pi M$ the same spaces with reversed parity $\bar p=1-p$.
We define the $\mb Z$-graded vector superspace (with parity still denoted by $\bar p$)
\begin{equation}\label{eq:WVM}
W^{\ch}(\Pi V,\Pi M)=\bigoplus_{k\geq-1}W^{\ch}_k(\Pi V,\Pi M)
\,,
\end{equation}
where $W^{\ch}_k(\Pi V,\Pi M)$ is the space of linear maps
$$
(\Pi V)^{\otimes(k+1)}\otimes\mc O^{\star T}_{k+1}\to
\Pi M[\lambda_0,\dots,\lambda_k]/\langle\partial+\lambda_0+\dots+\lambda_k\rangle
$$
satisfying the sesquilinearity conditions \eqref{20160629:eq4}, 
invariant with respect to the action of the symmetric group \eqref{20160629:eq5},
i.e.
\begin{align*}
& W^{\ch}_k(\Pi V,\Pi M) \\
& =\Hom_{\mc D^T_{k+1}}
\big(
(\Pi V)^{\otimes(k+1)}\otimes\mc O^{\star T}_{k+1},
\Pi M[\lambda_0,\dots,\lambda_k]/\langle\partial+\lambda_0+\dots+\lambda_k\rangle
\big)^{S_{k+1}}
\,.
\end{align*}
Of course, the Lie superalgebra $W^\ch(\Pi V)$ is a special case of \eqref{eq:WVM}
for $M=V$.

The space $W^{\ch}(\Pi V,\Pi M)$ is obtained as a subquotient of the universal
Lie superalgebra $W^{\ch}(\Pi V\oplus \Pi M)$,
via the canonical isomorphism of superspaces
\begin{equation}\label{100421:eq1-ch}
\begin{array}{c}
\mc U/\mc K
\stackrel{\sim}{\longrightarrow} W^{\ch}(\Pi V, \Pi M)\,,
\end{array}
\end{equation}
where $\mc U=\bigoplus_{k\geq-1}\mc U_k$ and $\mc K=\bigoplus_{k\geq-1}\mc K_k$,
and $\mc U_k$, $\mc K_k$ are the following subspaces of $W^{\ch}_k(\Pi V\oplus \Pi M)$:
\begin{align*}
&
\mc U_k
=
\Hom_{\mc D^T_{k+1}}
\!\!\!\Big(
(\Pi V\oplus \Pi M)^{\otimes(k+1)}\otimes\mc O^{\star T}_{k+1},
\Pi M[\lambda_0,\dots,\lambda_k]/\langle\partial\!+\!\lambda_0\!+\!\cdots\!+\!\lambda_k\rangle
\Big)^{S_{k+1}} , \\
&
\mc K_k
= 
\big\{Y \in \mc U_k \,\big|\, Y((\Pi V)^{\otimes(k+1)}\otimes\mc O^{\star T}_{k+1})=0\big\} \,,
\end{align*}
and \eqref{100421:eq1-ch} is the restriction map.
For example, we have the canonical isomorphisms 
\begin{equation}\label{20160914:eq1}
W^{\ch}_{-1}(\Pi V,\Pi M)\simeq \Pi M/\partial \Pi M
\,\,,\,\,
W^{\ch}_{0}(\Pi V,\Pi M)\simeq \Hom_{\mb F[\partial]}(\Pi V,\Pi M)
\,.
\end{equation}

The proof of the following proposition is obvious.
\begin{proposition}\label{100421:prop-ch}
Let\/ $X\in W^{\ch}_h(\Pi V\oplus \Pi M)$. Then the adjoint action of\/ $X$ 
on\/ $W^{\ch}(\Pi V\oplus\Pi M)$ leaves the subspaces\/ $\mc U$ and\/ $\mc K$ invariant
provided that the following two conditions hold:
\begin{enumerate}[(i)]
\item
$X^{z_0,\dots,z_h}_{\lambda_0,\dots,\lambda_h}(v_0,\dots,v_h;f)\in 
\Pi M[\lambda_0,\dots,\lambda_h]/\langle\partial+\lambda_0+\dots+\lambda_h\rangle$
if one of the arguments $v_i$ lies in $\Pi M$,
\item
$X^{z_0,\dots,z_h}_{\lambda_0,\dots,\lambda_h}(v_0,\dots,v_h;f)\in 
\Pi V[\lambda_0,\dots,\lambda_h]/\langle\partial+\lambda_0+\dots+\lambda_h\rangle$
if all the arguments $v_i$ lie in $\Pi V$.
\end{enumerate}
In this case, $\ad X$ induces a well-defined linear map on the space\/ $W^{\ch}(\Pi V,\Pi M)$,
via the isomorphism \eqref{100421:eq1-ch}.
\end{proposition}

\subsection{Cohomology of a vertex algebra with coefficients in a module}\label{sec:5.1}

As before, let $V$ and $M$ be vector superspaces with parity $p$,
endowed with $\mb F[\partial]$-module structures.
Consider the reduced superspace
$W^{\ch}(\Pi V,\Pi M)$
introduced in Section \ref{sec:4.2-ch}, with parity denoted by $\bar p$.

According to Definition \ref{def-mod},
to say that $V$ is a non-unital vertex algebra and $M$ is a $V$-module
is equivalent to say that
we have a vertex algebra structure 
$\int^\lambda\![\cdot\,_\lambda\,\cdot]^{\widetilde{}}$
on the $\mb F[\partial]$-module $V\oplus M$
extending that of $V$,
such that the integral of the $\lambda$-bracket
restricted to $V\otimes M$ is given by the vertex algebra action of $V$ on $M$,
and restricted to $M\otimes M$ vanishes.
Hence, such a structure corresponds, bijectively,
to an element $X$ of the following set:
\begin{equation}\label{100421:eq7-ch}
\begin{split}
\Big\{
X\in W^{\ch}_1(\Pi V&\oplus\Pi M)_{\bar1}\,\big|\,
[X,X]=0
\,,\,\,X_{\lambda_0,\lambda_1}(M\otimes M\otimes\mc O^{\star T}_2)=0
\\
& X_{\lambda_0,\lambda_1}(V\otimes V\otimes\mc O^{\star T}_2)\subset 
V[\lambda_0,\lambda_1]/\langle \partial+\lambda_0+\lambda_1\rangle
\,,\\
& X_{\lambda_0,\lambda_1}(V\otimes M\otimes\mc O^{\star T}_2)\subset 
M[\lambda_0,\lambda_1]/\langle \partial+\lambda_0+\lambda_1\rangle
\Big\}
\,.
\end{split}
\end{equation}
Explicitly, to $X$ in \eqref{100421:eq7-ch}
we associate the corresponding integral $\lambda$-bracket on $V$ given by \eqref{20160719:eq3}
and the corresponding integral $\lambda$-action of $V$ on $M$ given by
\begin{equation}\label{100421:eq8-ch}
\int^\lambda\!d\sigma\,v_\sigma m = 
(-1)^{p(v)}X^{z_0,z_1}_{\lambda,-\lambda-\partial}\Big(v,m;\frac1{z_{10}}\Big)
\,\,,\,\,\,\,
v\in V,\,m\in M\,.
\end{equation}

Note that every element $X$ in the set \eqref{100421:eq7-ch}
satisfies conditions (i) and (ii) in Proposition \ref{100421:prop-ch}.
Hence, $\ad X$ induces a well-defined endomorphism $d_X$ of $W^{\ch}(\Pi V,$ $\Pi M)$
such that $d_X^2=0$,
thus making $(W^{\ch}(\Pi V,\Pi M),d_X)$ a cohomology complex.
\begin{definition}
Let $V$ be a (non-unital) vertex algebra
and $M$ be a $V$-module.
The corresponding \emph{cohomology complex} of $V$ with coefficients in $M$ 
is defined as
$$
(W^\ch(\Pi V,\Pi M),d_X)
\,,
$$
where $X$ is the element of the set \eqref{100421:eq7-ch}
associated to the integral $\lambda$-bracket of $V$ by \eqref{20160719:eq3b}
and to the $V$-module structure of $M$ by \eqref{100421:eq8-ch}.
We denote by $H^{\ch}(V,M)=\bigoplus_{k\in\mb Z_+}H^k(V,M)$
the corresponding vertex algebra cohomology.
In the special case of the adjoint representation $M=V$
we recover the vertex algebra cohomology $H^{\ch}(V)$ from Definition \ref{def:vacoho}.
\end{definition}

The explicit formula for the differential $d_X$ can be obtained from \eqref{circ1}. 
In order to write such a formula, we need to split $h\in\mc O^{\star T}_{k+1}$ as in \eqref{circ2}.
For every $i=0,\dots,k$, we let
$$
h(z_0,\dots,z_k)
=
f_i(z_0,\stackrel{i}{\check{\dots}},z_k)g_i(z_0,\dots,z_k)
\,,
$$
where $g_i$ has no poles at $z_j=z_\ell$ for $j,\ell\neq i$,
and  for every $0\leq i<j\leq k$, we let
$$
h(z_0,\dots,z_k)
=
f_{ij}(z_i,z_j)g_{ij}(z_0,\dots,z_k)
\,,
$$
where $g_{ij}$ has no poles at $z_i=z_j$.
As a result,
for $Y\in W^{\ch}_{k-1}(\Pi V,\Pi M)$, we have
\begin{equation}\label{160912:eq1}
\begin{split}
& (d_X Y)^{z_0,\dots,z_k}_{\lambda_0,\dots,\lambda_k}(v_0,\dots,v_k;h(z_0,\dots,z_k))
= 
\sum_{i=0}^k (-1)^{(1+p(v_i))(s_{i+1,k}+k-i)}
\\
& \times
X^{w,z_i}_{\lambda_0+\stackrel{i}{\check{\dots}}+\lambda_k,\lambda_i}
\big(Y^{z_0,\stackrel{i}{\check{\dots}},z_k}_{
\lambda_0-\dd_{z_0},\stackrel{i}{\check{\dots}},\lambda_k-\dd_{z_k}}
(v_0,\stackrel{i}{\check{\dots}},v_k;f_i)_\to,v_i;g_i\big|_{z_0=\stackrel{i}{\check{\cdots}}=z_k=w}\big)
\\
& -(-1)^{\bar p(Y)}
\sum_{0\leq i<j\leq k} \!\!\!\!
(-1)^{(1+p(v_i))(s_{0,i-1}+i)+(1+p(v_j))(s_{0,i-1}+s_{i+1,j-1}+j-1)}
\\
& \times 
Y^{w,z_0,\stackrel{i}{\check{\dots}}\stackrel{j}{\check{\dots}},z_k}_{\lambda_i+\lambda_j,
\lambda_0,\stackrel{i}{\check{\dots}}\stackrel{j}{\check{\dots}},\lambda_k}
\big(X^{z_i,z_j}_{\lambda_i-\dd_{z_i},\lambda_j-\dd_{z_j}}
(v_i,v_j;f_{ij})_\to,v_0,\stackrel{i}{\check{\dots}}\,\stackrel{j}{\check{\dots}},v_k;g_{ij}\big|_{z_i=z_j=w}\big)
\,,
\end{split}
\end{equation}
where $s_{i,j}$ is given by
\begin{equation}\label{eq:sij}
s_{ij}=p(v_i)+\cdots+p(v_j) \,\,\text{ if }\,\, i\leq j
\,\,\text{ and }\,\,
s_{ij}=0 \,\,\text{ if }\,\, i>j
\,.
\end{equation}

As in \eqref{20160914:eq1}, we have the isomorphism 
$W^{\ch}_{-1}(\Pi V,\Pi M)\simeq M/\partial M$
obtained by identifying the map $Y:\,\mb F\to M/\langle\partial\rangle$ with
the element
\begin{equation}\label{20160920:eq1}
Y=Y(1)\in M/\partial M
\,\,,\,\,\text{ of parity }\,\,
p(Y)=1+\bar p(Y)
\,.
\end{equation}
We have the isomorphism $W^{\ch}_{0}(\Pi V,\Pi M)\simeq \Hom_{\mb F[\partial]}(V,M)$,
obtained by identifying the map 
$Y^z_{\lambda}(v;f(z)):\,V\simeq V\otimes\mc O^{\star T}_1\to M[\lambda_0]/\langle\partial+\lambda_0\rangle$
with the $\mb F[\partial]$-module homomorphism
$Y:\, V\to M$ given by
\begin{equation}\label{20160920:eq2}
Y(v)=Y^z_{\lambda}(v;1)
\,\,,\,\,\text{ of parity }\,\,
\bar p(Y)
\,.
\end{equation}
Finally, we identify an element
$$
Y^{z_0,z_1}_{\lambda_0,\lambda_1}(v_0,v_1;f(z_0,z_1)):\,V\otimes V \otimes\mc O^{\star T}_2
\to M[\lambda_0,\lambda_1]/\langle\partial+\lambda_0+\lambda_1\rangle
$$
of $W^{\ch}_{1}(\Pi V,\Pi M)$
with the integral $\lambda$-bracket 
$\tint^\lambda d\sigma[\cdot\,_\sigma\,\cdot]^Y:\, V\otimes V\to M[\lambda]$, 
given by (cf. \eqref{20160719:eq3})
\begin{equation}\label{20160920:eq3}
\int^{\lambda}\!d\sigma[u_\sigma v]^Y
=
(-1)^{p(u)}Y^{z_0,z_1}_{\lambda,-\lambda-\partial}\Big(u,v;\frac1{z_{10}}\Big)
\,\,,\,\,\text{ of parity }\,\,
1+\bar p(Y)
\,.
\end{equation}
As in Proposition \ref{20160719:prop},
the sesquilinearity and symmetry conditions for $Y$
translate to the corresponding sesquilinearity and symmetry conditions for
the integral $\lambda$-bracket, as in axioms (ii) and (iii) of Definition \ref{def5}.

We next write an explicit formula for the differential 
$d_X\colon W^\ch_{k-1}(\Pi V,\Pi M)\to W^\ch_k(\Pi V,\Pi M)$
in the special cases $k=0,1$ and $2$,
under the above identifications.
For $k=0$ we have $Y\in M/\partial M\simeq W^\ch_{-1}(\Pi V,\Pi M)$ and,
by equation \eqref{160912:eq1},
$d_XY$ corresponds to the following $\mb F[\partial]$-module homomorphism
from $V$ to $M$:
\begin{equation}\label{160916:eq1}
(d_X Y)(v)
=
X^{w,z_0}_{0,\lambda_0}(Y,v_0;1)
=
-(-1)^{(1+p(v))p(Y)}{v}_{-\partial}Y
\,.
\end{equation}
Next, for $k=1$, let $Y\in\Hom_{\mb F[\partial]}(V,M)\simeq W^\ch_0(\Pi V,\Pi M)$.
Then $d_XY$, given by equation \eqref{160912:eq1},
corresponds to the following integral $\lambda$-bracket of $u,v\in V$:
\begin{equation}\label{160916:eq2}
\begin{array}{l}
\vphantom{\Big(}
\displaystyle{
(-1)^{\bar p(Y)}
\int^\lambda\!d\sigma
[{u}_\sigma{v}]^{d_XY}
} \\
\vphantom{\Big(}
\displaystyle{
=
\int^\lambda\!d\sigma
[{Y(u)}_\sigma{v}]
+(-1)^{\bar p(Y) p(u)}
\int^\lambda\!d\sigma
[{u}_\sigma{Y(v)}]
-Y\Big(\int^\lambda\!d\sigma[{u}_\sigma{v}]\Big)
\,.}
\end{array}
\end{equation}
Finally, for $k=2$ we have $X\in W_1^\ch(\Pi V)_{\bar 1}$ and $Y\in W^\ch_1(\Pi V,\Pi M)$.
In this case $d_X(Y)=X\Box Y-(-1)^{\bar p(Y)} Y\Box X\in W^\ch_2(\Pi V,\Pi M)$,
where $X\Box Y$ is given by the same formula as in \eqref{20160718:eq5}.

\subsection{Casimirs, derivations and extensions}\label{sec:5.2}

Let $V$ be a non-unital vertex algebra and let $M$ be a $V$-module.
\begin{definition}\label{def:cas}
A \emph{Casimir element} is an element $\int m\in M/\partial M$
such that  $V_{-\partial}m=0$.
Denote by $\Cas(V,M)$ the space of Casimirs.
Note that, due to skewsymmetry of the $\lambda$-bracket,
$\Cas(V):=\Cas(V,V)=\big\{\tint v\in V/\partial V\,\big|\,[v_{\lambda}V]|_{\lambda=0}=0\big\}$.
\end{definition}
\begin{definition}\label{def:der}
A \emph{derivation} from $V$ to $M$
is an $\mb F[\partial]$-module homomorphism $D:\,V\to M$ such that 
\begin{equation}\label{eq:der}
D\Big(\int^\lambda\!d\sigma[{u}_\sigma{v}]\Big)
=
\int^\lambda\!d\sigma
({D(u)}_\sigma{v})
+(-1)^{p(D) p(u)}
\int^\lambda\!d\sigma
({u}_\sigma{D(v)})
\,.
\end{equation}
We say that a derivation is \emph{inner} if it has the following form:
\begin{equation}\label{eq:inner}
D_Y(v)=Y_{\lambda}v\,|_{\lambda=0}
\,\,\text{ for some }\,\, Y\in M/\partial M
\,.
\end{equation}
In the special case when $V=M$
we have the usual definition of a derivation of the vertex algebra $V$.
Denote by $\Der(V,M)$ the space of derivations from $V$ to $M$,
and by $\Inder(V,M)$ the subspace of inner derivations.
We also denote $\Der(V)=\Der(V,V)$ and $\Inder(V)=\Inder(V,V)$.
\end{definition}

We can now describe more explicitly the low degree cohomology.

\begin{theorem}\label{thm:lowcoho}
Let\/ $V$ be a (non-unital) vertex algebra and let\/ $M$ be a\/ $V$-module. Then:
\begin{enumerate}[(a)]
\item
$H^0(V,M)=\Cas(V,M)$.
In particular, $H^0(V)=\Cas(V)$.
\item
$H^1(V,M)=\Der(V,M)/\Inder(V,M)$.
In particular, $H^1(V)$ equals the factor of the Lie algebra\/ $\Der(V)$ 
of all derivations of\/ $V$
by the ideal of all inner derivations.
\item
$H^2(V,M)$ is the space of isomorhism classes of\/ $\mb F[\partial]$-split extension
of the vertex algebra\/ $V$ by the\/ $V$-module\/ $M$,
viewed as a (non unital) vertex algebra with trivial integral\/ $\lambda$-bracket.
\end{enumerate}
\end{theorem}
\begin{proof}
Straightforward,
using the explicit formulas \eqref{160916:eq1}, \eqref{160916:eq2} and \eqref{20160718:eq5} 
for the differential.
(Cf. \cite{DSK09} for a proof in the case of Lie conformal algebras.)
\end{proof}

\section{The associated graded of the chiral operad}\label{sec:grch}

\subsection{Filtration on $P^{\ch}$}\label{sec:filch}

We introduce an increasing filtration on the space of translation invariant rational functions 
$\mc O_{k+1}^{\star T} = \mb F[z_{ij}^{\pm 1}]_{0\leq i<j\leq k}$,
given by the number of divisors:
\begin{equation}\label{fil1}
\begin{split}
\fil^{-1} \mc O_{k+1}^{\star T} &= \{0\} \subset \fil^0 \mc O_{k+1}^{\star T} = \mc O_{k+1}^{T} = \mb F[z_{ij}] \subset
\fil^1 \mc O_{k+1}^{\star T} = \sum_{i<j} \mc O_{k+1}^{T} [z_{ij}^{-1}] \subset \\
\cdots&\subset \fil^r \mc O_{k+1}^{\star T} = \sum \mc O_{k+1}^{T} [z_{i_1,j_1}^{-1} \cdots z_{i_r,j_r}^{-1}] \subset\cdots
\subset  \mc O_{k+1}^{\star T}.
\end{split}
\end{equation}
In other words, the elements of $\fil^r \mc O_{k+1}^{\star T}$ are sums of rational functions with
at most $r$ divisors each (not counting multiplicities).
For example, 
\begin{equation}\label{fil2}
\frac1{z_{01} z_{12} z_{02}} = \frac1{z_{01} z_{02}^2} + \frac1{z_{12} z_{02}^2} \in \fil^2 \mc O_{3}^{\star T}
\end{equation}
has three divisors, but it lies in $\fil^2 \mc O_{3}^{\star T}$.
In fact, 
by using relations similar to \eqref{fil2},
it is not hard to prove (cf. Lemma \ref{lem:gr} below)
that the filtration \eqref{fil1} stabilizes:
$$
\fil^k\mc O^{\star T}_{k+1}=\mc O^{\star T}_{k+1}
\,.
$$
This filtration is invariant under the actions of $\mc D_{k+1}^T$ and of the symmetric group $S_{k+1}$.
It is compatible with the multiplication:
\begin{equation}\label{fil3}
(\fil^r \mc O_{k+1}^{\star T}) \cdot (\fil^s \mc O_{k+1}^{\star T}) \subset \fil^{r+s} \mc O_{k+1}^{\star T}.
\end{equation}
Furthermore, if $g\in \fil^s \mc O_{k+1}^{\star T}$ has no pole at $z_i=z_j$, then 
$g |_{z_i=z_j} \in \fil^s \mc O_{k}^{\star T}$.

Now we define a decreasing filtration of $P^\ch(k+1)$ by
\begin{equation}\label{fil4}
\fil^r P^\ch(k+1) = \bigl\{ X \in P^\ch(k+1) \,\big|\, X( V^{\otimes(k+1)} \otimes \fil^{r-1} \mc O_{k+1}^{\star T})=0 \bigr\}.
\end{equation}
We have: $\fil^0 P^\ch(k+1) = P^\ch(k+1)$ and $\fil^{k+1} P^\ch(k+1) = \{0\}$.

\begin{proposition}\label{prop:filch}
With the above filtration, $P^\ch$ is a filtered operad $($cf.\ \eqref{eq:filop}$)$.
\end{proposition}
\begin{proof}
The filtration \eqref{fil4} is invariant under the action of the symmetric group because the filtration \eqref{fil1} is.
In any operad, the compositions can be obtained from the $\circ_1$-product and the action of the symmetric group
(see \eqref{eq:operad8}, \eqref{eq:operad8a}). Thus, it is enough to prove that
$$
Y\circ_1 X \in\fil^{r+s} P^\ch(k+m+1) \quad\text{for}\quad X\in\fil^r P^\ch(k+1), \;\; Y\in\fil^s P^\ch(m+1). 
$$
To this end, we want to show that
the left-hand side of \eqref{circ1} vanishes for all $h\in \fil^{r+s-1} \mc O_{k+m+1}^{\star T}$.
By linearity, we can suppose that $h=fg$ as in \eqref{circ2}
and the number of divisors of $h$ is $\le r+s-1$. 
Since the divisors of $f$ and $g$ are disjoint,
the number of divisors of $h$ is the sum of the number of divisors of $f$ and $g$. 
Hence, $f\in \fil^{r-1} \mc O_{k+1}^{\star T}$ or $g\in \fil^{s-1} \mc O_{k+m+1}^{\star T}$.
Then we apply formula \eqref{circ1} to compute $Y\circ_1 X$.
In the first case, we have $X(f)=0$. In the second case, 
applying some derivatives and setting $z_1=\cdots=z_k=z_0$ in $g$, we will obtain an element
of $\fil^{s-1} \mc O_{m+1}^{\star T}$, which is annihilated by $Y$.
\end{proof}
As a consequence of Proposition \ref{prop:filch},
the associated graded spaces 
\begin{equation}\label{grr}
\gr^r P^\ch(n)
=
\fil^r P^\ch(n) / \fil^{r+1} P^\ch(n)
\end{equation}
form a graded operad (see the end of Section \ref{sec:3.1}).

\subsection{$n$-graphs}\label{sec:6a.1}

For $n\geq1$, we define an $n$-\emph{graph}
as a graph $\Gamma$ with the set of vertices $\{1,\dots,n\}$
and an arbitrary collection of oriented edges, denoted $E(\Gamma)$.
We denote 
by $\mc G(n)$ the collection of all $n$-graphs
without tadpoles,
and by $\mc G_0(n)$ the collection of all acyclic $n$-graphs,
i.e., $n$-graphs that have 
no cycles (including tadpoles and multiple edges).

For example, the set $\mc G_0(1)$ consists of the graph with a single vertex labelled $1$ and no edges,
the set $\mc G_0(2)$ consists of three graphs:
\begin{equation}\label{eq:2-graphs}
\begin{array}{l}
\begin{tikzpicture}
\draw (0.5,1) circle [radius=0.1];
\node at (0.5,0.7) {1};
\draw (1.5,1) circle [radius=0.1];
\node at (1.5,0.7) {2};
\node at (2,0.85) {,};
\draw (3.5,1) circle [radius=0.1];
\node at (3.5,0.7) {1};
\draw (4.5,1) circle [radius=0.1];
\node at (4.5,0.7) {2};
\draw[->] (3.6,1) -- (4.4,1);
\node at (5,0.85) {,};
\draw (6.5,1) circle [radius=0.1];
\node at (6.5,0.7) {1};
\draw (7.5,1) circle [radius=0.1];
\node at (7.5,0.7) {2};
\draw[<-] (6.6,1) --(7.4,1);
\node at (8,0.85) {,};
\end{tikzpicture}
\\
E(\Gamma)=\emptyset
\,\,\,\,\,,\,\,\,\,\qquad
E(\Gamma)\!=\!\{1\!\!\to\!\!2\}
\,\,,\,\,\qquad
E(\Gamma)\!=\!\{2\!\!\to\!\!1\}
\end{array}
\end{equation}
and $\mc G_0(3)$ consists of the following graphs,
with arbitrary orientation of all edges:
\begin{equation}\label{eq:3-graphs}
\begin{tikzpicture}
\draw (-4,2.5) circle [radius=0.1];
\node at (-4,2.2) {1};
\draw (-3.25,2.5) circle [radius=0.1];
\node at (-3.25,2.2) {2};
\draw (-2.5,2.5) circle [radius=0.1];
\node at (-2.5,2.2) {3};
\node at (-2,2.35) {,};
\draw (-1,2.5) circle [radius=0.1];
\node at (-1,2.2) {1};
\draw (-0.25,2.5) circle [radius=0.1];
\node at (-0.25,2.2) {2};
\draw (0.5,2.5) circle [radius=0.1];
\node at (0.5,2.2) {3};
\draw (-0.9,2.5) -- (-0.35,2.5);
\node at (1,2.35) {,};
\draw (2,2.5) circle [radius=0.1];
\node at (2,2.2) {1};
\draw (2.75,2.5) circle [radius=0.1];
\node at (2.75,2.2) {2};
\draw (3.5,2.5) circle [radius=0.1];
\node at (3.5,2.2) {3};
\draw (2.85,2.5) -- (3.4,2.5);
\node at (4,2.35) {,};
\draw (5,2.5) circle [radius=0.1];
\node at (5,2.2) {1};
\draw (5.75,2.5) circle [radius=0.1];
\node at (5.75,2.2) {2};
\draw (6.5,2.5) circle [radius=0.1];
\node at (6.5,2.2) {3};
\draw (5,2.6) to [out=90,in=90] (6.5,2.6);
\node at (7,2.35) {,};
\draw (-4,1) circle [radius=0.1];
\node at (-4,0.7) {1};
\draw (-3.25,1) circle [radius=0.1];
\node at (-3.25,0.7) {2};
\draw (-2.5,1) circle [radius=0.1];
\node at (-2.5,0.7) {3};
\draw (-3.9,1) -- (-3.35,1);
\draw (-3.15,1) -- (-2.6,1);
\node at (-2,0.85) {,};
\draw (-1,1) circle [radius=0.1];
\node at (-1,0.7) {1};
\draw (-0.25,1) circle [radius=0.1];
\node at (-0.25,0.7) {2};
\draw (0.5,1) circle [radius=0.1];
\node at (0.5,0.7) {3};
\draw (-0.9,1) -- (-0.35,1);
\draw (-1,1.1) to [out=90,in=90] (0.5,1.1);
\node at (1,0.85) {,};
\draw (2,1) circle [radius=0.1];
\node at (2,0.7) {1};
\draw (2.75,1) circle [radius=0.1];
\node at (2.75,0.7) {2};
\draw (3.5,1) circle [radius=0.1];
\node at (3.5,0.7) {3};
\draw (2.85,1) -- (3.4,1);
\draw (2,1.1) to [out=90,in=90] (3.5,1.1);
\node at (4,0.85) {.};
\end{tikzpicture}
\end{equation}
By convention, we also let $\mc G_0(0)=\mc G(0)$ be the set consisting of a single element (the empty graph,
with $0$ vertices).

An \emph{oriented cycle} $C$ of an $n$-graph $\Gamma\in\mc G(n)$ is, by definition, 
a collection of edges of $\Gamma$ forming a closed sequence (possibly with self intersections):
\begin{equation}\label{20170823:eq4a}
C=\{i_1\to i_2,\,i_2\to i_3,\dots,\,i_{s-1}\to i_s,\,i_s\to i_1\}\subset E(\Gamma)
\,.
\end{equation}

\subsection{The maps $X^\Ga$}\label{sec:grpch}

For an oriented graph $\Ga\in\mc G(n)$, we define
\begin{equation}\label{gr1}
p_\Ga = p_\Ga(z_1,\dots,z_n) = \prod_{(i\to j) \in E(\Ga)} z_{ij}^{-1} \in \fil^r \mc O_{n}^{\star T}, 
\qquad z_{ij}=z_i-z_j,
\end{equation}
where the product is over all edges of $\Ga$ and
$r$ is the number of edges. 
Note that if we change the orientation of a single edge of $\Ga$, then $p_\Ga$ will change sign.
For any graph $G$ with a set of vertices labeled by an index set $I$, we introduce the notation
\begin{equation}\label{com3}
\la_G = \sum_{i\in I} \la_i, \qquad \dd_{z_G} = \sum_{i\in I} \dd_{z_i}, 
\qquad \dd_G = \sum_{i\in I} \dd_i, 
\end{equation}
where $\dd_i=1\otimes\cdots\otimes\dd\otimes\cdots\otimes1$ denotes the action of $\dd$ on the $i$-th factor in a tensor product $V^{\otimes n}$.
Note that, by translation covariance, we have $\partial_{z_G}p_G=0$.
\begin{lemma}\label{lem:alb0}
For an acyclic graph\/ $\Gamma\in\mc G_0(n)$,
we have
$$
\mb F[\partial_{z_1},\dots,\partial_{z_n}]p_\Gamma
=
\mb F\big[z_{ij}^{-1}\,\big|\,(i\to j)\in E(\Gamma)\big]p_\Gamma
\,.
$$
\end{lemma}
\begin{proof}
Clearly, applying derivatives to the function $p_\Gamma$,
we get an element of the space 
$\mb F\big[z_{ij}^{-1}\,\big|\,(i\to j)\in E(\Gamma)\big]p_\Gamma$.
Hence, we only need to show the opposite inclusion, i.e. that
for arbitrary exponents $m_{ij}\geq1$,
we have
\begin{equation}\label{eq:alb1}
\prod_{(i\to j)\in E(\Gamma)}z_{ij}^{-m_{ij}}
\in
\mb F[\partial_{z_1},\dots,\partial_{z_n}]p_\Gamma
\,.
\end{equation}
Assuming, by induction, that \eqref{eq:alb1} holds,
we show how to apply derivatives in order to increase arbitrarily
the exponents of the function $\prod_{(i\to j)\in E(\Gamma)}z_{ij}^{-m_{ij}}$.
Fix an edge $e=(\alpha\to \beta)$ of the graph $\Gamma$,
and let $\Gamma\backslash e$ be the graph obtained by deleting the edge $e$
from $\Gamma$.
Since by assumption $\Gamma$ is acyclic,
the connected components $\Gamma_\alpha$ and $\Gamma_\beta$ of $\alpha$ and $\beta$ 
in $\Gamma\backslash e$ are disjoint.
Then it is easy to check that 
\begin{align*} 
-\frac1{m_{\alpha\beta}} \dd_{z_{\Gamma_\alpha}}
&\prod_{(i\to j)\in E(\Gamma)}z_{ij}^{-m_{ij}}
= \frac1{m_{\alpha\beta}} \dd_{z_{\Gamma_\beta}}
\prod_{(i\to j)\in E(\Gamma)}z_{ij}^{-m_{ij}}
\\
=
z_{\alpha\beta}^{-1} &\prod_{(i\to j)\in E(\Gamma)}z_{ij}^{-m_{ij}}
\,.
\end{align*}
The claim follows.
\end{proof}
\begin{lemma}\label{lem:alb1}
The space\/ $\fil^r\mc O_n^{\star T}$ is generated
as a\/ $\mc D_n^T$-module
by the functions\/ $p_\Gamma$,
with\/ $\Gamma\in\mc G_0(n)$ acyclic graphs with at most\/ $r$ edges.
\end{lemma}
\begin{proof}
Clearly, every function in $\fil^r\mc O_n^{\star T}$
can be written as a linear combination of functions of the form
\begin{equation}\label{eq:alb2}
f\, z_{i_1j_1}^{-m_{i_1j_1}}\cdots z_{i_rj_r}^{-m_{i_rj_r}}
\,,
\end{equation}
with $m_{i_\ell j_\ell}\geq0$ and $f$ polynomial.
We need to show that the function \eqref{eq:alb2}
can be obtained starting from some $p_\Gamma$
and acting with $\mc D_n^T$.
Let $\Gamma\in\mc G(n)$ be the graph with edges $(i_1\to j_1),\dots,(i_r\to j_r)$.
By a computation similar to \eqref{fil2} (cf. \eqref{cycle02} below), 
if the graph is not acyclic, then the function \eqref{eq:alb2} lies in $\fil^{r-1}\mc O_n^{\star T}$
and the claim holds by induction.
For an acyclic graph $\Gamma$, as an immediate consequence of 
Lemma \ref{lem:alb0}, we have that the function \eqref{eq:alb2} is generated by $p_\Gamma$.
\end{proof}

By restriction, for every $X\in P^\ch(n)$, we have maps
\begin{equation}\label{gr2}
\begin{split}
X^\Ga_{\lambda_1,\dots,\lambda_n} \colon
V^{\otimes n}
&\to
V[\lambda_1,\dots,\lambda_n]\big/\big\langle\partial+\lambda_1+\dots+\lambda_n\big\rangle
\,,\\
\vphantom{\Big(}
v_1\otimes\dots\otimes v_n
&\mapsto
X_{\lambda_1,\dots,\lambda_n}^{z_1,\dots,z_n}(v_1,\dots,v_n;p_\Ga)\,.
\end{split}
\end{equation}
By \eqref{fil4}, if $X\in \fil^r P^\ch(n)$, we have
$X^\Ga=0$ for graphs $\Ga$ with fewer than $r$ edges. Furthermore, relations among the $p_\Ga$'s lead to the following relations for the maps $X^\Ga$.
\begin{lemma}\label{lem:gr}
Let\/ $\Gamma\in\mc G(n)$ be a graph with\/ $r$ edges containing an oriented cycle\/
$C\subset E(\Gamma)$. Then we have the following \emph{cycle relations:}
\begin{enumerate}[(a)]
\item
$\vphantom{\Big(}
X^\Gamma=0$ for all\/ $X\in \fil^r P^\ch(n);$
\item
$\displaystyle{\sum_{e\in C}Y^{\Gamma\setminus e}=0}$ for all\/ $Y \in \fil^{r-1} P^\ch(n)$,
where\/ $\Gamma\setminus e$ is the graph obtained from\/ $\Gamma$ by removing the edge\/ $e$.
\end{enumerate}
\end{lemma}
\begin{proof}
After relabeling the vertices, we can assume that
\begin{equation*}
C=\{1\to 2,\,2\to 3,\dots,\,s-1\to s,\,s\to 1\}
\,.
\end{equation*}
Then $p_\Ga$ has a factor
\begin{equation*}
p_C= \frac1{z_{12} z_{23} \cdots z_{s-1,s} z_{s1}} \,.
\end{equation*}
Since $C$ has $s$ edges, we expect $p_C\in\fil^{s} \mc O_{n}^{\star T}$; however, we claim that in fact
$p_C\in\fil^{s-1} \mc O_{n}^{\star T}$. Indeed, using the relation
\begin{equation}\label{cycle01}
\sum_{e\in C} z_e = z_{12} + z_{23} + \cdots + z_{s-1,s} + z_{s1} = 0,
\end{equation}
we have as in \eqref{fil2}, 
\begin{equation}\label{cycle02}
\begin{split}
-p_C &= \frac{-z_{12}}{ z_{12}^2 z_{23} \cdots z_{s-1,s} z_{s1} } \\
&= \frac{z_{23}}{ z_{12}^2 z_{23} \cdots z_{s-1,s} z_{s1} } + \cdots 
+ \frac{z_{s-1,s}}{ z_{12}^2 z_{23} \cdots z_{s-1,s} z_{s1} } 
+ \frac{z_{s1}}{ z_{12}^2 z_{23} \cdots z_{s-1,s} z_{s1} }
\,.
\end{split}
\end{equation}
In particular, $p_\Ga\in\fil^{r-1} \mc O_{n}^{\star T}$, which implies
claim (a).
Claim (b) follows from the equation
\begin{equation*}
0 = \sum_{e\in C} z_e p_\Ga = \sum_{e\in C} p_{\Ga\setminus e} \,.
\end{equation*}
\end{proof}

\begin{example}\label{ex:sign}
Let $Y \in \fil^{r-1} P^\ch(n)$ and $\Ga' \in\mc G_0(n)$ be a graph with $r-1$ edges. For a fixed edge $e=(i\to j)$ of $\Ga'$, we denote by $\Ga$ 
the graph obtained by adding the opposite edge $e'=(j\to i)$. Then $\Ga$ has a $2$-cycle $C=\{e,e'\}$, and 
$\Ga''=\Ga\setminus e$ is obtained from $\Ga'=\Ga\setminus e'$ by reversing the orientation of $e$. In this case, \eqref{eq:cycle2} implies
$Y^{\Ga'} = -Y^{\Ga''}$.
\end{example}
We will derive additional relations from the sesquilinearity conditions \eqref{20160629:eq4} for $X\in \fil^r P^\ch(n)$.
\begin{lemma}\label{lem:gr2}
Let\/ $X\in\fil^r P^\ch(n)$ and $\Gamma\in\mc G(n)$ be a graph with\/ $r$ edges.
Denote the connected components of\/ $\Gamma$ by\/ $\Gamma_\alpha$.
Then we have the following \emph{sesquilinearity relations:}
\begin{enumerate}[(a)]
\item
$\displaystyle{
(\partial_{\lambda_i} - \partial_{\lambda_j})
X^\Ga_{\lambda_1,\dots,\lambda_n} = 0 \quad\text{for any}\quad (i\to j) \in E(\Ga)}$,
which means that\/ $X^\Ga$ is a polynomial of the sums\/ $\la_{\Ga_\al};$
\item
$\displaystyle{\vphantom{\Big(}
X^{\Gamma}_{\lambda_1,\dots,\lambda_n} \bigl( \dd_{\Ga_\al}(v_1\otimes\dots\otimes v_n) \bigr)
= - \la_{\Ga_\al} X^{\Gamma}_{\lambda_1,\dots,\lambda_n} (v_1\otimes\dots\otimes v_n)
}$.
\end{enumerate}
\end{lemma}
\begin{proof}

If $(i\to j)$ is an edge of a graph $\Ga$ with $r$ edges, then $z_{ij} p_\Ga \in\fil^{r-1} \mc O_{n}^{\star T}$. Hence,
$$
X_{\lambda_1,\dots,\lambda_n}^{z_1,\dots,z_n}(v_1,\dots,v_n; z_{ij} p_\Ga) = 0.
$$
Claim (a) then follows from the sesquilinearity condition \eqref{20160629:eq4}.
Next, let us prove claim (b).
Since $\Ga$ is a disjoint union of the $\Ga_\al$'s, the function $p_\Ga$ 
is the product of the corresponding $p_{\Ga_\al}$'s.
By the translation covariance of $p_{\Gamma_\alpha}$,
we have
$\dd_{z_{\Ga_\al}} p_{\Ga_\al} = 0$, 
and hence
$\dd_{z_{\Ga_\al}} p_{\Ga} = 0$.
Claim (b) then follows again from \eqref{20160629:eq4}.
\end{proof}

\subsection{Compositions of the maps $X^\Ga$}\label{sec:comgr}

Now we will investigate how the maps \eqref{gr2} compose. For $X\in P^\ch(k+1)$ and $Y\in P^\ch(m+1)$, their $\circ_1$-product $Y \circ_1 X \in P^\ch(k+m+1)$ is given by \eqref{circ1}. We want to find $(Y \circ_1 X)^\Ga$, where $\Ga\in\mc G(k+m+1)$ is a graph whose vertices are labeled by $0,1,\dots,k+m$. In order to apply \eqref{circ1} for $h=p_\Ga$, according to \eqref{circ2}, we factor
\begin{equation}\label{com1}
p_\Ga = f(z_0,\dots,z_{k}) g(z_0,\dots,z_{k+m}), \qquad f=p_{\Ga'}, \;\; g=p_{\Ga''}.
\end{equation}
Here $\Ga'$ is the subgraph of $\Ga$ with vertices $0,1,\dots,k$ and all edges from $\Ga$ among these vertices;
$\Ga''$ is the subgraph of $\Ga$ that includes all edges of $\Ga$ not in $\Ga'$. 
The factorization \eqref{com1} holds because $E(\Ga)$ is the disjoint union of $E(\Ga')$ and $E(\Ga'')$.

Setting $z_1=\cdots=z_k=z_0$ in $p_{\Ga''}$ corresponds to contracting the vertices $0,1,\dots,k$ to a single vertex labeled $0$.
We let $\bar\Ga''$ be the graph with vertices labeled $0,k+1,\dots,k+m$ and edges obtained from the edges of $\Ga''$ by replacing any vertex $0\le i\le k$ with $0$, keeping the same orientation. Then
\begin{equation}\label{com2}
p_{\Ga''} |_{z_1=\cdots=z_k=z_0} = p_{\bar\Ga''}.
\end{equation}

Finally, introduce graphs $G_i$ $(0\le i\le k)$ as follows. Take the connected component of the vertex $i$ in $\Ga''$ and remove from it the vertex $i$ and all edges connected to $i$. Then $G_i$ is the resulting subgraph of $\Ga''$. Note that, by construction, the vertices of $G_i$ form a subset of $\{k+1,\dots,k+m\}$.
For another description of the graphs $\Gamma'$, $\bar\Gamma''$
and $G_i$, see Examples \ref{ex:gamma} and \ref{ex:gamma2} below.

\begin{proposition}\label{prop:com}
With the above notation, suppose that the graph\/ $\bar\Ga''$ is acyclic. Then
\begin{equation}\label{com4}
\begin{split}
(Y \circ_1 X)_{\la_0,\la_1,\dots,\la_{k+m}}^\Ga &(v_0, v_1,\dots, v_{k+m} ) \\
= Y_{\la_{\Ga'},\la_{k+1},\dots,\la_{k+m}}^{\bar\Ga''} &\bigl( 
X_{\la_0+\la_{G_0}+\dd_{G_0},\dots,\la_k+\la_{G_k}+\dd_{G_k}}^{\Ga'} (v_0,\dots,v_k), v_{k+1},\dots,v_{k+m} \bigr)
\end{split}
\end{equation}
for\/ $X\in P^\ch(k+1)$ and\/ $Y\in P^\ch(m+1)$.
Here we also use the notation \eqref{com3} with\/ $\dd_i$ representing the action of\/ $\dd$ on\/ $v_i$.

Assume, in addition, that\/
$X\in\fil^r P^\ch(k+1)$, $Y\in\fil^s P^\ch(m+1)$,
$\Gamma'$ has\/ $r$ edges 
and\/ $\bar\Gamma''$ has $s$ edges.
Then equation \eqref{com4} holds without the assumption that\/ $\bar\Ga''$ is acyclic.
\end{proposition}
\begin{proof}
In order to apply \eqref{circ1}, we need to compute $\partial_{z_i}p_{\Gamma''}$.
After possibly changing a sign, we will assume that all edges of $\Ga''$ are oriented as $(i\to j)$ with $i<j$. 
The assumption that $\bar\Ga''$ has no cycles implies that $G_i$ and $G_l$ are disconnected for $0\le i,l\le k$.
Let $E_i$ be the set of all edges of $\Ga''$ starting from the vertex $i$.
Then we can write
\begin{equation*}
E(\Ga'') = \bigsqcup_{i=0}^k \bigl( E_i \sqcup E(G_i) \bigr) \sqcup F
\end{equation*}
for some subset $F$ of edges among vertices $k+1,\dots,k+m$.
Thus
$$
p_{\Gamma''}
=
p_F
\prod_{i=0}^k
p_{E_i}p_{G_i}
\,,
$$
where $p_F=\prod_{e\in F}z_e^{-1}$, and similarly for $p_{E_i}$.

For every edge $(i\to j) \in E_i$, we have $-\dd_{z_i} z_{ij}^{-1} = \dd_{z_j} z_{ij}^{-1}$.
Hence
$$
-\dd_{z_i}p_{E_i}
=
\sum_{(i\to j)\in E_i}\partial_{z_j}p_{E_i}
=
\partial_{z_{G_i}}
p_{E_i}
\,.
$$
Using that $\dd_{z_{G_i}} p_{G_i} = 0$ and $\partial_{z_{G_i}}p_F=0$, this implies
\begin{equation*}
-\dd_{z_i} p_{\Ga''} = \dd_{z_{G_i}} p_{\Ga''}.
\end{equation*}
The statement then follows from \eqref{circ1}, by applying equation \eqref{com2}
and the sesqui-linearity \eqref{20160629:eq4},
after observing that $\la_{\Ga'}=\la_0'= \la_0+\la_1+\cdots+\la_k$.

For the last assertion of the proposition,
if $\bar\Gamma''$ has a cycle, 
then by Lemma \ref{lem:gr}(a) the right-hand side of \eqref{com4} vanishes.
Hence, we need to check that, in this case, the left-hand side of \eqref{com4}
vanishes as well.
This follows from formula \eqref{circ1}
and the fact that, after differentiating $p_{\Gamma''}$ and setting $z_1=\dots=z_k=z_0$,
the resulting function is in $\fil^{r-1}\mc O^{\star T}_{m+1}$.
\end{proof}


We can summarize all the previous results as follows:
\begin{corollary}\label{cor:gr}
\begin{enumerate}[(a)]
\item
For every\/ $X\in\fil^r P^\ch(k+1)$ and every graph\/ $\Gamma\in\mc G(k+1)$ with at most\/ $r$ edges,
the map
$$
X^\Ga \colon
V^{\otimes(k+1)}
\to
V[\lambda_0,\dots,\lambda_k]\big/\big\langle\partial+\lambda_0+\dots+\lambda_k\big\rangle
$$
defined by \eqref{gr2},
satisfies the cycle relations (a) and (b) from Lemma \ref{lem:gr}
and the sesquilinearity relations (a) and (b) from Lemma \ref{lem:gr2}.
\item
For\/ $X\in\fil^r P^\ch(k+1)$, $Y\in\fil^s P^\ch(m+1)$,
and for\/ $\Gamma\in \mc G(k+m+1)$ such that\/ $\Gamma'$ 
has at most\/ $r$ edges and\/ $\bar\Gamma''$ has at most\/ $s$ edges,
equation \eqref{com4} holds.
\item
If\/ $X\in\fil^r P^\ch(k+1)$ is such that\/ $X^\Gamma=0$ for all graphs\/ $\Gamma\in\mc G(k+1)$
with\/ $r$ edges,
then\/ $X\in\fil^{r+1} P^\ch(k+1)$.
\end{enumerate}
\noindent
Hence, we have an induced injective map defined on the associated graded space\/ $\gr^rP^\ch(k+1)$,
such that
$$
\bar X
\,\mapsto\,
\tilde X=\{X^\Gamma\,|\,\Gamma\in\mc G(k+1) \,\,\,\,\mathrm{ with }\,\,\,\, r \,\,\,\,\mathrm{ edges }\}
\,.
$$
\end{corollary}
\begin{proof}
Claim (a) is given by Lemmas \ref{lem:gr}, \ref{lem:gr2}.
Claim (b) is given by Proposition \ref{prop:com}.
Claim (c) follows from Lemma \ref{lem:alb1} and and the sesquilinearity conditions.
\end{proof}
Using this corollary,
in Section \ref{sec:6} below, we will provide a more detailed description of the 
associated graded operad $\gr P^\ch$.

\subsection{Refinement of the filtration on $P^{\ch}$}\label{sec:filch-ref}

We refine the filtration of the chiral operad $P^\ch$ introduced in Section \ref{sec:filch}
as follows.
Let $V$ be a vector superspace with an increasing filtration
\begin{equation}\label{eq:last3}
\fil^{-1}V=\{0\}
\,\subset\,
\fil^0V
\,\subset\,
\fil^1V
\,\subset\,
\fil^2V
\,\subset\,
\cdots
\,\subset\,
V
\,.
\end{equation}
This induces an increasing filtration on the tensor products
$$
\fil^s\big(V^{\otimes(k+1)}\otimes \mc O_{k+1}^{\star T}\big)
=
\sum_{r_0+r_1+\dots+r_{k+1}=s}
\fil^{r_0}V
\otimes
\cdots
\otimes
\fil^{r_k}V
\otimes
\fil^{r_{k+1}}\mc O_{k+1}^{\star T}
\,,
$$
if $s\geq0$, and $\fil^s=0$ if $s<0$.
For example, for $k=1$, we have
\begin{equation}\label{eq:last2}
\fil^s\big(V^{\otimes2}\otimes \mc O_2^{\star T}\big)
=
\big(
\fil^s(V^{\otimes2})\otimes \mc O_2^T
\big)
\oplus
\big(
\fil^{s-1}(V^{\otimes2})\otimes \mc O_2^{\star T}
\big)
\,.
\end{equation}
The corresponding refined filtered space $\fil^r P^\ch(k+1)$ is defined 
as the set of elements $X \in P^\ch(k+1)$ such that
\begin{equation}\label{fil4-ref}
X\big( 
\fil^s(V^{\otimes(k+1)} \otimes \mc O_{k+1}^{\star T})
\big)
\subset
(\fil^{s-r}V)[\lambda_0,\dots,\lambda_k]/\langle\partial+\lambda_0+\dots+\lambda_k\rangle
\,,
\end{equation}
for every $s$.
This is a decreasing filtration, possibly infinite in both directions.

\begin{proposition}\label{prop:filch-ref}
With the above refined filtration, $P^\ch(V)$ is a filtered operad $($cf.\ \eqref{eq:filop}$)$.
Hence, 
we have the corresponding Lie superalgebra filtration $\fil^r W^\ch(V)$ of $W^\ch(V)$.
\end{proposition}
\begin{proof}
The proof of the first statement is the same as for Proposition \ref{prop:filch}.
The last assertion follows from Theorem \ref{20170603:thm2}(c).
\end{proof}

Recall that a \emph{filtered} vertex algebra is a vertex algebra $V$
with an increasing filtration \eqref{eq:last3}
such that 
\begin{equation}\label{eq:last4}
{:}(\fil^pV)(\fil^qV){:}
\,\subset\,
\fil^{p+q}V
\,\,\text{ and }\,\,
[\fil^pV\,_\lambda\,\fil^qV]
\,\subset\,
\fil^{p+q-1}V[\lambda]
\,,
\end{equation}
for all $p,q$.
\begin{theorem}\label{20160719:thm-ref}
Let\/ $V$ be a filtered vector superspace.
Under the correspondence from Theorem \ref{20160719:thm},
the structures of filtered vertex algebra on\/ $V$
are in bijection with the 
odd elements\/ $X\in\fil^1 W^{\ch}_1(\Pi V)$ satisfying\/ $X\Box X=0$.
\end{theorem}
\begin{proof}
If $V$ is a filtered vertex algebra,
then, due to \eqref{eq:last4}, the corresponding $X$ satisfies
$$
X^{z_0,z_1}_{\lambda_0,\lambda_1}(\fil^pV\otimes \fil^qV\otimes 1)
=
[\fil^pV_{\lambda_0} \fil^qV]
\subset
\fil^{p+q-1}V[\lambda_0]
$$
and
$$
X^{z_0,z_1}_{\lambda_0,\lambda_1}\big(\fil^pV\otimes \fil^qV\otimes \frac1{z_{10}}\big)
=
{:}(\fil^pV)(\fil^qV){:}
\subset
\fil^{p+q}V
\,.
$$
By \eqref{eq:last2}, this means that $X\in\fil^1 W^{\ch}_1(\Pi V)$.
\end{proof}

\section{The cooperad of $n$-graphs}\label{sec:6a}

\subsection{Cocomposition of $n$-graphs}\label{sec:6a.2}

As in Section \ref{sec:6a.1}, 
let $\mc G(n)$ be the collection of all $n$-graphs
which have no tadpoles,
and $\mc G_0(n)$ be the collection of all acyclic $n$-graphs.
Fix an $n$-tuple $(m_1,\dots,m_n)$ of positive integers,
and let $M_1,\dots,M_n$ as in \eqref{20170821:eq2a}.
We define the \emph{cocomposition map}
\begin{equation}\label{20170614:eq2}
\Delta^{m_1\dots m_n}
\colon
\mc G(M_n)\to\mc G(n)\times\mc G(m_1)\times\dots\times\mc G(m_n)
\,,
\end{equation}
denoted
\begin{equation}\label{20170614:eq3}
\Gamma\,\mapsto\,
\Delta^{m_1\dots m_n}_0(\Gamma),\Delta^{m_1\dots m_n}_1(\Gamma),\dots,
\Delta^{m_1\dots m_n}_n(\Gamma)
\,,
\end{equation}
as follows.
$\Delta^{m_1\dots m_n}_1(\Gamma)$ is the subgraph of $\Gamma$ associated 
to the vertices $\{1,\dots,M_1\}$,
$\Delta^{m_1\dots m_n}_2(\Gamma)$ is the subgraph of $\Gamma$ associated 
to the vertices $\{M_1+1,\dots,M_2\}$
(which we relabel $\{1,\dots,m_2\}$),
and so on up to
$\Delta^{m_1\dots m_n}_n(\Gamma)$, which is the subgraph of $\Gamma$ 
associated to the last $m_n$ vertices 
$\{M_{n-1}+1,\dots,M_n\}$
(which we relabel $\{1,\dots,m_n\}$),
and finally $\Delta^{m_1\dots m_n}_0(\Gamma)$ is the graph obtained by collapsing 
the first $m_1$ vertices of $\Gamma$ (and all edges among them) into a single vertex (which we label $1$),
the second $m_2$ vertices of $\Gamma$ into a single vertex (which we label $2$),
and so on up to the last $m_n$ vertices of $\Gamma$ into a single vertex (which we label $n$).

For example, consider the list of integers $(3,3,1,2)$,
and the $9$-graph 
\begin{equation}\label{eq:graph-example}
\,\begin{tikzpicture}
\node at (-1,0) {$\Gamma=$};
 \draw (0,0) circle [radius=0.1];
\node at (0,-0.3) {1};
\draw (1,0) circle [radius=0.1];
\node at (1,-0.3) {2};
\draw (2,0) circle [radius=0.1];
\node at (2,-0.3) {3};
\draw (3,0) circle [radius=0.1];
\node at (3,-0.3) {4};
\draw (4,0) circle [radius=0.1];
\node at (4,-0.3) {5};
\draw (5,0) circle [radius=0.1];
\node at (5,-0.3) {6};
\draw (6,0) circle [radius=0.1];
\node at (6,-0.3) {7};
\draw (7,0) circle [radius=0.1];
\node at (7,-0.3) {8};
\draw (8,0) circle [radius=0.1];
\node at (8,-0.3) {9};
\draw[->] (0.1,0) -- (0.9,0);
\draw[->] (0,0.1) to [out=90,in=90] (2,0.1);
\draw[->] (0,0.1) to [out=90,in=90] (3,0.1);
\draw[->] (2,-0.1) to [out=270,in=270] (5,-0.1);
\draw[->] (2,-0.1) to [out=270,in=270] (8,-0.1);
\draw[<-] (3.1,0) -- (3.9,0);
\draw[<-] (7.1,0) -- (7.9,0);
\draw [dotted,thin] (1,0) circle [radius=1.3];
\draw [dotted,thin] (4,0) circle [radius=1.3];
\draw [dotted,thin] (6,0) circle [radius=0.6];
\draw [dotted,thin] (7.5,0) circle [radius=0.8];
\node at (9,0) {$\in\mc G_0(9)\,,$};
\end{tikzpicture}
\end{equation}
Then, the cocomposition 
$\Delta^{3312}(\Gamma)\in\mc G(4)\times\mc G_0(3)\times\mc G_0(3)\times\mc G_0(1)\times\mc G_0(2)$ 
consists of the following graphs:
the subgraph of $\Gamma$ associated to the first three vertices, is
$$
\,\begin{tikzpicture}
\node at (-2,0) {$\Delta^{3312}_1(\Gamma)=$};
\draw (0,0) circle [radius=0.1];
\node at (0,-0.3) {1};
\draw (1,0) circle [radius=0.1];
\node at (1,-0.3) {2};
\draw (2,0) circle [radius=0.1];
\node at (2,-0.3) {3};
\draw[->] (0.1,0) -- (0.9,0);
\draw[->] (0,0.1) to [out=90,in=90] (2,0.1);
\node at (3,0) {$\in\mc G_0(3)\,,$};
\end{tikzpicture}
$$
the subgraph $\Gamma$ associated to the second three vertices
(and relabeling the vertices), is
$$
\,\begin{tikzpicture}
\node at (-2,0) {$\Delta^{3312}_2(\Gamma)=$};
\draw (0,0) circle [radius=0.1];
\node at (0,-0.3) {1};
\draw (1,0) circle [radius=0.1];
\node at (1,-0.3) {2};
\draw (2,0) circle [radius=0.1];
\node at (2,-0.3) {3};
\draw[<-] (0.1,0) -- (0.9,0);
\node at (3,0) {$\in\mc G_0(3)\,,$};
\end{tikzpicture}
$$
the subgraph associated to the seventh vertex is just $\Delta^{3312}_3(\Gamma)=\substack{\circ \\ 1}\in\mc G_0(1)$,
the subgraph of $\Gamma$ associated to the last two vertices is
$$
\,\begin{tikzpicture}
\node at (-2,0) {$\Delta^{3312}_4(\Gamma)=$};
\draw (0,0) circle [radius=0.1];
\node at (0,-0.3) {1};
\draw (1,0) circle [radius=0.1];
\node at (1,-0.3) {2};
\draw[<-] (0.1,0) -- (0.9,0);
\node at (2,0) {$\in\mc G_0(2)\,,$};
\end{tikzpicture}
$$
and finally, collapsing all these subgraphs into single vertices, we get
$$
\begin{tikzpicture}
\node at (-2,0) {$\Delta^{3312}_0(\Gamma)=$};
\draw (0,0) circle [radius=0.1];
\node at (0,-0.3) {1};
\draw (1,0) circle [radius=0.1];
\node at (1,-0.3) {2};
\draw (2,0) circle [radius=0.1];
\node at (2,-0.3) {3};
\draw (3,0) circle [radius=0.1];
\node at (3,-0.3) {4};
\draw[->] (0,0.1) to [out=90,in=90] (1,0.1);
\draw[->] (0,-0.1) to [out=270,in=270] (1,-0.1);
\draw[->] (0,-0.1) to [out=270,in=270] (3,-0.1);
\node at (4,0) {$\in\mc G(4)\,.$};
\end{tikzpicture}
$$
Note that
if $\Gamma$ is acyclic,
then all the subgraphs $\Delta^{m_1\dots m_n}_i(\Gamma)$, for $i=1,\dots,n$, are acyclic as well,
while, in general, this is not the case for $\Delta^{m_1\dots m_n}_0(\Gamma)$.
\begin{example}\label{ex:gamma}
A special case is when $m_1=k+1$ and $m_2=\dots=m_{n}=1$.
With the notation of Section \ref{sec:comgr}, we have in this case
$\Delta^{(k+1)1\dots1}_1(\Gamma)=\Gamma'$,
$\Delta^{(k+1)1\dots1}_2(\Gamma)=\dots=\Delta^{(k+1)1\dots1}_{n}(\Gamma)=\circ$,
and
$\Delta^{(k+1)1\dots1}_0(\Gamma)=\bar\Gamma''$.
\end{example}
\begin{lemma}\label{20170823:lem1}
For every\/ $m_1,\dots,m_n$,
there is a natural bijective correspondence
\begin{equation}\label{20170823:eq5}
\Delta
\colon
E(\Gamma)
\,\stackrel{\sim}{\longrightarrow}\,
E\big(\Delta^{m_1\dots m_n}_0(\Gamma)\big)
\sqcup
E\big(\Delta^{m_1\dots m_n}_1(\Gamma)\big)
\sqcup\dots\sqcup
E\big(\Delta^{m_1\dots m_n}_n(\Gamma)\big)
\,.
\end{equation}
\end{lemma}
\begin{proof}
An edge $e\in E(\Gamma)$
has either both tail and head contained in one of the subsets $\{M_{i-1}+1,\dots,M_i\}$,
for some $i=1,\dots,n$,
in which case it corresponds to an edge of $\Delta^{m_1\dots m_n}_i(\Gamma)$,
or it does not, in which case it corresponds to an edge of $\Delta^{m_1\dots m_n}_0(\Gamma)$.
\end{proof}
It follows from Lemma \ref{20170823:lem1} that the cooperad of graphs $\mc G$ is graded by the number of edges.
\begin{lemma}\label{20170823:lem2}
Let\/ $C\subset E(\Gamma)$ be an oriented cycle of an\/ $n$-graph\/ $\Gamma\in\mc G(n)$.
Then, 
\begin{enumerate}[(a)]
\item
either\/ $\Delta(C)\subset E\big(\Delta^{m_1\dots m_n}_i(\Gamma)\big)$,
in which case\/ $\Delta(C)$ is an oriented cycle 
of\/ $\Delta^{m_1\dots m_n}_i(\Gamma)\in\mc G(m_i)$;
\item
or, $\Delta(C)\cap E(\Delta^{m_1\dots m_n}_0(\Gamma))$ 
is an oriented cycle of\/ $\Delta^{m_1\dots m_n}_0(\Gamma)\in\mc G(n)$.
\end{enumerate}
\end{lemma}
\begin{proof}
Obvious.
\end{proof}
Let, as above, $m_1,\dots,m_n$ be positive integers,
and let $\Gamma\in\mc G(M_n)$.
We now introduce an important notion, which will be essential in Section \ref{sec:6}.
\begin{definition}\label{20170824:de1}
Let $k\in\{1,\dots,M_n\}$ and $j\in\{1,\dots,n\}$.
We say that $j$ is \emph{externally connected} to $k$ (via the graph $\Gamma$
and its cocomposition $\Delta^{m_1\dots m_n}(\Gamma)$)
if there there is an unoriented path (without repeating edges) of $\Delta^{m_1\dots m_n}_0(\Gamma)$
joining $j$ to $i$, where $i\in\{1,\dots,n\}$ is such that $k\in\{M_{i-1}+1,\dots,M_i\}$,
and the edge out of $i$ is the image, via the map $\Delta$ in \eqref{20170823:eq5},
of an edge which has its head or tail in $k$.
We denote by 
$$
\mc E(k)
=
\mc E(\Gamma,m_1,\dots,m_n;k)
\,
\subset\{1,\dots,n\}
\,,
$$
the set of all $j\in\{1,\dots,n\}$
which are externally connected to $k$.
Moreover, given a set of variables $x_1,\dots,x_{n}$,
we denote
\begin{equation}\label{20170824:eq1}
X(k)
=
X(\Gamma,m_1,\dots,m_n;k)
=
\sum_{j\in\mc E(k)}x_j
\,.
\end{equation}
\end{definition}
For example, for the graph in \eqref{eq:graph-example},
we have
$$
\begin{array}{l}
\displaystyle{
\vphantom{\Big(}
X(1)=x_1+x_2+x_4
\,,\,\,
X(2)=0
\,,\,\,
X(3)=x_1+x_2+x_4
\,,\,\,
X(4)=x_1+x_2+x_4
\,,} \\
\displaystyle{
\vphantom{\Big(}
X(5)=0
\,,\,\,
X(6)=x_1+x_2+x_4
\,,\,\,
X(7)=0
\,,\,\,
X(8)=0
\,,\,\,
X(9)=x_1+x_2
\,.}
\end{array}
$$
Note that, if $k\in\{M_{i-1}+1,\dots,M_i\}$,
then $i\not\in\mc E(k)$ unless $\Delta^{m_1\dots m_n}_0(\Gamma)$ is not acyclic.
\begin{example}\label{ex:gamma2}
In the setting of Example \ref{ex:gamma},
let $m_1=k+1$ and $m_2=\dots=m_n=1$.
Assuming that $\bar\Gamma''$ is acyclic,
for every $\ell=0,\dots,k$, the set $\mc E(\ell)$ coincides with the set of vertices of the graph $G_\ell$
defined in Section \ref{sec:comgr}.
\end{example}

\subsection{Coassociativity of the cocomposition map of $n$-graphs}\label{sec:6a.22}

The collection of sets $\mc G(n)$, $n\geq0$, together with the cocomposition maps \eqref{20170614:eq2},
defines a \emph{cooperad} \cite{LV12},
or, equivalently, the dual $\mc G^*$ is naturally an operad.

We will not give a formal definition of what a cooperad is (since we will never use it),
but we will prove here the main conditions: coassociativity, in Proposition \ref{20170615:prop1} below,
and coequivariance with respect to the action of the symmetric group,
in the next Section \ref{sec:6a.3}, see Proposition \ref{20170615:prop2}.

Fix a list 
$m_1,\dots,m_n$ of $n$ positive integers,
denote $M_i=\sum_{j=1}^im_j$, $i=0,\dots,n$, as in \eqref{20170821:eq2a},
then fix a list $\ell_1,\dots,\ell_{M_n}$ of $M_n$ positive integers,
and denote $L_j=\sum_{k=1}^j\ell_k$, $j=0,\dots,M_n$, as in \eqref{20170821:eq2b}.
Given a graph $\Gamma\in\mc G_0(L_{M_n})$,
we can apply to it the cocomposition $\Delta^{\ell_1\dots\ell_{M_n}}$, to get
$$
\begin{array}{l}
\displaystyle{
\vphantom{\Big(}
\Delta_0^{\ell_1\dots\ell_{M_n}}(\Gamma)\in\mc G(M_n)
\,,} \\
\displaystyle{
\vphantom{\Big(}
\Delta_1^{\ell_1\dots\ell_{M_n}}(\Gamma)\in{\mc G}(\ell_1)
\,,\,\dots\,,
\Delta_{M_n}^{\ell_1\dots\ell_{M_n}}(\Gamma)\in{\mc G}(\ell_{M_n})
\,,}
\end{array}
$$
and, to the first graph above, we can further apply the cocomposition map $\Delta^{m_1\dots m_n}$
\eqref{20170614:eq2}, to get
$$
\begin{array}{l}
\displaystyle{
\vphantom{\Big(}
\Delta^{m_1\dots m_n}_0\big(\Delta_0^{\ell_1\dots\ell_{M_n}}(\Gamma)\big)\in\mc G(n)
\,,} \\
\displaystyle{
\vphantom{\Big(}
\Delta^{m_1\dots m_n}_1\big(\Delta_0^{\ell_1\dots\ell_{M_n}}(\Gamma)\big)\in\mc G(m_1)
\,,\dots,
\Delta^{m_1\dots m_n}_n\big(\Delta_0^{\ell_1\dots\ell_{M_n}}(\Gamma)\big)\in\mc G(m_n)
\,.}
\end{array}
$$
Alternatively, we can consider the $n$ integers (summing to $L_{M_n}$)
$$
\begin{array}{l}
\displaystyle{
\vphantom{\big(}
K_1:=L_{M_1}=\sum_{j=1}^{M_1}\ell_j
\,,\,\,
K_2:=L_{M_2}-L_{M_1}=\sum_{j=M_1+1}^{M_2}\ell_j
\,,\,\dots} \\
\displaystyle{
\vphantom{\big(}
\dots\,,\,
K_n:=L_{M_n}-L_{M_{n-1}}=\sum_{j=M_{n-1}+1}^{M_n}\ell_j
\,,}
\end{array}
$$
we can apply the corresponding cocomposition map 
$\Delta^{K_1\dots K_n}$ to $\Gamma$, to get
$$
\begin{array}{l}
\displaystyle{
\vphantom{\Big(}
\Delta^{K_1\dots K_n}_0(\Gamma)\in\mc G(n)
\,,} \\
\displaystyle{
\vphantom{\Big(}
\Delta^{K_1\dots K_n}_1(\Gamma)\in{\mc G}(K_1)
\,,\dots,
\Delta^{K_1\dots K_n}_n(\Gamma)\in{\mc G}(K_n)
\,,}
\end{array}
$$
and, to each of the graph in the second line,
we can apply the corresponding cocomposition map 
$\Delta^{\ell_{M_{i-1}+1}\dots\ell_{M_i}}$, $i=1,\dots,n$
to get
$$
\begin{array}{l}
\displaystyle{
\vphantom{\Big(}
\Delta_0^{\ell_{M_{i-1}+1}\dots\ell_{M_i}}
\big(\Delta^{K_1\dots K_n}_i(\Gamma)\big)\in\mc G(m_i)
\,,} \\
\displaystyle{
\vphantom{\Big(}
\Delta_1^{\ell_{M_{\!i\!-\!1}\!\!+1}\dots\ell_{M_i}}\!
\big(\Delta^{K_1\dots K_n}_i(\Gamma)\big)\in{\mc G}(\ell_{M_{\!i\!-\!1}\!\!+1})
\,,\dots,
\Delta_{m_i}^{\ell_{M_{\!i\!-\!1}\!\!+1}\dots\ell_{M_i}}\!
\big(\Delta^{K_1\dots K_n}_i(\Gamma)\big)\in{\mc G}(\ell_{M_i})
.}
\end{array}
$$
\begin{proposition}\label{20170615:prop1}
The cocomposition maps \eqref{20170614:eq2} of graphs
satisfy the following coassociativity conditions:
\begin{enumerate}[(i)]
\item
$\Delta^{m_1\dots m_n}_0\big(\Delta_0^{\ell_1\dots\ell_{M_n}}(\Gamma)\big)
=\Delta^{K_1\dots K_n}_0(\Gamma)$
in\/ $\mc G(n)$;
\item
$\Delta^{m_1\dots m_n}_i\big(\Delta_0^{\ell_1\dots\ell_{M_n}}(\Gamma)\big)
=
\Delta_0^{\ell_{M_{i-1}+1}\dots\ell_{M_i}}\big(\Delta^{K_1\dots K_n}_i(\Gamma)\big)$
in\/ $\mc G(m_i)$,
for every\/ $i=1,\dots,n$;
\item
$\Delta_{M_{i-1}+j}^{\ell_1\dots\ell_{M_n}}(\Gamma)
=
\Delta_j^{\ell_{M_{i-1}+1}\dots\ell_{M_i}}\big(\Delta^{K_1\dots K_n}_i(\Gamma)\big)$
in\/ ${\mc G}(\ell_{ij})$,
for every\/ $i=1,\dots,n$ and\/ $j=1,\dots,m_i$.
\end{enumerate}
\end{proposition}
\begin{proof}
All claims become obvious if they are explained ``pictorially''.
Consider an arbitrary graph, which we can depict as follows:
$$
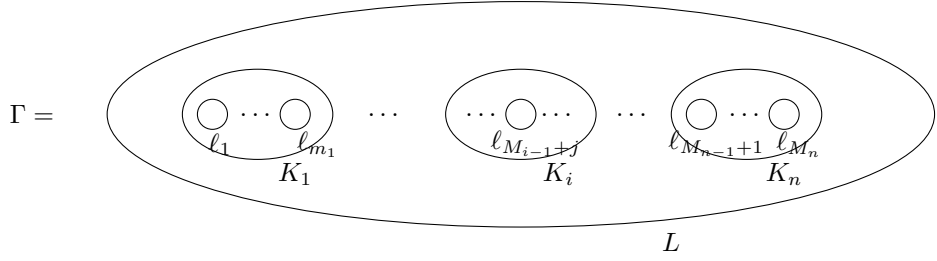
\begin{tikzpicture}
\node at (-6.5,0) {$\Gamma=$};
\draw (0,0) ellipse (5.5 and 1.5);
\draw (-3.5,0) ellipse (1 and 0.6);
\node at (-1.8,0) {$\dots$};
\draw (0,0) ellipse (1 and 0.6);
\node at (1.5,0) {$\dots$};
\draw (3,0) ellipse (1 and 0.6);
\draw (-4.1,0) circle (0.2);
\node at (-3.5,0) {$\dots$};
\draw (-3,0) circle (0.2);
\node at (-0.5,0) {$\dots$};
\draw (0,0) circle (0.2);
\node at (0.5,0) {$\dots$};
\draw (2.4,0) circle (0.2);
\node at (3,0) {$\dots$};
\draw (3.5,0) circle (0.2);
\node at (2,-1.7) {$L$};
\node at (-3,-0.8) {$K_1$};
\node at (0.5,-0.8) {$K_i$};
\node at (3.5,-0.8) {$K_n$};
\node at (-4,-0.4) {$\ell_{1}$};
\node at (-2.7,-0.4) {$\ell_{m_1}$};
\node at (0.2,-0.4) {$\ell_{M_{i-1}+j}$};
\node at (2.6,-0.4) {$\ell_{M_{n-1}+1}$};
\node at (3.7,-0.4) {$\ell_{M_n}$};
\end{tikzpicture}
$$
where each the intermidiate ovals surround subgraphs of $K_1,\dots,K_n$ vertices respectively,
and, inside the $i$-th oval, the inner circles surround subgraphs 
of $\ell_{M_{i-1}+1}$,$\dots,$ $\ell_{M_i}$ vertices respectively.

In condition (i), the graph $\Delta^{K_1\dots K_n}_0(\Gamma)$ in the right-hand side
is obtained starting from $\Gamma$ and collapsing all intermidiate subgraphs (= intermidiate ovals)
to single vertices.
On the other hand, 
the graph $\Delta^{m_1\dots m_n}_0\big(\Delta_0^{\ell_1\dots\ell_{M_n}}(\Gamma)\big)$ in the left-hand side
is obtained by first collapsing all the inner subgraphs (= inner circles) to single vertices and then,
in the resulting graph, by further collapsing the intermidiate subgraphs (= intermidiate ovals) to single vertices.
The result is obviously the same.

In condition (ii),
the graph $\Delta^{m_1\dots m_n}_i\big(\Delta_0^{\ell_1\dots\ell_{M_n}}(\Gamma)\big)$ in the left-hand side
is obtained starting from $\Gamma$ by collapsing all inner subgraphs (= inner circles)
to single vertices,
and then, in the resulting graph, by taking the $i$-th intermidiate subgraph.
On the other hand,
the graph $\Delta_0^{\ell_{M_{i-1}+1}\dots\ell_{M_i}}\big(\Delta^{K_1\dots K_n}_i(\Gamma)\big)$ in the right-hand side
is obtained by first taking the $i$-th intermidiate subgraph (= intermidiate oval) of $\Gamma$,
and then, inside it, by collapsing all inner subgraphs (=inner circles) to single vertices.
The result is obviously the same.

Finally, in condition (iii),
the graph $\Delta_{M_{i-1}+j}^{\ell_1\dots\ell_{M_n}}(\Gamma)$ in the left-hand side
is obtained by looking at the $M_{i-1}+j$-th inner subgraph (= inner circle) 
of $\Gamma$, which is the $j$-th circle inside the $i$-th intermidiated oval,
while the graph 
$\Delta_j^{\ell_{M_{i-1}+1}\dots\ell_{M_i}}\big(\Delta^{K_1\dots K_n}_i(\Gamma)\big)$ in the right-hand side
is obtained by first taking the $i$-th intermidiate subgraph (= intermidiate oval) of $\Gamma$,
and then, inside it, by taking the $j$-th inner subgraph (= inner circle).
The result is obviously the same.
\end{proof}

\subsection{Coequivariance of the cocomposition map of $n$-graphs}\label{sec:6a.3}

For every $n\geq1$, there is a natural (left) action of the symmetic group $S_n$
on the set $\mc G_0(n)$ of acyclic $n$-graphs,
and on the set $\mc G(n)$ of all $n$-graphs.
It is defined as follows:
given the $n$-graph $\Gamma$ and the permutation $\sigma\in S_n$,
we define $\sigma(\Gamma)$ to be the same graph as $\Gamma$,
but with the vertex which was labelled $1$ relabelled as $\sigma(1)$,
and so on up to the vertex which was labelled $n$,
which is relabelled as $\sigma(n)$.
For example, 
if $\Gamma\in\mc G_0(4)$ is the following $4$-graph:
$$
\begin{tikzpicture}
\node at (-1,0) {$\Gamma=$};
\draw (0,0) circle [radius=0.1];
\node at (0,-0.3) {1};
\draw (1,0) circle [radius=0.1];
\node at (1,-0.3) {2};
\draw (2,0) circle [radius=0.1];
\node at (2,-0.3) {3};
\draw (3,0) circle [radius=0.1];
\node at (3,-0.3) {4};
\draw[<-] (0.1,0) -- (0.9,0);
\draw[->] (0,0.1) to [out=90,in=90] (2,0.1);
\draw[<-] (2.1,0) -- (2.9,0);
\node at (4,0) {$\in{\mc G}(4)\,.$};
\end{tikzpicture}
$$
and $\sigma\in S_4$ is the permutation $\sigma=(1,3,4)$ 
(in the standard cycle decomposition),
then
$$
\begin{tikzpicture}
\node at (-1,0) {$\sigma(\Gamma)=$};
\draw (0,0) circle [radius=0.1];
\node at (0,-0.3) {3};
\draw (1,0) circle [radius=0.1];
\node at (1,-0.3) {2};
\draw (2,0) circle [radius=0.1];
\node at (2,-0.3) {4};
\draw (3,0) circle [radius=0.1];
\node at (3,-0.3) {1};
\draw[<-] (0.1,0) -- (0.9,0);
\draw[->] (0,0.1) to [out=90,in=90] (2,0.1);
\draw[<-] (2.1,0) -- (2.9,0);
\node at (4,0) {$=$};
\draw (5,0) circle [radius=0.1];
\node at (5,-0.3) {1};
\draw (6,0) circle [radius=0.1];
\node at (6,-0.3) {2};
\draw (7,0) circle [radius=0.1];
\node at (7,-0.3) {3};
\draw (8,0) circle [radius=0.1];
\node at (8,-0.3) {4};
\draw[->] (6.1,0) -- (6.9,0);
\draw[->] (5,0.1) to [out=90,in=90] (8,0.1);
\draw[->] (7.1,0) -- (7.9,0);
\end{tikzpicture}
$$
\begin{proposition}\label{20170615:prop2}
For every positive integers\/ $n,m_1,\dots,m_n$,
every permutations\/ $\sigma\in S_n$,
$\tau_1\in S_{m_1}$, $\dots$, $\tau_n\in S_{m_n}$,
and every graph\/ $\Gamma\in\mc G_0(m_1+\dots+m_n)$,
we have
\begin{equation}\label{20170615:eq4}
\begin{array}{l}
\displaystyle{
\vphantom{\Big(}
\Delta^{m_{\sigma^{-1}(1)}\dots m_{\sigma^{-1}(n)}}
\big(
(\sigma(\tau_1,\dots,\tau_n))(\Gamma)
\big)
} \\
\displaystyle{
\vphantom{\Big(}
=
\Big(
\sigma\big(\Delta_0^{m_1\dots m_n}(\Gamma)\big)
,
\tau_{\sigma^{-1}(1)}\big(\Delta_{\sigma^{-1}(1)}^{m_1\dots m_n}(\Gamma)\big)
,\dots,
\tau_{\sigma^{-1}(n)}\big(\Delta_{\sigma^{-1}(n)}^{m_1\dots m_n}(\Gamma)\big)
\Big)
\,,}
\end{array}
\end{equation}
where the composition of permutations\/ $\sigma(\tau_1,\dots,\tau_n)$
is defined by \eqref{eq:operad19}.
\end{proposition}
\begin{proof}
Also for this proposition we provide a ``pictorial'' proof.
Consider an arbitrary acyclic $m_1\!+\!\dots\!+\!m_n$-graph $\Gamma$,
which we depict as:
\begin{equation}\label{20170616:eq1}
\begin{tikzpicture}
\node at (-6.5,0) {$\Gamma=$};
\draw (0,0) ellipse (5.5 and 1.5);
\draw (-3.5,0) ellipse (1 and 0.6);
\node at (-1.8,0) {$\dots$};
\draw (0,0) ellipse (1 and 0.6);
\node at (1.5,0) {$\dots$};
\draw (3,0) ellipse (1 and 0.6);
\draw [fill] (-4.1,0) circle (0.1);
\node at (-3.5,0) {$\dots$};
\draw [fill] (-3,0) circle (0.1);
\node at (-0.5,0) {$\dots$};
\draw [fill] (0,0) circle (0.1);
\node at (0.5,0) {$\dots$};
%
\node at (3,0) {$\dots$};
\draw [fill] (3.5,0) circle (0.1);
\node at (-4.1,-0.4) {1};
\node at (-2.8,-0.4) {$m_1$};
\draw [->] (-1,-2) -- (0,-0.2);
\draw [->] (4,-2) -- (3.5,-0.2);
\node at (-1,-2.5) {$m_1+\dots+m_{i-1}+j$};
\node at (4,-2.5) {$m_1+\dots+m_n$};
\end{tikzpicture}
\end{equation}
where we represented only the vertices (not the edges), labelled from $1$ to $m_1+\dots+m_n$,
grouped (by the inner ovals) in groups of $m_1,\dots,m_n$ vertices.
Hence, as indicated, the vertex in the $i$-th oval ($i=1,\dots,n$), in the $j$-th position within that oval
($j=1,\dots,m_i$) is labelled $m_1+\dots+m_{i-1}+j$.

When we apply the permutation $\sigma(\tau_1,\dots,\tau_n)\in S_{m_1+\dots+m_n}$
to the graph $\Gamma$, we get, by the way the symmetric group acts on $\mc G_0(m_1+\dots,+m_n)$,
the exact same picture,
but with the vertices labelled according to the action of the permutation $\sigma(\tau_1,\dots,\tau_n)$,
given by formula \eqref{eq:operad17}.
Hence, we have
$(\sigma(\tau_1,\dots,\tau_n))(\Gamma)=$
$$
\begin{tikzpicture}
%
\draw (0,0) ellipse (5.5 and 1.5);
\draw (-3.5,0) ellipse (1 and 0.6);
\node at (-1.8,0) {$\dots$};
\draw (0,0) ellipse (1 and 0.6);
\node at (1.5,0) {$\dots$};
\draw (3,0) ellipse (1 and 0.6);
\draw [fill] (-4.1,0) circle (0.1);
\node at (-3.5,0) {$\dots$};
\draw [fill] (-3,0) circle (0.1);
\node at (-0.5,0) {$\dots$};
\draw [fill] (0,0) circle (0.1);
\node at (0.5,0) {$\dots$};
\node at (3,0) {$\dots$};
\draw [fill] (3.5,0) circle (0.1);
\draw [->] (0,-2) -- (0,-0.2);
\node at (-1.6,-2) {vertex labeled:};
\node at (0,-2.5) {$m_{\sigma^{-1}(1)}+\dots+m_{\sigma^{-1}(\sigma(i)-1)}+\tau_i(j)$};
\end{tikzpicture}
$$
Then, to get the picture of $(\sigma(\tau_1,\dots,\tau_n))(\Gamma)$,
with the vertices in the correct order,
we should rearrange the vertices of picture \eqref{20170616:eq2}
by moving the vertex labelled by $1$ 
(which, in the picture \eqref{20170616:eq2}, 
is in the $i={\sigma^{-1}(1)}$-th oval, in $\tau_{\sigma^{-1}(1)}^{-1}(1)$-th position)
in first position, the vertex labelled $2$ in second position, and so on.
Hence, in this rearrangement, the $i$-th oval of picture \eqref{20170616:eq1}
will be moved to position $\sigma(i)$,
and, within that oval, the $j$-th vertex will be moved to position $\tau_i(j)$.

Note that, while, in picture \eqref{20170616:eq1}
the ovals contain, in the order they are depicted, $m_1,\dots,m_n$ vertices respectively,
in the rearranged graph $(\sigma(\tau_1,\dots,\tau_n))(\Gamma)$,
where the vertex labelled $1$ come first, the vertex labelled $2$ comes second, and so on,
the vertices will be grouped in ovals containing $m_{\sigma^{-1}(1)},\dots,m_{\sigma^{-1}(n)}$ vertices
respectively.
Hence, we should apply the cocomposition map 
$\Delta^{m_{\sigma^{-1}(1)} \dots m_{\sigma^{-1}(n)}}$
to it.

According to the definition, the graph 
$$
\Delta^{m_{\sigma^{-1}(1)} \dots m_{\sigma^{-1}(n)}}_0\big((\sigma(\tau_1,\dots,\tau_n))(\Gamma)\big)
$$
is obtained by collapsing all the ovals in picture \eqref{20170616:eq2} to single vertices:
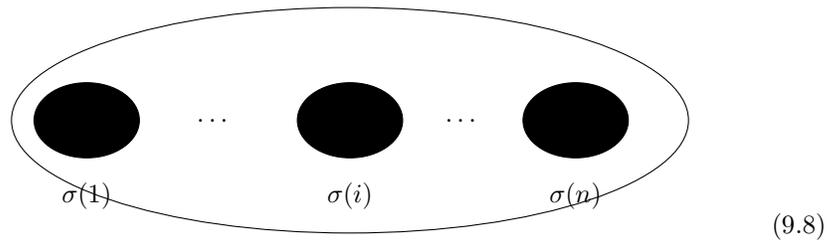
\begin{equation}\label{20170616:eq2}
\begin{tikzpicture}
%
\draw (0,0) ellipse (4.5 and 1.5);
\draw [fill] (-3.5,0) ellipse (0.7 and 0.5);
\node at (-1.8,0) {$\dots$};
\draw [fill] (0,0) ellipse (0.7 and 0.5);
\node at (1.5,0) {$\dots$};
\draw [fill] (3,0) ellipse (0.7 and 0.5);
\node at (-3.5,-1) {$\sigma(1)$};
\node at (0,-1) {$\sigma(i)$};
\node at (3,-1) {$\sigma(n)$};
\end{tikzpicture}
\end{equation}
Obviously, this is the same graph as
$$
\sigma(\Delta^{m_1\dots m_n}(\Gamma))
\,,
$$
where we first collaps all the inner ovals of $\Gamma$ in picture \eqref{20170616:eq1}
to single vertices,
and then we apply the permutation $\sigma\in S_n$, i.e. we relabel the vertices
according to $\sigma$.

Next, according to the definition, the graph 
$$
\Delta^{m_{\sigma^{-1}(1)} \dots m_{\sigma^{-1}(n)}}_{\sigma(i)}\big((\sigma(\tau_1,\dots,\tau_n))(\Gamma)\big)
$$
is the subgraph corresponding to the $\sigma(i)$-th oval of the graph
$(\sigma(\tau_1,\dots,\tau_n))(\Gamma)$ (rearranged),
i.e. the $i$-th oval of picture \eqref{20170616:eq2}:
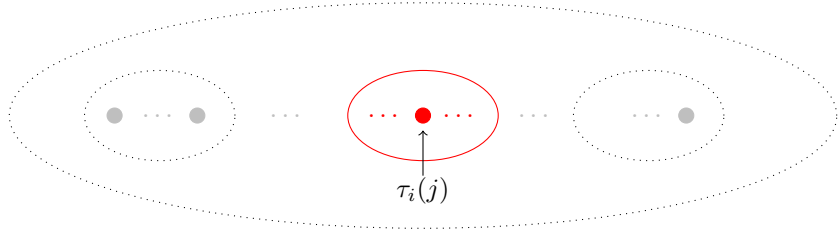
$$
\begin{tikzpicture}
%
\draw [dotted] (0,0) ellipse (5.5 and 1.5);
\draw [dotted] (-3.5,0) ellipse (1 and 0.6);
\node [color=lightgray] at (-1.8,0) {$\dots$};
\draw [color=red] (0,0) ellipse (1 and 0.6);
\node [color=lightgray] at (1.5,0) {$\dots$};
\draw [dotted] (3,0) ellipse (1 and 0.6);
\draw [fill,color=lightgray] (-4.1,0) circle (0.1);
\node [color=lightgray] at (-3.5,0) {$\dots$};
\draw [fill,color=lightgray] (-3,0) circle (0.1);
\node [color=red] at (-0.5,0) {$\dots$};
\draw [fill,color=red] (0,0) circle (0.1);
\node [color=red] at (0.5,0) {$\dots$};
\node [color=lightgray] at (3,0) {$\dots$};
\draw [fill,color=lightgray] (3.5,0) circle (0.1);
\draw [->] (0,-0.8) -- (0,-0.2);
\node at (0,-1) {$\tau_i(j)$};
\end{tikzpicture}
$$
Obviously, this is the same as the graph
$$
\tau_i(\Delta^{m_1\dots m_n}_i(\Gamma))
\,,
$$
where we first take the subgraph of $\Gamma$ corresponding to the $i$-th oval 
of picture \eqref{20170616:eq1},
and then we apply the permutation $\tau_i\in S_{m_i}$, i.e. we relabel the vertices
according to $\tau_i$.
\end{proof}

\section{The operad governing Poisson vertex superalgebras}\label{sec:6}

\subsection{Definition of a Poisson vertex superalgebra}\label{sec:6.0}

Recall that a \emph{Poisson vertex superalgebra} (abbreviated PVA) is a commutative associative superalgebra $V$
endowed with an even derivation $\partial$
and a Lie conformal superalgebra $\lambda$-bracket $\{\cdot\,_\lambda\,\cdot\}$
satisfying the left Leibniz rule:
\begin{equation}\label{20170614:eq1}
\{a_\lambda bc\}=\{a_\lambda b\}c+(-1)^{p(b)p(c)}\{a_\lambda c\}b
\,.
\end{equation}

\subsection{Definition of the operad $P^\cl$}\label{sec:6.2}

Let $V=V_{\bar 0}\oplus V_{\bar 1}$ be a vector superspace endowed
with an even endomorphism $\partial\in\End V$.
The operad $P^\cl$ is 
the collection of superspaces $P^\cl(n)$ defined as follows.
As a vector superspace, 
$P^\cl(n)$ is the space of all maps 
\begin{equation}\label{20170614:eq4}
f\colon
\mc G(n)\times V^{\otimes n}
\longrightarrow
V[\lambda_1,\dots,\lambda_n] / \langle \dd+\lambda_1+\dots+\lambda_n \rangle
\,,
\end{equation}
which are linear in the second factor,
mapping the $n$-graph $\Gamma\in\mc G(n)$ 
and the monomial $v_1\otimes\,\cdots\,\otimes v_n\in V^{\otimes n}$
to the polynomial
\begin{equation}\label{20170614:eq5}
f^{\Gamma}_{\lambda_1,\dots,\lambda_n}(v_1\otimes\,\cdots\,\otimes v_n)
\,,
\end{equation}
satisfying 
the \emph{cycle relations}
and the \emph{sesquilinearity conditions} described below.
The cycle relations state that
\begin{equation}\label{eq:cycle1}
f^{\Gamma}=0 
\,\,\text{ unless }\,\,
\Gamma\in\mc G_0(n)
\,,
\end{equation}
and if $C\subset E(\Gamma)$ 
is an oriented cycle of $\Gamma$,
then
\begin{equation}\label{eq:cycle2}
\sum_{e\in C}f^{\Gamma\backslash e}
=0
\,,
\end{equation}
where $\Gamma\backslash e$ is the graph obtained from $\Gamma$ by removing the edge $e$.
Note that these are the same relations as in Lemma \ref{lem:gr}.
Condition \eqref{eq:cycle2} follows from \eqref{eq:cycle1}
unless $\Gamma$ contains a unique oriented cycle.
In the special case of oriented cycles of length $2$,
the cycle relation \eqref{eq:cycle2}
means that changing orientation of a single edge of the $n$-graph $\Gamma\in\mc G(n)$
amounts to a change of sign of $f^{\Gamma}$.

To write the \emph{sesquilinearity conditions},
let $\Gamma=\Gamma_1\sqcup\dots\sqcup\Gamma_s$
be the decomposition of $\Gamma$ as disjoint union of its connected components,
and let $I_1,\dots,I_s\subset\{1,\dots,n\}$ be the sets of vertices associated to these
connected components.
For example, for the graph $\Gamma$ in \eqref{eq:graph-example},
we have $\Gamma=\Gamma_1\sqcup\Gamma_2$,
with $I_1=\{1,2,3,4,5,6,8,9\}$ and $I_2=\{7\}$.
Then for every $\alpha=1,\dots,s$, we have two sesquilinearity conditions.
The first one states
\begin{equation}\label{eq:sesq1}
\frac{\partial}{\partial \lambda_i}
f^{\Gamma}_{\lambda_1,\dots,\lambda_n}(v_1\otimes\dots\otimes v_n)
\,\text{ is the same for all }\,
i\in I_\alpha
\,.
\end{equation}
In other words, the polynomial 
$f^{\Gamma}_{\lambda_1,\dots,\lambda_n}(v_1\otimes\dots\otimes v_n)$
is a function of the variables 
$\lambda_{\Gamma_\alpha}=\sum_{i\in I_\alpha}\lambda_i$, $\alpha=1,\dots,s$
(cf.\ \eqref{com3}),
and not of the variables $\lambda_1,\dots,\lambda_n$ separately.
The second sesquilinearity condition is,
again in the notation \eqref{com3},
\begin{equation}\label{eq:sesq2}
f^{\Gamma}_{\lambda_1,\dots,\lambda_n}
(\partial_{\Gamma_\alpha}(v_1\otimes\dots\otimes v_n))
=-\lambda_{\Gamma_\alpha}
f^{\Gamma}_{\lambda_1,\dots,\lambda_n}(v_1\otimes\dots\otimes v_n)
\,.
\end{equation}
These are the same relations as in Lemma \ref{lem:gr2}.

\begin{remark}\label{ref:trcov2}
Since $\Ga$ is a disjoint union of its of its connected components $\Ga_\al$,
the second sesquilinearity condition \eqref{eq:sesq2} implies
\begin{equation}\label{trcov2}
f^{\Gamma}_{\lambda_1,\dots,\lambda_n}(\dd_\Ga v) = 
-\sum_{i=1}^n \la_i \, f^{\Gamma}_{\lambda_1,\dots,\lambda_n}(v)
= \dd \bigl( f^{\Gamma}_{\lambda_1,\dots,\lambda_n}(v) \bigr), 
\qquad v\in V^{\otimes n}
\end{equation}
(cf.\ Remark \ref{ref:trcov}).
\end{remark}

The space $P^\cl(n)$ decomposes as a direct sum
\begin{equation}\label{eq:pclgr}
P^\cl(n) = \bigoplus_{r\ge0} \gr^r P^\cl(n),
\end{equation}
where $\gr^r P^\cl(n)$ is the subspace of all maps \eqref{20170614:eq4} vanishing 
on graphs $\Ga$ with number of edges not equal to $r$.
\begin{remark}\label{rem:last}
Let $V=\bigoplus_r\gr^r V$ be a graded vector space,
and consider the induced grading of the tensor powers $V^{\otimes k}$.
Then the classical operad $P^\cl(V)$
has a refined grading defined as follows:
$f\in\gr^rP^\cl(k)(V)$ if, for every graph $\Gamma\in\mc G(k)$ with $s$ edges,
we have
$$
f^\Gamma_{\lambda_1,\dots,\lambda_k}(\gr^t V^{\otimes k})
\,\subset\,
(\gr^{s+t-r}V)[\lambda_1,\dots,\lambda_k]/
\langle\partial+\lambda_1+\dots+\lambda_k\rangle
\,.
$$
The grading \eqref{eq:pclgr} corresponds to the special case when $V=\gr^0V$.
\end{remark}

The $\mb Z/2\mb Z$-grading of the superspace $P^\cl(n)$ is induced 
by that of the vector superspace $V$
(as before, the variables $\lambda_i$ are even and commute).
We also have a natural right action of the symmetric group $S_n$ on $P^\cl(n)$
by (parity preserving) linear maps,
defined by the following formula
($f\in P^\cl(n)$, $\Gamma\in\mc G(n)$, $v_1,\dots,v_n\in V$):
\begin{equation}\label{20170615:eq1}
(f^\sigma)^\Gamma_{\lambda_1,\dots,\lambda_n}(v_1\otimes\,\dots\,\otimes v_n)
=
f^{\sigma(\Gamma)}_{\sigma(\lambda_1,\dots,\lambda_n)}
(\sigma(v_1\otimes\cdots\otimes v_n))
\,,
\end{equation}
where 
$\sigma(\lambda_1,\dots,\lambda_n)$ is defined by \eqref{20170615:eq2},
$\sigma(v_1\otimes\dots\otimes v_n)$ is defined by \eqref{eq:operad6},
and $\sigma(\Gamma)$ is defined in Section \ref{sec:6a.3}.

Next, we define the composition maps of the operad $P^\cl$.
Let $f\in P^\cl(n)$ and $g_1\in P^\cl(m_1),\dots,g_n\in P^\cl(m_n)$.
Let $M_i$, $i=0,\dots,n$, and $\Lambda_i$, $i=1,\dots,n$,
be as in \eqref{20170821:eq2}.
Let $\Gamma\in\mc G(M_n)$
and consider its cocomposition 
$\Delta^{m_1\dots m_n}(\Gamma)$ defined in Section \ref{sec:6a.2}.
We let
$$
(f(g_1,\dots,g_n))^\Gamma:\, V^{\otimes M_n}
\to
V[\lambda_1,\dots,\lambda_{M_n}] / \langle \dd+\lambda_1+\dots+\lambda_{M_n} \rangle
$$
be defined by the following formula:
\begin{equation}\label{20170616:eq3}
\begin{array}{l}
\displaystyle{
\vphantom{\Big(}
(f(g_1,\dots,g_n))^\Gamma_{\lambda_1,\dots,\lambda_{M_n}}
(v_1\otimes\dots\otimes v_{M_n})
} \\
\displaystyle{
\vphantom{\Big(}
=
f^{\Delta^{m_1\dots m_n}_0(\Gamma)}_{\Lambda_1,\dots,\Lambda_n}
\bigg(
\Big(
\Big(
\Big|_{x_1=\Lambda_1+\partial}
(g_1)^{\Delta^{m_1\dots m_n}_1(\Gamma)}_{\lambda_{1}+X(1),\dots,\lambda_{M_1}+X(M_1)}
\Big)
\otimes\cdots
} \\
\displaystyle{
\vphantom{\Big(}
\,\,\,\,\,\,\,\,\,
\dots\otimes
\Big(
\Big|_{x_n=\Lambda_n+\partial}
(g_n)^{\Delta^{m_1\dots m_n}_n(\Gamma)}_{
\lambda_{M_{n-1}+1}+X(M_{n-1}+1),\dots,\lambda_{M_n}+X(M_n)}
\Big)
\Big)
(v_1\otimes\dots\otimes v_{M_n})
\bigg)
\,.}
\end{array}
\end{equation}
In formula \eqref{20170616:eq3} 
we are using the following notation.
Given the graphs $\Gamma_1\in\mc G(m_1),\,\dots,\Gamma_n\in\mc G(m_n)$,
we let, recalling \eqref{20170821:eq5},
\begin{equation}\label{20170821:eq5c}
\begin{array}{l}
\displaystyle{
\vphantom{\Big(}
\big(
(g_1)^{\Gamma_1}_{\lambda_{1},\dots,\lambda_{M_1}}
\otimes\dots\otimes
(g_n)^{\Gamma_n}_{\lambda_{M_{n-1}+1},\dots,\lambda_{M_n}}
\big)
(v_1\otimes\dots\otimes v_{M_n})
} \\
\displaystyle{
\vphantom{\Big(}
:=\pm\,
(g_1)^{\Gamma_1}_{\lambda_{1},\dots,\lambda_{M_1}}\!\!(v_1\otimes\dots\otimes v_{M_1})
\otimes\dots\otimes
(g_n)^{\Gamma_n}_{\lambda_{M_{\!n\!-\!1}\!+1},\dots,\lambda_{M_n}}\!\!(v_{M_{\!n\!-\!1}\!+1}\otimes\dots\otimes v_{M_n})
,}
\end{array}
\end{equation}
with $\pm$ the same as \eqref{20170821:eq5b}.
We are using the notation \eqref{20170824:eq1}
for the variables $X(1),\dots,X(M_n)$ appearing in \eqref{20170821:eq5c}.
%
Finally,
for polynomials $P(\lambda)=\sum_mp_m\lambda^m$ and 
$Q(\mu)=\sum_nq_n\mu^n$ with coefficients in $V$,
we denote
\begin{equation}\label{20170824:eq2}
\big(\big|_{x=\partial}P(\lambda+y)\big)
\otimes
\big(\big|_{y=\partial}Q(\mu+x)\big)
=
\sum_{m,n}((\mu+\partial)^np_m)\otimes((\lambda+\partial)^mq_n)
\,,
\end{equation}
and by $\partial g_\lambda(w_1\otimes\dots\otimes w_m)$ 
we mean $\partial (g_\lambda(w_1\otimes\dots\otimes w_m))$.
\begin{remark}\label{rem:gr}
In view of Examples \ref{ex:gamma} and \ref{ex:gamma2},
in the special case $m_1=k+1, m_2=\dots=m_n=1$ and letting $n=k+m+1$,
formula \eqref{20170616:eq3} reduces to \eqref{com4}.
\end{remark}
\begin{lemma}\label{lem:pcl1}
With the above notation, the right-hand side of \eqref{20170616:eq3} is a well-defined element
of\/ $V[\lambda_1,\dots,\lambda_{M_n}] / \langle \dd+\lambda_1+\dots+\lambda_{M_n} \rangle$,
for every\/ $f\in P^\cl(n)$ and\/ $g_1\in P^\cl(m_1),\dots,g_n\in P^\cl(m_n)$.
\end{lemma}
\begin{proof}
First observe that, if $\Delta_0^{m_1\dots m_n}(\Gamma)$
is not acyclic, then the right-hand side of \eqref{20170616:eq3} is $0$,
since by assumption $f$ satisfies \eqref{eq:cycle1}.
On the other hand, if $\Delta_0^{m_1\dots m_n}(\Gamma)$ is acyclic,
then by the observation at the end of Section \ref{sec:6a.2},
the variable $x_i$ does not appear in $X(k)$ when $k\in\{M_{i-1}+1,\dots,M_i\}$.
This makes the right-hand side of \eqref{20170616:eq3} a well-defined polynomial for given 
polynomials 
\begin{equation}\label{welldef1}
(g_i)^{\Gamma_i}_{\lambda_{M_{\!i\!-\!1}\!+1},\dots,\lambda_{M_i}}\!\!(v_{M_{\!i\!-\!1}\!+1}\otimes\dots\otimes v_{M_i}).
\end{equation}

However, \eqref{welldef1} are only determined up to adding elements of 
$$\langle \dd+\lambda_{M_{\!i\!-\!1}\!+1}+\dots+\lambda_{M_i} \rangle
= \langle \Lambda_i+\dd \rangle,$$
and we need to check that the right-hand side of \eqref{20170616:eq3} will remain the same after that.
Fix $1\le i\le n$, and replace in \eqref{20170616:eq3} the polynomial \eqref{welldef1} with
\begin{equation*}\label{welldef2}
(\Lambda_i+\dd) (h_i)^{\Gamma_i}_{\lambda_{M_{\!i\!-\!1}\!+1},\dots,\lambda_{M_i}}\!\!(v_{M_{\!i\!-\!1}\!+1}\otimes\dots\otimes v_{M_i})
\end{equation*}
for some map
\begin{equation*}\label{welldef3}
h_i\colon
\mc G(m_i)\times V^{\otimes m_i}
\longrightarrow
V[\lambda_{M_{\!i\!-\!1}\!+1},\dots,\lambda_{M_i}]
\,.
\end{equation*}
Let us introduce the shorthand notation
\begin{equation}\label{welldef4}
\tilde g_i^G
=\Big(
\Big|_{x_i=\Lambda_i+\partial}
(g_i)^G_{
\lambda_{M_{i-1}+1}+X(M_{i-1}+1),\dots,\lambda_{M_i}+X(M_i)}
\Big),
\end{equation}
for an arbitrary graph $G$.
Then in \eqref{20170616:eq3}, we need to replace $\tilde g_i^{\Delta^{m_1\dots m_n}_i(\Gamma)}$ with
\begin{equation}\label{welldef5}
\Big|_{x_i=\Lambda_i+\partial}
\bigl( \Lambda_i+\dd+X(M_{i-1}+1)+\dots+X(M_i) \bigr)
\tilde h_i^{\Delta^{m_1\dots m_n}_i(\Gamma)}.
\end{equation}

It follows from the definition \eqref{20170824:eq1} of $X(k)$, that
$$
X(M_{i-1}+1)+\dots+X(M_i) = \sum_j x_j,
$$
where the sum is over all $1\le j\le n$ such that $j$ is connected by an unoriented path with $i$ in the graph
$\Delta^{m_1\dots m_n}_0(\Gamma)$. Together with $x_i$, this gives the sum of all $x_j$ where $j$ is a vertex of the connected component $\Ga^0_i$ of $i$ in $\Delta^{m_1\dots m_n}_0(\Gamma)$. Then after setting all
$x_j=\Lambda_j+\partial$, we obtain
$$
\sum_{j\in \Ga^0_i} \Lambda_j+\partial_j = \Lambda_{\Ga^0_i} + \dd_{\Ga^0_i},
$$
using again the notation \eqref{com3}.
Then after applying $f^{\Delta^{m_1\dots m_n}_0(\Gamma)}_{\Lambda_1,\dots,\Lambda_n}$ we get $0$,
because $f$ satisfies the second sesquilinearity condition \eqref{eq:sesq2}.
\end{proof}
\begin{lemma}\label{lem:pcl2}
For every\/ $f\in P^\cl(n)$ and\/ $g_1\in P^\cl(m_1),\dots,g_n\in P^\cl(m_n)$, the composition\/
$f(g_1,\dots,g_n)$, defined by
\eqref{20170616:eq3}, is an element of\/ $P^\cl(M_n)$.
\end{lemma}
\begin{proof}
We need to check that $f(g_1,\dots,g_n)$
satisfies the cycle relations \eqref{eq:cycle1}, \eqref{eq:cycle2}
and the sesquilinearity conditions \eqref{eq:sesq1}, \eqref{eq:sesq2}.
Observe that
if $\Gamma\in\mc G(M_n)$ contains a cycle,
then one of the graphs $\Delta^{m_1\dots m_n}_i(\Gamma)$, $i=0,1,\dots,n$,
must contain a cycle as well. Evaluating $f$ for $i=0$ or $g_i$ for $1\le i\le n$, we obtain $0$,
because $f$ and $g_i$ satisfy \eqref{eq:cycle1}.
Therefore, $f(g_1,\dots, g_n)$
satisfies the first cycle relation \eqref{eq:cycle1}.

To prove the second cycle relation \eqref{eq:cycle2},
consider an oriented cycle $C\subset E(\Gamma)$ of $\Gamma\in\mc G(n)$.
Recalling \eqref{20170616:eq3}, we need to show that
\begin{equation}\label{20170823:eq2}
\sum_{e\in C}
f^{\Delta^{m_1\dots m_n}_0(\Gamma\backslash e)}_{\Lambda_1,\dots,\Lambda_n}
\bigl( \tilde g_1^{\Delta^{m_1\dots m_n}_i(\Gamma\backslash e)}
\otimes\cdots\otimes \tilde g_n^{\Delta^{m_1\dots m_n}_n(\Gamma\backslash e)} \bigr) 
=0,
\end{equation}
where we use the notation \eqref{welldef4}.
Given an edge $e\in C$,
consider its image $\Delta(e)$ under the map \eqref{20170823:eq5}.
Clearly, we have
\begin{equation*}
\Delta^{m_1\dots m_n}_i(\Gamma\backslash e)
=
\left\{
\begin{array}{l}
\displaystyle{
\vphantom{\Big(}
\Delta^{m_1\dots m_n}_i(\Gamma)\backslash\Delta(e),
\;\text{ if }\;
\Delta(e)\in E\big(\Delta^{m_1\dots m_n}_i(\Gamma)\big)
,} \\
\displaystyle{
\vphantom{\Big(}
\Delta^{m_1\dots m_n}_i(\Gamma),
\;\text{ otherwise.}
}
\end{array}
\right.
\end{equation*}
Hence, by 
Lemma \ref{20170823:lem1}, 
the left-hand side of \eqref{20170823:eq2} is equal to
\begin{align}
\notag
&\sum_{e'\in \Delta(C)\cap E(\Delta^{m_1\dots m_n}_0(\Gamma))}
\!\!\!\!
f^{\Delta^{m_1\dots m_n}_0(\Gamma)\backslash e'}_{\Lambda_1,\dots,\Lambda_n}
\bigl( \tilde g_1^{\Delta^{m_1\dots m_n}_i(\Gamma\backslash e)}
\otimes\cdots\otimes \tilde g_n^{\Delta^{m_1\dots m_n}_n(\Gamma\backslash e)} \bigr) 
\\ \label{20170823:eq3}
&+
\sum_{i=1}^n
\sum_{e'\in \Delta(C)\cap E(\Delta^{m_1\dots m_n}_i(\Gamma))}
\!\!\!\!\!\!\!\!\!\!\!\!\!\!
f^{\Delta^{m_1\dots m_n}_0(\Gamma)}_{\Lambda_1,\dots,\Lambda_n}
\big(
\tilde g_1^{\Delta^{m_1\dots m_n}_1(\Gamma)} \otimes\cdots
\\ \notag
&
\qquad\qquad\qquad\qquad\dots\otimes
\tilde g_i^{\Delta^{m_1\dots m_n}_i(\Gamma)
\backslash e'}
\otimes\cdots\otimes
\tilde g_n^{\Delta^{m_1\dots m_n}_n(\Gamma)} 
\big)\,.
\end{align}
If $\Delta(C)\subset E(\Delta^{m_1\dots m_n}_i(\Gamma))$,
then \eqref{20170823:eq3} reduces to
\begin{equation}\label{20170823:eq4}
\sum_{e'\in \Delta(C)}
f^{\Delta^{m_1\dots m_n}_0(\Gamma)}_{\Lambda_1,\dots,\Lambda_n}
\big(
\tilde g_1^{\Delta^{m_1\dots m_n}_1(\Gamma)} \otimes\cdots\otimes
\tilde g_i^{\Delta^{m_1\dots m_n}_i(\Gamma)
\backslash e'}
\otimes\cdots\otimes
\tilde g_n^{\Delta^{m_1\dots m_n}_n(\Gamma)} 
\big)\,.
\end{equation}
In this case, by Lemma \ref{20170823:lem2}(a),
$\Delta(C)$ is an oriented cycle of $\Delta^{m_1\dots m_n}_i(\Gamma)$,
hence \eqref{20170823:eq4} vanishes by the second cycle condition \eqref{eq:cycle2} for
$g_i$.
On the other hand,
if $\Delta(C)$ is not contained in $E(\Delta^{m_1\dots m_n}_i(\Gamma))$
for any $i=1,\dots,n$,
then by Lemma \ref{20170823:lem2},
$\Delta(C)\cap E(\Delta^{m_1\dots m_n}_0(\Gamma))$ is an oriented cycle 
of $\Delta^{m_1\dots m_n}_0(\Gamma)$.
In this case, the first sum of \eqref{20170823:eq3}
vanishes since $f$ satisfies \eqref{eq:cycle2}.
Moreover, each term in the second sum of \eqref{20170823:eq3}
vanishes as well, since $\Delta^{m_1\dots m_n}_0(\Gamma)$ is not acyclic
and $f$ satisfies \eqref{eq:cycle1}.
We conclude that $f(g_1,\dots, g_n)$
satisfies the second cycle condition \eqref{eq:cycle2} as claimed.

Next, we will prove that $f(g_1,\dots, g_n)$
satisfies the first sesquilinearity relation \eqref{eq:sesq1}.
Let $(h\to k)$ be an edge in the graph $\Gamma$.
We need to prove that the right-hand side of \eqref{20170616:eq3}
is a polynomial of $\lambda_h+\lambda_k$ and not of $\lambda_h$ and $\lambda_k$ separately.
First, suppose that for some $i=1,\dots,n$, we have 
$$
h,k\in\{M_{i-1}+1,\dots,M_i\}, \quad\text{i.e.} \quad h=M_{i-1}+r, \;\; k=M_{i-1}+q,
$$
for some $r,q\in\{1,\dots,m_i\}$.
In this case, $\lambda_h$ and $\lambda_k$ are both summands of $\Lambda_i$;
hence $f_{\Lambda_1,\dots,\Lambda_n}$ has the required property
(of being polynomial of $\lambda_h+\lambda_k$ and not of $\lambda_h$ and $\lambda_k$ separately).
The image of $(h\to k)$ under the map \eqref{20170823:eq5}
is an edge $(r\to q)$ in $\Delta^{m_1\dots m_n}_i(\Gamma)$.
Thus,
$(g_i)^{\Delta^{m_1\dots m_n}_i(\Gamma)}_{\lambda_{M_{i-1}+1},\dots,\lambda_{M_i}}$
is a polynomial of $\lambda_h+\lambda_k$
by the first sesquilinearity property of $g_i$.

Now suppose that 
$$h\in\{M_{i-1}+1,\dots,M_i\}, \qquad k\in\{M_{j-1}+1,\dots,M_j\},
$$
for different $i,j\in\{1,\dots,n\}$.
In this case, $(i\to j)$ is an edge in the graph 
$\Delta^{m_1\dots m_n}_0(\Gamma)$.
Therefore, $f^{\Delta^{m_1\dots m_n}_0(\Gamma)}_{\Lambda_1,\dots,\Lambda_j}$
is a polynomial of $\Lambda_i+\Lambda_j$,
and hence of $\lambda_h+\lambda_k$.
Furthermore, 
by the assumption \eqref{eq:sesq1} on $g_t$ $(t=1,\dots,n)$
and by the definition \eqref{20170824:eq1}
of the variables $X(1),\dots,X(M_n)$,
all the $\lambda_l$'s of the same connected component of $\lambda_h$ and $\lambda_k$
appear as summed in the polynomial
$\tilde g_t^{\Delta^{m_1\dots m_n}_t(\Gamma)}$. 
%
We conclude that \eqref{eq:sesq1} holds for $f(g_1,\dots, g_n)$, as claimed.

Finally, we will show that $f(g_1,\dots, g_n)$
satisfies the second sesquilinearity relation \eqref{eq:sesq2}.
Let $G$ be one of the connected components of $\Ga$, and consider the image 
$G_i = \Delta^{m_1\dots m_n}_i(G)$ $(0\le i\le n)$
of $G$ under the map \eqref{20170823:eq5}. 
Note that if $G_0$ contains a cycle, then $\Delta^{m_1\dots m_n}_0(\Ga)$ does,
which implies $(f(g_1,\dots, g_n))^\Ga=0$. Hence, we can suppose that $G_0$ is acyclic.
Then it is easy to see that all $G_i$ $(0\le i\le n)$ are connected. 
Furthermore, the set of vertices of $G$ is the disjoint union of the sets of vertices of $G_i$ $(1\le i\le n)$.
Thus, using again the notation \eqref{com3}, we have
$$
\la_G+\dd_G = \sum_{i=1}^n \la_{G_i}+\dd_{G_i},
$$
and to prove \eqref{eq:sesq2} for $f(g_1,\dots, g_n)$, it is enough to show that
\begin{equation}\label{welldef6}
(f(g_1,\dots,g_n))^\Gamma_{\lambda_1,\dots,\lambda_{M_n}}
\big((\la_{G_i}+\dd_{G_i})v\big) = 0, \qquad v\in V^{\otimes M_n}, \;\; 1\le i\le n.
\end{equation}
Since $g_i$ itself satisfies \eqref{eq:sesq2}, we have from \eqref{20170616:eq3}:
$$
f^{\Delta^{m_1\dots m_n}_0(\Gamma)}_{\Lambda_1,\dots,\Lambda_n}
\bigl( \tilde g_1^{\Delta^{m_1\dots m_n}_i(\Gamma)}
\otimes\cdots\otimes \tilde g_n^{\Delta^{m_1\dots m_n}_n(\Gamma)} \bigr) 
\Bigl( \Bigl( \la_{G_i}+\dd_{G_i} + \sum_{k\in G_i} X(k) \Bigr) v\Bigr)
=0.
$$
As in the proof of Lemma \ref{lem:pcl1} (cf.\ \eqref{welldef5}), we see
from the definition \eqref{20170824:eq1} of $X(k)$, that
$$
\sum_{k\in G_i} X(k) = \!\!\sum_{j\in G_0\backslash\{i\}} x_j.
$$
After setting all $x_j=\Lambda_j+\partial$, we can add $x_i$ to the above sum, because
$g_i$ is defined only up to adding elements of $\langle \Lambda_i+\partial \rangle$.
We obtain
$$
\sum_{j\in G_0} x_j \big|_{x_j=\Lambda_j+\partial} = \Lambda_{G_0}+\dd_{G_0}.
$$
After applying $f^{\Delta^{m_1\dots m_n}_0(\Gamma)}_{\Lambda_1,\dots,\Lambda_n}$ to this we get $0$,
since $f$ satisfies \eqref{eq:sesq2}.
This proves \eqref{welldef6} and finishes the proof of the lemma.
\end{proof}
\begin{theorem}\label{20170616:thm1}
The vector superspaces\/ $P^\cl(n)$, $n\geq0$,
together with the actions of the symmetric groups\/ $S_n$ given by \eqref{20170615:eq1}
and the composition maps defined by \eqref{20170616:eq3},
form an operad, which is graded by \eqref{eq:pclgr}.
\end{theorem}
\begin{proof}
First, let us check that $f^\si\in P^\cl(n)$ for every $f\in P^\cl(n)$ and $\sigma\in S_n$.
The cycle relations \eqref{eq:cycle1} and \eqref{eq:cycle2} for $f^\si$ are obvious, using the fact that
if $C\subset E(\Gamma)$ is an oriented cycle of $\Gamma$,
then $\sigma(C)$ (obtained by applying $\sigma$ to the tails and heads of all edges in $C$)
is an oriented cycle of $\sigma(\Gamma)$.
Next, if $\Gamma=\Gamma_1\sqcup\dots\sqcup\Gamma_s$
is a disjoint union of connected components, then
$\sigma(\Gamma)$ is a disjoint union of connected components $\sigma(\Gamma_1)\sqcup\dots\sqcup\sigma(\Gamma_s)$.
From here, it is easy to derive the sesquilinearity conditions \eqref{eq:sesq1} and \eqref{eq:sesq2} for $f^\sigma$.
Thus, $f^\si\in P^\cl(n)$. 

We have already shown in Lemma \ref{lem:pcl2} that the composition 
$f(g_1,\dots,g_n) \in P^\cl(M_n)$ for $f\in P^\cl(n)$ and $g_i\in P^\cl(m_i)$.
It is clear by construction that the action of the symmetric group \eqref{20170615:eq1}
and the composition maps \eqref{20170616:eq3}
are parity preserving linear maps $P^\cl(n)\to P^\cl(n)$ and
$P^\cl(n)\otimes P^\cl(m_1)\otimes\dots\otimes P^\cl(m_n)\to P^\cl(M_n)$, respectively.
The unity axioms \eqref{eq:operad3} are obvious, where the unit $1\in P^\cl(1)$ is the identity operator
$$
1^\bullet_\la(v) = v + \langle \dd+\la \rangle \in V[\la]/\langle \dd+\la \rangle \cong V,
$$
and $\bullet$ represents the graph with one vertex.
The fact that $P^\cl$ is a graded operad follows from Lemma \ref{20170823:lem1} and the definition \eqref{20170615:eq1}, \eqref{20170616:eq3} of the operad structure. 
To finish the proof of the theorem, we need to verify the associativity \eqref{eq:operad2} of the composition
and the equivariance \eqref{eq:operad4} of the symmetric group action.

To prove the associativity axiom, given $f\in P^\cl(n)$, $g_i\in P^\cl(m_i)$ and
$h_{ij}\in P^\cl(\ell_{ij})$, we need to show that $\ph=\psi$, where
\begin{align*}
\ph &= f\bigl(g_1(h_{11},\dots,h_{1m_1}),\dots,g_n(h_{n1},\dots,h_{nm_n})\bigr),
\\
\psi &= \bigl(f(g_1,\dots,g_n)\bigr)(h_{11},\dots,h_{1m_1},\dots,h_{n1},\dots,h_{nm_n}).
\end{align*}
Let us introduce the lexicographically ordered index sets
\begin{align*}
\mc J &= \{(ij)\,|\, 1\leq i\leq n, \; 1\leq j\leq m_i \}
\,, \\
\mc K &= \{(ijk)\,|\, 1\leq i\leq n, \; 1\leq j\leq m_i, \; 1\leq k\leq\ell_{ij} \}
\,,
\end{align*}
and the notation
\begin{align*}
\Lambda_i&=\sum_{j=1}^{m_i}\sum_{k=1}^{\ell_{ij}}\lambda_{ijk} \,, & 
\Lambda_{ij}&=\sum_{k=1}^{\ell_{ij}}\lambda_{ijk} \,,
\qquad 1\le i\le n \,, \;\; (ij)\in\mc J \,,
\\
L_i&=\sum_{j=1}^{m_i}\ell_{ij} \,, & 
L&=\sum_{i=1}^n L_i=\sum_{i=1}^n\sum_{j=1}^{m_i}\ell_{ij} \,.
\end{align*}
Then for any graph $\Gamma\in\mc G_0(L)$ and vectors $v_{ijk} \in V$,
we find from the definition of composition \eqref{20170616:eq3}:
\begin{equation*}
\begin{array}{l}
\displaystyle{
\vphantom{\Bigg(}
\ph^{\Gamma}_{(\lambda_{ijk})_{(ijk)\in\mc K}}
\Big(\bigotimes_{(ijk)\in\mc K}v_{ijk}\Big)
=
f^{\Delta_0^{L_1\dots L_n}(\Gamma)}_{\Lambda_1,\dots,\Lambda_n}
\Big(
} \\
\displaystyle{
\vphantom{\Bigg(}
\qquad
\bigotimes_{i=1}^n
(g_i)
^{\Delta_0^{\ell_{i1}\dots\ell_{im_i}}(\Delta_i^{L_1\dots L_n}(\Gamma))}
_{\Lambda_{i1},\dots,\Lambda_{im_i}}
\Big(
\bigotimes_{j=1}^{m_i}
(h_{ij})
^{\Delta_j^{\ell_{i1}\dots\ell_{im_i}}(\Delta_i^{L_1\dots L_n}(\Gamma))}
_{\lambda_{ij1},\dots,\lambda_{ij\ell_{ij}}}
\big(\bigotimes_{k=1}^{\ell_{ij}}v_{ijk}\big)
\Big)
\Big)
,} \\
\displaystyle{
\vphantom{\Bigg(}
\psi^{\Gamma}_{(\lambda_{ijk})_{(ijk)\in\mc K}}
\Big(\bigotimes_{(ijk)\in\mc K}v_{ijk}\Big)
} 
\displaystyle{
\vphantom{\Bigg(}
=
f^{\Delta_0^{m_1\dots m_n}(\Delta_0^{\ell_{11}\dots\ell_{nm_n}}(\Gamma))}
_{\Lambda_1,\dots,\Lambda_n}
\Big( 
} \\
\displaystyle{
\vphantom{\Bigg(}
\qquad
\bigotimes_{i=1}^n
(g_i)
^{\Delta_i^{m_1\dots m_n}(\Delta_0^{\ell_{11}\dots\ell_{nm_n}}(\Gamma))}
_{\Lambda_{i1},\dots,\Lambda_{im_i}}
\Big(
\bigotimes_{j=1}^{m_i}
(h_{ij})
^{\Delta_{m_1+\dots+m_{i-1}+j}^{\ell_{11}\dots\ell_{nm_n}}(\Gamma)}
_{\lambda_{ij1},\dots,\lambda_{ij\ell_{ij}}}
\big(\bigotimes_{k=1}^{\ell_{ij}}v_{ijk}\big)
\Big)
\Big)
.}
\end{array}
\end{equation*}
The above right-hand sides are equal by Proposition \ref{20170615:prop1},
thus proving the associativity axiom \eqref{eq:operad2}.

Now we will prove the supersymmetric equivariance \eqref{eq:operad4}.
Let $f\in P^\cl(n)$, $g_i\in P^\cl(m_i)$ as before, and $\si\in S_n$, $\tau_i\in S_{m_i}$ for
$1\le i\le n$.
Then for a graph $\Gamma\in\mc G_0(m_1+\dots+m_n)$ and vectors $v_{ij}\in V$, we compute:
\begin{equation*}
\begin{array}{l}
\displaystyle{
\vphantom{\Bigg(}
\big(f^{\sigma}(g_1^{\tau_1},\dots,g_n^{\tau_n})\big)^{\Gamma}_{\lambda_{11},\dots,\lambda_{nm_n}}
\Big(\bigotimes_{(ij)\in\mc J}v_{ij}\Big)
} \\
\displaystyle{
\vphantom{\Bigg(}
=
(f^{\sigma})^{\Delta_0^{m_1\dots m_n}(\Gamma)}_{\Lambda_1,\dots,\Lambda_n}
\Big(
\bigotimes_{i=1}^n
(g_i^{\tau_i})^{\Delta_i^{m_1\dots m_n}(\Gamma)}_{\lambda_{i1},\dots,\lambda_{im_i}}
\big(\bigotimes_{j=1}^{m_i} v_{ij}\big)
\Big)
} \\
\displaystyle{
\vphantom{\Bigg(}
=
f^{\sigma(\Delta_0^{m_1\dots m_n}(\Gamma))}_{\Lambda_{\sigma^{-1}(1)},\dots,\Lambda_{\sigma^{-1}(n)}}
\Big(
\sigma\Big(
\bigotimes_{i=1}^n
(g_i)^{\tau_i(\Delta_i^{m_1\dots m_n}(\Gamma))}_{\lambda_{i\tau_i^{-1}(1)},\dots,\lambda_{i\tau_i^{-1}(m_i)}}
\big(\tau_i\big(\bigotimes_{j=1}^{m_i} v_{ij}\big)\big)
\Big)
\Big)
} \\
\displaystyle{
\vphantom{\Bigg(}
=
\epsilon_g(\sigma)\epsilon_{\bigotimes_{i=1}^n\tau_i(\bigotimes_{j=1}^{m_1}v_{ij})}(\sigma) \,
f^{\sigma(\Delta_0^{m_1\dots m_n}(\Gamma))}_{\Lambda_{\sigma^{-1}(1)},\dots,\Lambda_{\sigma^{-1}(n)}}
\Big(
} \\
\displaystyle{
\vphantom{\Bigg(}
\,\,\,\,\,\,
\bigotimes_{i=1}^n
(g_{\sigma^{-1}(i)})
^{\tau_{\sigma^{-1}(i)}(\Delta_{\sigma^{-1}(i)}^{m_1\dots m_n}(\Gamma))}
_{\lambda_{{\sigma^{-1}(i)}\tau_{\sigma^{-1}(i)}^{-1}(1)},\dots,\lambda_{{\sigma^{-1}(i)}\tau_{\sigma^{-1}(i)}^{-1}(m_{\sigma^{-1}(i)})}}
\Big(\tau_{\sigma^{-1}(i)}\big(\bigotimes_{j=1}^{m_{\sigma^{-1}(i)}} v_{{\sigma^{-1}(i)}j}\big)\Big)
\Big)
} \\
\displaystyle{
\vphantom{\Bigg(}
=
\epsilon_g(\sigma)\epsilon_{\bigotimes_{i=1}^n\tau_i(\bigotimes_{j=1}^{m_1}v_{ij})}(\sigma) \,
f
^{\Delta_0^{m_{\sigma^{-1}(1)}\dots m_{\sigma^{-1}(n)}}\big((\sigma(\tau_1,\dots,\tau_n))(\Gamma)\big)}
_{\Lambda_{\sigma^{-1}(1)},\dots,\Lambda_{\sigma^{-1}(n)}}
\Big(
} \\
\displaystyle{
\vphantom{\Bigg(}
\,\,\,\,\,\,
\bigotimes_{i=1}^n
(g_{\sigma^{-1}(i)})
^{\Delta_i^{m_{\sigma^{-1}(1)}\dots m_{\sigma^{-1}(n)}}\big((\sigma(\tau_1,\dots,\tau_n))(\Gamma)\big)}
_{\lambda_{{\sigma^{-1}(i)}\tau_{\sigma^{-1}(i)}^{-1}(1)},\dots,\lambda_{{\sigma^{-1}(i)}\tau_{\sigma^{-1}(i)}^{-1}(m_{\sigma^{-1}(i)})}}
\Big(\tau_{\sigma^{-1}(i)}\big(\bigotimes_{j=1}^{m_{\sigma^{-1}(i)}} v_{{\sigma^{-1}(i)}j}\big)\Big)
\Big)
} \\
\displaystyle{
\vphantom{\Bigg(}
=
\epsilon_g(\sigma)
\big(f(g_{\sigma^{-1}(1)},\dots,g_{\sigma^{-1}(n)})\big)
^{(\sigma(\tau_1,\dots,\tau_n))(\Gamma)}
_{\lambda_{\sigma^{-1}(1)\tau_{\sigma^{-1}(1)}^{-1}(1)},\dots,
\lambda_{\sigma^{-1}(n)\tau_{\sigma^{-1}(n)}^{-1}(m_{\sigma^{-1}(n)})}}
\Big(
} \\
\displaystyle{
\vphantom{\Bigg(}
\,\,\,\,\,\,\,\,
\big(\sigma(\tau_1,\dots,\tau_n)\big)\big(\bigotimes_{(ij)\in\mc J}v_{ij}\big)
\Big)
} \\
\displaystyle{
\vphantom{\Bigg(}
=
\epsilon_g(\sigma)
\big(
(f(g_{\sigma^{-1}(1)},\dots,g_{\sigma^{-1}(n)}))^{\sigma(\tau_1,\dots,\tau_n)}
\big)
^{\Gamma}
_{\lambda_{11},\dots,\lambda_{nm_n}}
\Big(\bigotimes_{(ij)\in\mc J}v_{ij}\Big)
.}
\end{array}
\end{equation*}
For the first equality above,
we used the definition \eqref{20170616:eq3} of the composition;
for the second equality, the definition \eqref{20170615:eq1} 
of the action of the symmetric group on $P^\cl(n)$;
for the third equality, the definition \eqref{eq:operad6}--\eqref{eq:operad14} 
of the action of $S_n$ on a tensor product of $n$ vector superspaces;
for the fourth equality, Proposition \ref{20170615:prop2};
for the fifth equality, we used again \eqref{20170616:eq3}, \eqref{eq:operad6}
and the definition of the composition of permutations;
and for the last equality, we used again \eqref{20170615:eq1}.

This completes the proof of the theorem.
\end{proof}

\subsection{Poisson vertex algebras and the operad $P^\cl$}\label{sec:6.3}

As in Section \ref{sec:5.2},
given the vector superspace $V$, with parity $p$,
and the even endomorphism $\partial\in\End(V)$,
we denote by $\Pi V$ the same vector space with reversed parity $\bar p=1-p$.
Consider the corresponding operad $P^\cl(\Pi V)$
from Section \ref{sec:6.2}
and the associated $\mb Z$-graded Lie superalgebra 
$W^{\cl}(\Pi V):=W(P^\cl(\Pi V))$
given by Theorem \ref{20170603:thm2}.
\begin{theorem}\label{20170616:thm2}
We have a bijective correspondence between 
the odd elements\/ $X\in W^{\cl}_1(\Pi V)$ such that\/ $X\Box X=0$
and the Poisson vertex superalgebra structures on\/ $V$,
defined as follows.
The  commutative associative product and the\/ $\lambda$-bracket 
of the Poisson vertex superalgebra\/ $V$ corresponding to\/ $X$
are given by
\begin{equation}\label{20170616:eq4}
ab=(-1)^{p(a)}
X^{\bullet\!-\!\!\!\!\to\!\bullet}(a\otimes b)
\,\,,\,\,\,\,
[a_\lambda b]=(-1)^{p(a)}X^{\bullet\,\,\bullet}_{\lambda,-\lambda-\partial}(a\otimes b)
\,.
\end{equation}
\end{theorem}
\begin{proof}
Note that, by the first sesquilinearity condition \eqref{eq:sesq1},
the polynomial $X^{\bullet\!-\!\!\!\!\to\!\bullet}_{\lambda_1,\lambda_2}$
depends only on $\lambda_1+\lambda_2\equiv-\partial$.
Hence, it is independent of $\lambda_1,\lambda_2$.
For this reason, in the first equation of \eqref{20170616:eq4} 
we omitted the subscripts $\lambda_1,\lambda_2$.

First, we check that the symmetry of $X$ translates
to the commutativity of the product  $ab$ and the skewsymmetry 
of the $\lambda$-bracket $[a_\lambda b]$.
We have
\begin{align*}
X^{\bullet\!-\!\!\!\!\to\!\bullet}(v_1\otimes v_2)
&=
(X^{(12)})^{\bullet\!-\!\!\!\!\to\!\bullet}(v_1\otimes v_2)
=
(-1)^{\bar p(v_1)\bar p(v_2)}X^{\bullet\!\leftarrow\!\!\!\!-\!\bullet}(v_2\otimes v_1)
\\
& =
(-1)^{p(v_1)+p(v_2)+p(v_1)p(v_2)} X^{\bullet\!-\!\!\!\!\to\!\bullet}(v_2\otimes v_1)
\,,
\end{align*}
which, by the first equation in \eqref{20170616:eq4}, is equivalent to the symmetry condition of the product:
$v_1v_2=(-1)^{p(v_1)p(v_2)}v_2v_1$.
Similarly, 
evaluating the identity $X=X^{(12)}$ on the disconnected graph $\bullet\,\,\,\bullet$, we get
$$
X^{\bullet\,\,\bullet}_{\lambda,-\lambda-\partial}(v_1\otimes v_2)
=
(X^{(12)})^{\bullet\,\,\bullet}_{\lambda,-\lambda-\partial}(v_1\otimes v_2)
=
(-1)^{\bar p(v_1)\bar p(v_2)}
X^{\bullet\,\,\bullet}_{-\lambda-\partial,\lambda}(v_2\otimes v_1)
\,,
$$
which, by the second equation in \eqref{20170616:eq4}, 
is equivalent to the skewsymmetry condition \eqref{20170612:eq5}
of the $\lambda$-bracket.

Next, we need to prove that the condition $X\Box X=0$ translates to three conditions:
the associativity of the product $ab$, the Jacobi identity \eqref{20170612:eq5} 
for the $\lambda$-bracket $[a_\lambda b]$,
and the Leibniz rule \eqref{20170614:eq1}.
Recall that, by \eqref{eq:box1},
$$
X\Box X=X\circ_1X+X\circ_2X+(X\circ_2 X)^{(12)}
\,.
$$
Since, by construction, $X\Box X$ is invariant by the action of the symmetric group,
to impose the condition $X\Box X=0$ is the same as to impose
$(X\Box X)^\Gamma=0$ for each of the three graphs:
\begin{equation}\label{eq:graphs3}
\begin{tikzpicture}
\draw (-4,2.5) circle [radius=0.1];
\node at (-4,2.2) {1};
\draw (-3.25,2.5) circle [radius=0.1];
\node at (-3.25,2.2) {2};
\draw (-2.5,2.5) circle [radius=0.1];
\node at (-2.5,2.2) {3};
\node at (-2,2.35) {,};
\draw (-1,2.5) circle [radius=0.1];
\node at (-1,2.2) {1};
\draw (-0.25,2.5) circle [radius=0.1];
\node at (-0.25,2.2) {2};
\draw (0.5,2.5) circle [radius=0.1];
\node at (0.5,2.2) {3};
\draw[->] (-0.15,2.5) -- (0.4,2.5);
\node at (1,2.35) {,};
\draw (2,2.5) circle [radius=0.1];
\node at (2,2.2) {1};
\draw (2.75,2.5) circle [radius=0.1];
\node at (2.75,2.2) {2};
\draw (3.5,2.5) circle [radius=0.1];
\node at (3.5,2.2) {3};
\draw[->] (2.1,2.5) -- (2.65,2.5);
\draw[->] (2.85,2.5) -- (3.4,2.5);
\node at (4,2.35) {.};
\end{tikzpicture}
\end{equation}
Evaluating all three summands of $X\Box X$  on the disconnected graph $\bullet\,\,\,\bullet\,\,\,\bullet$,
we get, 
by the definition \eqref{20170616:eq3} of the composition maps,
\begin{align*}
(X\circ_1 X)^{\bullet\,\,\bullet\,\,\bullet}_{\lambda_1,\lambda_2,\lambda_3}(v_1\otimes v_2\otimes v_3)
&\vphantom{\Big(}
=
X^{\bullet\,\,\bullet}_{\lambda_1+\lambda_2,\lambda_3}
\big(
X^{\bullet\,\,\bullet}_{\lambda_1,\lambda_2}(v_1\otimes v_2)\, \otimes v_3
\big))
\\
&\vphantom{\Big(} 
=
(-1)^{p(v_2)}[[{v_1}_{\lambda_1}v_2]_{\lambda_1+\lambda_2}v_3]
\,,\\
(X\circ_2 X)^{\bullet\,\,\bullet\,\,\bullet}_{\lambda_1,\lambda_2,\lambda_3}(v_1\otimes v_2\otimes v_3)
&\vphantom{\Big(}
=
(-1)^{\bar p(v_1)}
X^{\bullet\,\,\bullet}_{\lambda_1,\lambda_2+\lambda_3}
\big(
v_1\otimes X^{\bullet\,\,\bullet}_{\lambda_2,\lambda_3}(v_2\otimes v_3)
\big))
\\
&\vphantom{\Big(}
=
(-1)^{1+p(v_2)}[{v_1}_{\lambda_1}[{v_2}_{\lambda_2}v_3]]
\,,\\
((X\circ_2 X)^{(12)})^{\bullet\,\,\bullet\,\,\bullet}_{\lambda_1,\lambda_2,\lambda_3}
(v_1\otimes v_2\otimes v_3)
&\vphantom{\Big(}
=
(-1)^{\bar p(v_1)\bar p(v_2)}
(X\circ_2 X)^{\bullet\,\,\bullet\,\,\bullet}_{\lambda_2,\lambda_1,\lambda_3}(v_2\otimes v_1\otimes v_3)
\\
&\vphantom{\Big(}
=
(-1)^{p(v_2)+p(v_1)p(v_2)}[{v_2}_{\lambda_2}[{v_1}_{\lambda_1}v_3]]
\,.
\end{align*}
Hence, the condition $(X\Box X)^{\bullet\,\,\bullet\,\,\bullet}=0$
is equivalent to the Jacobi identity \eqref{20170612:eq5} for the $\lambda$-bracket.

Evaluating all three summands of $X\Box X$  on the second graph in \eqref{eq:graphs3},
we get,
by the definition \eqref{20170616:eq3} of the composition maps,
\begin{align*}
\vphantom{\Big(}
(X\circ_1 X&)^{\bullet\,\,\,\,\bullet\!-\!\!\!\!\to\!\bullet}_{\lambda_1,\lambda_2,\lambda_3}
(v_1\otimes v_2\otimes v_3)
=
X^{\bullet\!-\!\!\!\!\to\!\bullet}
\big(
X^{\bullet\,\,\,\bullet}_{\lambda_1,\lambda_2+x_3}(v_1\otimes v_2)\, \otimes 
\big(\big|_{x_3=\lambda_3+\partial}v_3\big)
\big)
\\
&\vphantom{\Big(} 
=
(-1)^{p(v_2)}[{v_1}_{\lambda_1}v_2]v_3
\,,\\
\vphantom{\Big(}
(X\circ_2 X&)^{\bullet\,\,\,\,\bullet\!-\!\!\!\!\to\!\bullet}_{\lambda_1,\lambda_2,\lambda_3}
(v_1\otimes v_2\otimes v_3)
=
(-1)^{\bar p(v_1)}
X^{\bullet\,\,\,\bullet}_{\lambda_1,\lambda_2+\lambda_3}
\big(
v_1\otimes X^{\bullet\!-\!\!\!\!\to\!\bullet}(v_2\otimes v_3)
\big))
\\
&\vphantom{\Big(}
=
(-1)^{1+p(v_2)}[{v_1}_{\lambda_1} v_2 v_3]
\,,\\
\vphantom{\Big(}
((X\circ_2 X)^{(12)}&)^{\bullet\,\,\,\,\bullet\!-\!\!\!\!\to\!\bullet}_{\lambda_1,\lambda_2,\lambda_3}
(v_1\otimes v_2\otimes v_3)
=
(-1)^{\bar p(v_1)\bar p(v_2)}
(X\circ_2 X)^{
\begin{tikzpicture}
\draw[fill] (0,0) circle [radius=0.04];
\draw[fill] (0.3,0) circle [radius=0.04];
\draw[fill] (0.6,0) circle [radius=0.04];
\draw[->] (0,0.1) to [out=90,in=90] (0.6,0.1);
\end{tikzpicture}
}_{\lambda_2,\lambda_1,\lambda_3}(v_2\otimes v_1\otimes v_3)
\\
&\vphantom{\Big(}
=
(-1)^{\bar p(v_2)+\bar p(v_1)\bar p(v_2)}
X^{\bullet\!-\!\!\!\!\to\!\bullet}
\big(
\big(\big|_{x_2=\lambda_2+\partial}v_2\big)\otimes 
X^{\bullet\,\,\,\bullet}_{\lambda_1,\lambda_3+x_2}(v_1\otimes v_3)
\big))
\\
&\vphantom{\Big(}
=
(-1)^{p(v_2)+p(v_1)p(v_2)}{v_2}[{v_1}_{\lambda_1}v_3]
\,.
\end{align*}
Hence, the condition $(X\Box X)^{\bullet\,\,\,\,\bullet\!-\!\!\!\!\to\!\bullet}=0$
is equivalent to the Leibniz rule \eqref{20170614:eq1}.

Finally, we evaluate all three summands of $X\Box X$  on the third graph in \eqref{eq:graphs3}.
We get
\begin{align*}
\vphantom{\Big(}
(X\circ_1 X&)^{\bullet\!-\!\!\!\!\to\!\bullet\!-\!\!\!\!\to\!\bullet}_{\lambda_1,\lambda_2,\lambda_3}
(v_1\otimes v_2\otimes v_3)
=
X^{\bullet\!-\!\!\!\!\to\!\bullet}
\big(
X^{\bullet\!-\!\!\!\!\to\!\bullet}(v_1\otimes v_2)\, \otimes v_3
\big)
\\
&\vphantom{\Big(} 
=
(-1)^{p(v_2)}({v_1}v_2)v_3
\,,\\
\vphantom{\Big(}
(X\circ_2 X&)^{\bullet\!-\!\!\!\!\to\!\bullet\!-\!\!\!\!\to\!\bullet}_{\lambda_1,\lambda_2,\lambda_3}
(v_1\otimes v_2\otimes v_3)
=
(-1)^{\bar p(v_1)}
X^{\bullet\!-\!\!\!\!\to\!\bullet}
\big(
v_1\otimes X^{\bullet\!-\!\!\!\!\to\!\bullet}(v_2\otimes v_3)
\big))
\\
&\vphantom{\Big(}
=
(-1)^{1+p(v_2)}{v_1} (v_2 v_3)
\,,\\
\vphantom{\Big(}
((X\circ_2 X)^{(12)}&)^{\bullet\!-\!\!\!\!\to\!\bullet\!-\!\!\!\!\to\!\bullet}_{\lambda_1,\lambda_2,\lambda_3}
(v_1\otimes v_2\otimes v_3)
=
(-1)^{\bar p(v_1)\bar p(v_2)}
(X\circ_2 X)^{
\begin{tikzpicture}
\draw[fill] (0,0) circle [radius=0.04];
\draw[fill] (0.3,0) circle [radius=0.04];
\draw[fill] (0.6,0) circle [radius=0.04];
\draw[->] (0,0.1) to [out=90,in=90] (0.6,0.1);
\draw[<-] (0,-0.1) to [out=-90,in=-90] (0.3,-0.1);
\end{tikzpicture}
}_{\lambda_2,\lambda_1,\lambda_3}(v_2\otimes v_1\otimes v_3)
\\
&\vphantom{\Big(}
=
(-1)^{\bar p(v_2)+\bar p(v_1)\bar p(v_2)}
X^{
\begin{tikzpicture}
\draw[fill] (0,0) circle [radius=0.04];
\draw[fill] (0.3,0) circle [radius=0.04];
\draw[->] (0,0.1) to [out=90,in=90] (0.3,0.1);
\draw[<-] (0,-0.1) to [out=-90,in=-90] (0.3,-0.1);
\end{tikzpicture}
}
(\cdots)=0
\,.
\end{align*}
The last equality holds because $X$ evaluated on a graph containing a cycle is,
by definition, zero.
Hence, the condition $(X\Box X)^{\bullet\!-\!\!\!\!\to\!\bullet\!-\!\!\!\!\to\!\bullet}=0$
is equivalent to the associativity of the product.
This completes the proof.
\end{proof}

In view of Theorem \ref{20170616:thm2},
we have the definition of the corresponding cohomology complex.
\begin{definition}
Let $V$ be a Poisson vertex superalgebra.
The corresponding \emph{PVA cohomology complex} is defined as
$$
(W^{\cl}(\Pi V),\ad X)
\,,
$$
where $X\in W_1(\Pi V)_{\bar 1}$ is given by \eqref{20170616:eq4}.
\end{definition}
\begin{remark}\label{rem:last2}
Let $V=\bigoplus_r\gr^rV$ be a graded vector superspace.
Recall from Remark \ref{rem:last}
that we have the corresponding grading of the operad $P^\cl(\Pi V)$,
and hence of the Lie superalgebra $W^\cl(\Pi V)$.
It is easy to check, as for Theorem \ref{20160719:thm-ref},
that under the correspondence from Theorem \ref{20170616:thm2},
the structures of graded Poisson vertex algebra on $V$
are in bijection with the 
odd elements $X\in\gr^1 W^{\cl}_1(\Pi V)$ satisfying $X\Box X=0$.
(Recall that $V$ is a graded Poisson vertex algebra if $(\gr^pV)(\gr^qV)\subset\gr^{p+q}V$
and $[\gr^p V_\lambda\gr^qV]\subset\gr^{p+q-1}V[\lambda]$.)
\end{remark}

\subsection{Relation between $\gr P^\ch$ and $P^\cl$}\label{sec:6.4}

Recall from Corollary \ref{cor:gr} that for every $\bar X\in\gr^r P^\ch(k+1)$
(with a representative $X\in\fil^r P^\ch(k+1)$)
and every graph $\Ga\in\mc G(k+1)$ with $r$ edges, 
we have the map 
$$
X^\Ga (= X(p_\Gamma)) \colon
V^{\otimes(k+1)}
\to
V[\lambda_0,\dots,\lambda_k]\big/\big\langle\partial+\lambda_0+\dots+\lambda_k\big\rangle.
$$

\begin{theorem}\label{thm:mor}
For every vector superspace\/ $V$ with an\/ $\mb F[\dd]$-module structure, 
there is a canonical injective morphism of graded operads
\begin{equation}\label{eq:morph}
\gr P^\ch(V) \,{\hookrightarrow}\, P^\cl(V)
\,,
\end{equation}
mapping $\bar X\in\gr^r P^\ch(k+1)$ to:
\begin{equation}\label{eq:nc1}
\{X^\Gamma\,|\,\Gamma\in\mc G(k+1)\,\,\,\,\mathrm{ with }\,\,\,\, r \,\,\,\,\mathrm{ edges}\}
\,\in\gr^r P^\cl(k+1)
\,.
\end{equation}
This morphism is a bijection $\gr P^\ch(k+1)\stackrel{\sim}{\rightarrow} P^\cl(k+1)$
for $k=-1,0,1$.
\end{theorem}
\begin{proof}
As a consequence of Corollary \ref{cor:gr}(a)-(b) and the definition of the graded operad $P^\cl$,
the map \eqref{eq:nc1} is a well defined morphism of operads,
and by Corollary \ref{cor:gr}(c) this morphism is injective.
Surjectivity of the morphism for $k=-1,0$ is immediate. Let us prove it for $k=1$.
Recall that $\gr P^\cl(2)=\gr^0 P^\cl(2)\oplus\gr^1P^\cl(2)$.
By definition, $\gr^0P^\cl(2)$ consists of maps
$$
X^{\bullet\,\,\bullet}\,:\,\,V^{\otimes2}\,\longrightarrow\,V[\lambda]
\,,
$$
satisfying the sesquilinearity conditions
$$
X^{\bullet\,\,\bullet}_\lambda((\partial v_0)\otimes v_1)
=
-\lambda X^{\bullet\,\,\bullet}_\lambda(v_0\otimes v_1)
\,\,,\,\,\,\,
X^{\bullet\,\,\bullet}_\lambda(v_0\otimes(\partial v_1))
=
(\lambda +\partial) X^{\bullet\,\,\bullet}_\lambda(v_0\otimes v_1)
\,.
$$
A preimage $\tilde X\in F^0P^\ch(2)=P^\ch(2)$ can be constructed by letting
$$
\tilde X(v_0,v_1;z_{01}^n)
=
\left\{
\begin{array}{l}
\displaystyle{
\big(-\frac\partial{\partial\lambda}\big)^n
X^{\bullet\,\,\bullet}_\lambda(v_0\otimes v_1)
\,\,\text{ if }\,\, n\geq0\,, 
} \\
\displaystyle{
(-1)^m\!\!\!
\int_0^\lambda d\mu_1\dots\int_0^{\mu_{m-1}}\!\!\!\!\!\!
d\mu_m\,
X^{\bullet\,\,\bullet}_{\mu_m}(v_0\otimes v_1)
\,\,\text{ if }\,\, n=-m\leq-1\,.
}
\end{array}
\right.
$$
It is not hard to check that, indeed, $\tilde X$ is a well defined element of $P^\ch(2)$
and the image of its coset $[\tilde X]\in\gr^0 P^\ch(2)=F^0P^\ch(2)/F^1P^\ch(2)$
via the morphism \eqref{eq:morph} coincides with $X^{\bullet\,\,\bullet}$.
Next, $\gr^1P^\cl(2)$ consists of $\mb F[\partial]$-module homomorphisms
$$
X^{\bullet\!\to\!\bullet}\,:\,\,V^{\otimes2}\,\longrightarrow\,V
\,.
$$
A preimage $\tilde X\in F^1P^\ch(2)$ can be constructed by letting
$$
\tilde X(v_0,v_1;z_{01}^n)
=
\left\{
\begin{array}{l}
\displaystyle{
0\,\,\text{ if }\,\, n\geq0\,, 
} \\
\displaystyle{
\frac{(-1)^m}{m!}
X^{\bullet\!\to\!\bullet}((\partial+\lambda)^mv_0\otimes v_1)
\,\,\text{ if }\,\, n=-m-1\leq-1\,.
}
\end{array}
\right.
$$
Again, it is not hard to check that $\tilde X$ is a well defined element 
of $F^1P^\ch(2)=\gr^1 P^\ch(2)$
and its image via the morphism \eqref{eq:morph} coincides with $X^{\bullet\!\to\!\bullet}$.
\end{proof}
\begin{example}\label{ex:last}
Let $V$ be a non-unital vertex algebra.
By Theorem \ref{20160719:thm}, 
the vertex algebra structure of $V$ corresponds to an odd element 
$X\in W^\ch_1(\Pi V)=P^\ch(2)(\Pi V)$ such that $X\Box X=0$.
The filtration \eqref{fil4} of $P^\ch(2)$ is
$$
\fil^0P^\ch(2)=P^\ch(2)
\,\,,\,\,\,\,
\fil^1P^\ch(2)
=
\big\{f\,\big|\, f(\mc O_2^T)=0\big\}
\,\,,\,\,\,\,
\fil^2P^\ch(2)=0
\,.
$$
Hence, the image $X_0$ of $\gr^0 X$ in $P^\cl(2)$
via the map defined in Theorem \ref{thm:mor},
is the element
$$
{X_0}^{\bullet\,\,\,\bullet}=X(1)
\,\,\,\,,\,\,\,\,\,\,\,\,
{X_0}^{\bullet\!-\!\!\!\!\to\!\bullet}=0
\,.
$$
Thus we obtain a PVA structure on $V$,
where the $\lambda$-bracket is the same as for the vertex algebra $V$,
and the commutative associative product is zero.
\end{example}
Recall that if $V$ is a filtered vector space, then $P^\ch(V)$ is a filtered operad
with respect to the refined filtration introduced in Section \ref{sec:filch-ref},
and $P^\cl(\gr V)$ is a graded operad
with respect to the refined grading introduced in Remark \ref{rem:last}.
Then Theorem \ref{thm:mor} still holds:
\begin{theorem}\label{thm:mor-ref}
We have a canonical injective morphism of graded operads
from\/ 
\begin{equation}\label{eq:morph-gr}
\gr P^\ch(V)
\,\hookrightarrow\,
P^\cl(\gr V)
\,.
\end{equation}
Explicitly, $\bar f\in\gr^rP_k^\ch(V)$, with a representative\/ $f\in\fil^rP^\ch(V)$,
is mapped to the element\/ $\tilde f\in\gr^rP_k^\cl(\gr V)$ defined as follows.
If\/ $\Gamma\in\mc G(k)$ has $s$ edges
and\/ $\bar v_1\otimes \dots\otimes \bar v_k\in\gr^{t}(V^{\otimes k})
=\bigoplus_{r_1+\dots+r_k=t}\gr^{r_1}V\otimes\dots\otimes\gr^{r_k}V$,
we let
\begin{equation}\label{eq:nc2}
\tilde{f}^{\Gamma}_{\lambda_1,\dots,\lambda_k}(\bar v_1\otimes\dots\otimes \bar v_k)
=
f^{z_1,\dots,z_k}_{\lambda_1,\dots,\lambda_k}(v_1,\dots, v_k;p_\Gamma)
+
\fil^{s+t-r-1}V
\end{equation}
in\/ $(\gr^{s+t-r}V)[\lambda_1,\dots,\lambda_k]/\langle\partial+\lambda_1+\dots+\lambda_k\rangle$.
\end{theorem}
\begin{proof}
Straightforward.
\end{proof}
Let $V$ be a filtered vertex algebra
and let $\gr V$ be the associated graded Poisson vertex algebra.
By Theorem \ref{20160719:thm-ref}, 
the vertex algebra structure of $V$ corresponds to an odd element 
$X\in \fil^1W^\ch_1(\Pi V)=\fil^1P^\ch(2)(\Pi V)$ such that $X\Box X=0$.
Moreover, by Remark \ref{rem:last2},
the Poisson vertex algebra structure of $\gr V$ corresponds to an odd element 
$\tilde X\in \gr^1W^\cl_1(\Pi\gr V)=\gr^1 P^\cl(2)(\Pi\gr V)$ such that $\tilde X\Box\tilde X=0$.
\begin{theorem}\label{thm:last2}
The image of\/ $\bar X\in\gr^1W^\ch_1(\Pi V)$ via the morphism defined 
by Theorem \ref{thm:mor-ref} is\/ $\tilde X$.
\end{theorem}
\begin{proof}
The proof follows by construction.
\end{proof}
Obviously, the morphism of operads defined in Theorem \ref{thm:mor-ref} 
induces a Lie superalgebra injective homomorphism 
\begin{equation}\label{eq:nomore}
\gr W^\ch(\Pi V)
\,\hookrightarrow\,
W^\cl(\gr\Pi V)
\,.
\end{equation}
Moreover, by Theorem \ref{thm:last2},
$\bar X=\gr X$, where $X\in W_1^\ch(\Pi V)$ is associated to the vertex algebra structure of $V$, 
is mapped by the homomorphism \eqref{eq:nomore} to $\tilde X\in W_1^\cl(\gr\Pi V)$,
associated to the PVA structure of $\gr V$.
Summarizing, we have:
\begin{theorem}\label{cor:nomore}
Let\/ $V$ be a filtered vertex algebra and let\/ $\gr V$ be the associated graded Poisson vertex algebra.
Denote by\/ $X\in \fil^1 W^\ch_1(\Pi V)$ the element corresponding to the vertex algebra structure of\/ $V$
by \eqref{20160719:eq3b} (cf.\ Theorem \ref{20160719:thm-ref}),
and denote by\/ $\tilde X\in \gr^1 W^\cl_1(\gr \Pi V)$ 
the element corresponding to the Poisson vertex algebra structure of\/ $\gr V$
by \eqref{20170616:eq4} (cf.\ Remark \ref{rem:last2}).
\begin{enumerate}[(a)]
\item
There is a canonical injective homomorphism of graded Lie superalgebras
\begin{equation}\label{eq:nomore2}
\gr W^\ch(\Pi V)
\,\hookrightarrow\,
W^\cl(\gr\Pi V)
\,,
\end{equation}
mapping\/ $\bar X\in\gr^1W^\ch(\Pi V)$ to\/ $\tilde X\in \gr^1 W^\cl_1(\gr \Pi V)$.
\item
Hence, we have an injective morphism of complexes
\begin{equation}\label{eq:nomore3}
(\gr W^\ch(\Pi V),d_{\bar X}=\gr\ad X)
\,\longrightarrow\,
(W^\cl(\gr\Pi V),d_{\tilde X}=\ad\tilde X)
\,.
\end{equation}
\item
As a consequence, we have the corresponding Lie superalgebra
homomorphism of cohomologies:
\begin{equation}\label{eq:nomore4}
H(\gr W^\ch(\Pi V),d_{\bar X})
\longrightarrow
\,
H(W^\cl(\Pi V),d_{\tilde X})
\,.
\end{equation}
\end{enumerate}
\end{theorem}
\begin{remark}\label{rem:stop}
It is interesting to understand whether the morphism \eqref{eq:morph-gr}
is in fact an isomorphism. 
In the recent paper \cite{BDSHK19}, we prove this under the assumption that the filtration of $V$ is induced from a grading such that $V\cong\gr V$ as $\mb F[\partial]$-modules.
In this case, \eqref{eq:nomore2} and \eqref{eq:nomore4} 
are Lie superalgebra isomorphisms
and, since the cohomology of a complex is majorized
by the cohomology of the associated graded complex,
we get the inequalities
\begin{equation}\label{eq:ineq}
\dim H^k(W^\ch(\Pi V),d_X)
\leq
\dim H^k(W^\cl(\gr\Pi V),d_{\tilde X})
\end{equation}
for every $k\geq-1$.
These inequalities obviously always hold for $k=-1,0$.
\end{remark}

\subsection{A finite analog of the operad $P^\cl$}\label{sec:6.4b}

For a vector superspace $V$, we can define a finite analog $P^\fn$ 
of the operad $P^\cl$ introduced in Section \ref{sec:6.2} as follows
(cf. \cite{Mar96}).
We let
$P^\fn(n)$ be the space of all maps 
\begin{equation}\label{20170614:eq4-fin}
f\colon
\mc G(n)\times V^{\otimes n}
\longrightarrow V
\,,
\end{equation}
which are linear in the second factor,
mapping the $n$-graph $\Gamma\in\mc G(n)$ 
and the monomial $v_1\otimes\,\cdots\,\otimes v_n\in V^{\otimes n}$
to the vector
$f^{\Gamma}(v_1\otimes\,\cdots\,\otimes v_n)$,
satisfying 
the cycle relations \eqref{eq:cycle1} and \eqref{eq:cycle2}.
The action of the symmetric groups $S_n$ is given by 
\begin{equation}\label{20170615:eq1-fin}
(f^\sigma)^\Gamma(v_1\otimes\cdots\otimes v_n)
=
f^{\sigma(\Gamma)}
(\sigma(v_1\otimes\cdots\otimes v_n))
\,,
\end{equation}
where 
$\sigma(v_1\otimes\dots\otimes v_n)$ is defined by \eqref{eq:operad6},
and $\sigma(\Gamma)$ is defined in Section \ref{sec:6a.3}.
As for the composition maps,
using the cocomposition maps on graphs defined in \eqref{20170614:eq3},
we let
\begin{equation}\label{20170616:eq3-fin}
(f(g_1,\dots,g_n))^\Gamma
=
f^{\Delta^{m_1\dots m_n}_0(\Gamma)}
\big(
{g_1}^{\Delta^{m_1\dots m_n}_1(\Gamma)}
\otimes\dots\otimes
{g_n}^{\Delta^{m_1\dots m_n}_n(\Gamma)}
\big)
\,,
\end{equation}
for $f\in P^\fn(n)$, $g_1\in P^\fn(m_1),\,\dots,g_n\in P^\fn(m_n)$,
and $\Gamma\in\mc G(M_n)$.

The same proof as for Theorem \ref{20170616:thm2} leads to
\begin{theorem}\label{20170616:thm2-fin}
We have a bijective correspondence between 
the odd elements\/ $X\in W^{\fn}_1(\Pi V)$ such that\/ $X\Box X=0$
and the Poisson superalgebra structures on\/ $V$,
given by
\begin{equation}\label{20170616:eq4-fin}
ab=(-1)^{p(a)}
X^{\bullet\!-\!\!\!\!\to\!\bullet}(a\otimes b)
\,\,,\,\,\,\,
\{a,b\}=(-1)^{p(a)}X^{\bullet\,\,\,\bullet}(a\otimes b)
\,.
\end{equation}
\end{theorem}

\section{The variational Poisson cohomology and the PVA cohomology}
\label{sec:565}

\subsection{The Lie superalgebra $W^{\partial,\as}(\Pi V)$}

In this section, we review the construction of the cohomology complex
associated to a Poisson vertex algebra introduced in \cite{DSK13}.
Let $V$ be a commutative associative superalgebra with an even derivation $\partial$.
As usual, we denote by $p$ the parity of $V$
and by $\Pi V$ the space $V$ with reversed parity $\bar p$.
For $k\geq-1$, we let $W^{\partial,\as}_k(\Pi V)$ be the subspace
of $W_k^\partial(\Pi V)$ (cf. Section \ref{sec:LCA.2})
consisting of all linear maps
$$
f\colon V^{\otimes n}\,\longrightarrow\,\mb F_-[\lambda_1,\dots,\lambda_n]\otimes_{\mb F[\partial]}V
\,\,,\,\,\,\,
v_1\otimes\,\dots\,\otimes v_n\mapsto f_{\lambda_1,\dots,\lambda_n}(v_1\otimes\,\dots\,\otimes v_n)
\,,
$$
satisfying the sesquilinearity conditions \eqref{20170613:eq1}
and the following Leibniz rules:
\begin{equation}\label{eq:leib}
\begin{split}
& f_{\lambda_1,\dots,\lambda_n}
(v_1,\dots,u_iw_i,\dots,v_n)
\\
&\vphantom{\Big(} =
(-1)^{p(w_i)(s_{i+1,k}+k-i)}
f_{\lambda_1,\dots,\lambda_i+\partial,\dots,\lambda_n}(v_1,\dots,u_i,\dots,v_n)_\to w_i
\\
& +
(-1)^{p(u_i)(p(w_i)+s_{i+1,k}+k-i)}
f_{\lambda_1,\dots,\lambda_i+\partial,\dots,\lambda_n}(v_1,\dots,w_i,\dots,v_n)_\to u_i
\,,
\end{split}
\end{equation}
where the arrow means that $\partial$ is moved to the right
and $s_{ij}$ is as in \eqref{eq:sij}.
\begin{proposition}[{\cite[Prop.\ 5.1-5.2]{DSK13}}]\label{prop:pvacomplex}
The space 
$$
W^{\partial,\as}(\Pi V)=\bigoplus_{k\geq-1}W_k^{\partial,\as}(\Pi V)
$$
is a subalgebra of the Lie superalgebra\/ $W^{\partial}(\Pi V)$.
Moreover, there is a bijective correspondence between 
the odd elements\/ $\bar X\in W_1^{\partial,\as}(\Pi V)$
such that\/ $[\bar X,\bar X]=0$ 
and the Poisson vertex algebra\/ $\lambda$-brackets on\/ $V$, 
given by
\begin{equation}\label{eq:pvaX}
[a_\lambda b]
=
(-1)^{p(a)}\bar X_{\lambda,-\lambda-\partial}(a\otimes b)
\,.
\end{equation}
As a consequence, 
given a Poisson vertex algebra\/ $\lambda$-bracket on\/ $V$,
we have the corresponding cohomology complex\/
$(W^{\partial,\as}(\Pi V),d_{\bar X})$
with differential\/ $d_{\bar X}=\ad \bar X$.
\end{proposition}

\subsection{Relation between $W^\cl(\Pi V)$ and $W^{\partial,\as}(\Pi V)$}

Let $V$ be a Poisson vertex algebra.
Recall that, by Theorem \ref{20170616:thm2},
associated to the PVA structure of $V$ there is an odd element $X\in W_1^\cl(\Pi V)$
such that $[X,X]=2X\Box X=0$,
and we thus have the corresponding cohomology complex 
\begin{equation}\label{eq:cohom1}
(W^\cl(\Pi V),\ad X)
\,.
\end{equation}
Moreover, by Proposition \ref{prop:pvacomplex},
we also have an odd element $\bar X\in W_1^{\partial,\as}(\Pi V)$
such that $[\bar X,\bar X]=0$,
and we thus have the corresponding cohomology complex 
\begin{equation}\label{eq:cohom2}
(W^{\partial,\as}(\Pi V),\ad\bar X)
\,.
\end{equation}
By \eqref{20170616:eq4} and \eqref{eq:pvaX}, we have
\begin{equation}\label{eq:barX}
\bar X=X^{\bullet\,\,\bullet}
\,.
\end{equation}
It is natural to ask what is the relation between the two cohomology theories \eqref{eq:cohom1}
and \eqref{eq:cohom2}.

Recall that the operad $P^\cl(\Pi V)$, hence the Lie superalgebra $W^\cl(\Pi V)$,
has a grading $\gr^r$ defined in \eqref{eq:pclgr}:
an element $f\in \gr^r\,W_k^\cl(\Pi V)$ vanishes on all graphs $\Gamma$
with $k+1$ vertices and number of edges not equal to $r$.
Hence, every $f\in W_k^\cl(\Pi V)$ decomposes as a finite sum
\begin{equation}\label{eq:fgr}
f=\sum_{r\geq0}f_r
\,,
\end{equation}
where
${f_r}^\Gamma=f^\Gamma$ if $\Gamma$ has $r$ edges, and ${f_r}^\Gamma=0$ otherwise.
In particular, the element $X$ decomposes as
$$
X=X_0+X_1
\,,
$$
and the condition $[X,X]=0$ is equivalent to
$$
[X_0,X_0]=[X_1,X_1]=[X_0,X_1]=0
\,.
$$
Hence, we have two anticommuting differentials $d_{X_0}=\ad X_0$ and $d_{X_1}=\ad X_1$
on $W^\cl(\Pi V)$, which are homogeneous of degree $0$ and $1$ respectively.
\begin{lemma}\label{lem:gr0}
We have a natural Lie algebra isomorphism
\begin{equation}\label{eq:wdgr}
W^\partial(\Pi V)\,\stackrel{\sim}{\longrightarrow}\,\gr^0W^\cl(\Pi V)
\,,
\end{equation}
mapping\/ $\bar f\in W^\partial(\Pi V)$ to the element\/ $f\in\gr^0W^\cl(\Pi V)$
such that 
$$
f^{\bullet\,\cdots\,\bullet}=\bar f
\,\,\text{ and }\,\,
f^\Gamma=0
\,\text{ if }\,
|E(\Gamma)|\neq\emptyset
\,.
$$
\end{lemma}
\begin{proof}
It follows from the definitions,
that we have an isomorphism between the operads $\mc Chom$ and $\gr^0 P^\cl$.
The statement of the lemma is an obvious consequence of this fact.
\end{proof}
\begin{lemma}\label{lem:gr1}
Let\/ $\bar f\in W^\partial_k(\Pi V)$ and let $f_0$ be its image in\/ $\gr^0W^\cl_k(\Pi V)$
via the isomorphism \eqref{eq:wdgr}.
We have:
\begin{enumerate}[(a)]
\item
$d_X f_0=0\,\,\Longleftrightarrow\,\,d_{X_0}f_0= d_{X_1}f_0=0;$
\item
$d_{X_0}f_0=0\,\,\Longleftrightarrow\,\,d_{\bar X}\bar f=0;$
\item
$d_{X_1}f_0=0\,\,\Longleftrightarrow\,\,\bar f\in W^{\partial,\as}_k(\Pi V)$.
\end{enumerate}
Hence,
$$
f_0\in\Ker(d_X)
\,\,\,\,\Longleftrightarrow\,\,\,\,
\bar f\in\Ker\big(d_{\bar X}\big|_{W^{\partial,\as}(\Pi V)}\big)
\,.
$$
\end{lemma}
\begin{proof}
Claim (a) is obvious, by looking at the various degrees separately.
Claim (b) follows from Lemma \ref{lem:gr0}.
Let us prove claim (c).
Note that $d_{X_1}f_0=[X_1,f_0]\in\gr^1W_{k+1}(\Pi V)$.
Hence, to impose the condition $[X_1,f_0]=0$
it is enough to evaluate it on graphs with $1$ edge,
and, by symmetry, on the graph
$$
\Gamma=
\bullet\,\,\,\,\,\,\,\bullet\,\,\,\cdots\,\,\,\bullet\,\,\,\,\,\,\,\bullet\!-\!\!\!\to\!\!\bullet
\,.
$$
By definition, $[X_1,f_0]=X_1\Box f_0-(-1)^{\bar p(f_0)}f_0\Box X_1$,
and we will compute the two summands separately.
By \eqref{eq:box} and \eqref{20170615:eq1}, we have
\begin{align*}
& (X_1\Box f_0)^{\Gamma}_{\lambda_1,\dots,\lambda_{k+2}}(v_1\otimes\dots\otimes v_{k+2})
\\
& =
\sum_{\sigma\in S_{k+1,1}}
\big((X_1\circ_1 f_0)^{\sigma^{-1}}\big)^{\Gamma}_{\lambda_1,\dots,\lambda_{k+2}}(v_1\otimes\dots\otimes v_{k+2})
\\
& =
\sum_{\sigma\in S_{k+1,1}}
(X_1\circ_1 f_0)^{\sigma^{-1}(\Gamma)}_{\sigma^{-1}(\lambda_1,\dots,\lambda_{k+2})}
(\sigma^{-1}(v_1\otimes\dots\otimes v_{k+2}))
\,.
\end{align*}
Observe that, since $f_0$ has zero degree,
$(X_1\circ_1 f_0)^{\sigma^{-1}(\Gamma)}=0$ 
if the subgraph obtained from $\sigma^{-1}(\Gamma)$
by deleting the vertex labeled $k+2$ has an edge.
This leaves only two shuffles in the above sum:
$\sigma=$ the identity and $\sigma=$ the transposition of $k+1$ and $k+2$.
In the latter case, $\sigma^{-1}(\Gamma)$ is the same as $\Gamma$ with reversed orientation
of the edge, which leads to a minus sign.
Hence, by \eqref{20170616:eq3} and \eqref{20170616:eq4},
we get
\begin{align*}
& (X_1\Box f_0)^{\Gamma}_{\lambda_1,\dots,\lambda_{k+2}}(v_1\otimes\dots\otimes v_{k+2})
\\
& =
(X_1\circ_1 f_0)^{\Gamma}_{\lambda_1,\dots,\lambda_{k+2}}
(v_1\otimes\dots\otimes v_{k+2})
\\
& - (-1)^{\bar p(v_{k+1})\bar p(v_{k+2})}
(X_1\circ_1 f_0)^{\Gamma}_{\lambda_1,\dots,\lambda_{k+2},\lambda_{k+1}}
(v_1\otimes\dots\otimes v_{k+2}\otimes v_{k+1})
\\
& =
X_1^{\bullet\!-\!\!\!\!\to\!\bullet}
\Big(
\bar f_{\lambda_1,\dots,\lambda_{k+1}+x_{k+2}}(v_1\otimes\dots\otimes v_{k+1})
\otimes
\big(\big|_{x_{k+2}=\lambda_{k+2}+\partial}v_{k+2}\big)
\Big)
\\
& 
- (-1)^{\bar p(v_{k+1})\bar p(v_{k+2})}
X_1^{\bullet\!-\!\!\!\!\to\!\bullet}
\Big(
\bar f_{\lambda_1,\dots,\lambda_k,\lambda_{k+2}+x_{k+1}}(v_1\otimes\dots\otimes v_k\otimes v_{k+2})
\otimes\\
&\,\,\,\,\,\,\,\,\,\,\,\,\,\,\,\,\,\,\,\,\,\,\,\,\,\,\,\,\,\,\,\,\,\,\,\,\,\,\,\,\,\,\,\,\,\,\,\,\,\,\,\,\,\,\,\,\,\,\,\,\,\,\,
\otimes
\big(\big|_{x_{k+1}=\lambda_{k+1}+\partial}v_{k+1}\big)
\Big)
\\
& =
(-1)^{1+\bar p(\bar f)+\bar p(v_1)+\dots+\bar p(v_{k+1})}
\Big(
\bar f_{\lambda_1,\dots,\lambda_{k+1}+\lambda_{k+2}+\partial}(v_1\otimes\dots\otimes v_{k+1})_\to
v_{k+2}
\\
& 
+ (-1)^{p(v_{k+1}) p(v_{k+2})}
\bar f_{\lambda_1,\dots,\lambda_{k+1}+\lambda_{k+2}+\partial}(v_1\otimes\dots\otimes v_{k+2})_\to
v_{k+1}
\Big)
\,.
\end{align*}
As for the second summand in the bracket $[X_1,f_0]$, we have,
by \eqref{eq:box} and \eqref{20170615:eq1},
\begin{align*}
& (f_0\Box X_1)^{\Gamma}_{\lambda_1,\dots,\lambda_{k+2}}(v_1\otimes\dots\otimes v_{k+2})
\\
& =
\sum_{\sigma\in S_{2,k}}
\big((f_0\circ_1 X_1)^{\sigma^{-1}}\big)^{\Gamma}_{\lambda_1,\dots,\lambda_{k+2}}(v_1\otimes\dots\otimes v_{k+2})
\\
& =
\sum_{\sigma\in S_{2,k}}
(f_0\circ_1 X_1)^{\sigma^{-1}(\Gamma)}_{\sigma^{-1}(\lambda_1,\dots,\lambda_{k+2})}
(\sigma^{-1}(v_1\otimes\dots\otimes v_{k+2}))
\,.
\end{align*}
Since $f_0$ has zero degree,
$(f_0\circ_1 X_1)^{\sigma^{-1}(\Gamma)}=0$ 
unless the only edge of the graph $\sigma^{-1}(\Gamma)$
connect the vertices labelled $1$ and $2$.
This happens for only one shuffle, given by
$$
\sigma(1)=k+1,\,
\sigma(2)=k+2,\,
\sigma(i)=i-2
\,\,\text{ for }\,\,
i=3,\dots,k+2
\,.
$$
Hence, by \eqref{20170616:eq3} and \eqref{20170616:eq4}, we have
\begin{align*}
& (f_0\Box X_1)^{\Gamma}_{\lambda_1,\dots,\lambda_{k+2}}(v_1\otimes\dots\otimes v_{k+2})
\\
& =
\bar \epsilon_\sigma(v)
\bar f_{\lambda_{k+1}+\lambda_{k+2},\lambda_1\dots,\lambda_k}
\big(
X^{\bullet\!-\!\!\!\!\to\!\bullet}(v_{k+1}\otimes v_{k+2})
\otimes
v_1\otimes\dots\otimes v_k
\big)
\\
& =
(-1)^{p(v_{k+1})}
\bar \epsilon_\sigma(v)
\bar f_{\lambda_{k+1}+\lambda_{k+2},\lambda_1\dots,\lambda_k}
(
v_{k+1} v_{k+2}
\otimes
v_1\otimes\dots\otimes v_k
)
\\
& =
(-1)^{1+\bar p(v_1)+\dots+\bar p(v_{k+1})}
\bar f_{\lambda_1\dots,\lambda_k,\lambda_{k+1}+\lambda_{k+2}}
(
v_1\otimes\dots\otimes v_k
\otimes v_{k+1} v_{k+2}
)
\,,
\end{align*}
where
$$
\bar \epsilon_\sigma(v)
=
(-1)^{(\bar p(v_{k+1})+\bar p(v_{k+2}))\sum_{i=1}^k\bar p(v_{i})}
\,.
$$
In the last equality we used the symmetry condition on $\bar f\in W^\partial(\Pi V)$,
and the fact that 
$$
\bar p(v_{k+1}v_{k+2})=1+\bar p(v_{k+1})+\bar p(v_{k+2})
\,.
$$
Combining the above results,
we conclude that the condition $[X_1,f_0]=0$
is equivalent to the equation
\begin{align*}
& \bar f_{\lambda_1,\dots,\lambda_{k+1}+\lambda_{k+2}+\partial}(v_1\otimes\dots\otimes v_{k+1})_\to
v_{k+2}
\\
& + (-1)^{p(v_{k+1}) p(v_{k+2})}
\bar f_{\lambda_1,\dots,\lambda_{k+1}+\lambda_{k+2}+\partial}(v_1\otimes\dots\otimes v_{k+2})_\to
v_{k+1}
\\
& =
\bar f_{\lambda_1\dots,\lambda_k,\lambda_{k+1}+\lambda_{k+2}}
(
v_1\otimes\dots\otimes v_k
\otimes v_{k+1} v_{k+2}
)
\,,
\end{align*}
i.e., $\bar f$ satisfies the Leibniz rule \eqref{eq:leib}.
This proves claim (c).
The last assertion of the lemma is an obvious consequence of the previous claims.
\end{proof}
\begin{theorem}\label{thm:last}
We have a canonical injective homomorphism of Lie superalgebras
\begin{equation}\label{eq:last}
H(W^{\partial,\as}(\Pi V),d_{\bar X})
\,\hookrightarrow\,
H(W^\cl(\Pi V),d_{X})
\end{equation}
induced by the map \eqref{eq:wdgr}.
\end{theorem}
\begin{proof}
By Lemmas \ref{lem:gr0} and \ref{lem:gr1},
the map \eqref{eq:wdgr}
restricts to a Lie superalgebra isomorphism
\begin{equation}\label{eq:wdgr2}
\ker\big(d_{\bar X}\big|_{W^{\partial,\as}(\Pi V)}\big)
\,\stackrel{\sim}{\longrightarrow}\,
\Ker(d_X)\cap\gr^0 W^\cl(\Pi V)
\,.
\end{equation}
Note that, by degree considerations, we have
$$
d_X(W^\cl(\Pi V))\cap\gr^0 W^\cl(\Pi V)
=
\Big\{
[X_0,g_0]
\,\big|\,
g_0\in \gr^0 W^\cl(\Pi V)
\,,\,\,
[X_1,g_0]=0
\Big\}
\,.
$$
It follows that, under the isomorphism \eqref{eq:wdgr2},
$d_{\bar X}(W^{\partial,\as}(\Pi V))$
maps bijectively to
$d_X(W^\cl(\Pi V))\cap\gr^0 W^\cl(\Pi V)$.
Hence, \eqref{eq:wdgr2} induces an isomorphism
\begin{equation}\label{eq:wdgr3}
H(W^{\partial,\as}(\Pi V),d_{\bar X})
\,\stackrel{\sim}{\longrightarrow}\,
\frac{\Ker(d_X)\cap\gr^0 W^\cl(\Pi V)}{d_X(W^\cl(\Pi V))\cap\gr^0 W^\cl(\Pi V)}
\,.
\end{equation}
The claim follows since
the RHS of \eqref{eq:wdgr3}
is a subalgebra of $H(W^\cl(\Pi V),d_X)$.
\end{proof}
\begin{remark}\label{rem:stop2}
The map \eqref{eq:last} is an isomorphism for the $0$-th and $1$-st cohomologies.
Therefore, by Remark \ref{rem:stop}
we have the following inequality
\begin{equation}\label{eq:stop1}
\dim H^k(W^\ch(\Pi V),d_X)
\leq
\dim H^k(W^{\partial,\as}(\gr\Pi V),d_{\tilde X})
\end{equation}
for $k=-1,0$.
In \cite{BDSHKV19}, we prove that  \eqref{eq:last} is an isomorphism, provided that, as a differential algebra, $V$ is an algebra of differential polynomials in finitely many variables.
\end{remark}

\subsection{Application to the free boson}\label{sec:9}

Let $\mc F$ be a differential field with the derivation $\partial$.
Consider the Lie conformal algebra of $N$ free bosons
$$
R=\mc F[\partial]u_1\oplus\dots\oplus\mc F[\partial]u_N\oplus\mc F K
\,,
$$
with the $\lambda$-brackets on the generators $u_1,\dots,u_N$
given by 
$$
[{u_i}_\lambda{u_j}]=\lambda\delta_{ij}K
\,,\qquad
i,j=1,\dots,N
\,,
$$
where $K$ is central and $\partial K=0$.
Its universal enveloping vertex algebra is
$$
\widetilde B=\mc F[K,u_i^{(n)}\,|\,i=1,\dots,N,\,n\in\mb Z_+]
\,,\qquad
\partial u_i^{(n)}=u_i^{(n+1)}
\,,
$$
with the increasing filtration defined by letting the degrees of $u_i^{(n)}$ and $K$ equal $1$.
The associated graded of the vertex algebra $\widetilde B$ is the Poisson vertex algebra
$$
\widetilde{\mc B}:=\gr \widetilde B=\mc F[K,u_i^{(n)}\,|\,i=1,\dots,N,\,n\in\mb Z_+]
\,,\qquad
\partial u_i^{(n)}=u_i^{(n+1)}
\,,\quad
\partial K=0
\,,
$$
with the $\lambda$-bracket on generators given by
$\{{u_i}_\lambda{u_j}\}=\lambda\delta_{ij}K$ for $i,j=1,\dots,N$,
where again $K$ is central.
By \eqref{eq:stop1} we have
\begin{equation}\label{eq:vic1t}
\dim H^k(\widetilde B)\leq\dim H^k(\widetilde{\mc B})
\,,
\end{equation}
for $k=0,1$,
where on the left we have the cohomology of the vertex algebra $\widetilde B$
while on the right we have the variational Poisson cohomology
of the PVA $\widetilde{\mc B}$.
In fact, due to Remarks \ref{rem:stop} and \ref{rem:stop2}, the inequality \eqref{eq:vic1t} holds for all $k\ge0$.

We are interested in the quotients $B=\widetilde B/\langle K-1\rangle$ and $\mc B=\widetilde{\mc B}/\langle K-1\rangle$ by the ideals generated by $K-1$. Then $B$ is the vertex algebra of $N$ free bosons
$$
B=\mc F[u_i^{(n)}\,|\,i=1,\dots,N,\,n\in\mb Z_+]
\,,
$$
while $\mc B$ is the Poisson vertex algebra
$$
\mc B=\mc F[u_i^{(n)}\,|\,i=1,\dots,N,\,n\in\mb Z_+]
\,,
$$
with the $\lambda$-bracket on generators given by
$\{{u_i}_\lambda{u_j}\}=\lambda\delta_{ij}$ for $i,j=1,\dots,N$.
It is not hard to relate the cohomologies of $B$ and $\mc B$ to those of $\widetilde B$ and $\widetilde{\mc B}$, respectively, and to show that
\begin{equation}\label{eq:vic1}
\dim H^k(B)\leq\dim H^k({\mc B})
\,,\qquad k\ge0
\end{equation}
(see \cite{BDSK19} for details).

It was proved in \cite{DSK12} and \cite{DSK13}, respectively,
that $\dim H^k(\mc B)=\binom{N+1}{k+1}$
if $\mc F=\mb F$ with $\partial\mb F=0$,
and $\dim H^k(\mc B)=\binom{N}{k+1}$
if $\mc F$ is linearly closed.
The representatives of cohomology classes were explicitly computed.
Using those results,
it is easy to find representatives of a basis of the space of Casimirs for $\mc B$,
and of the space of derivations of $\mc B$ modulo inner derivations.
For $\mc F=\mb F$, 
representatives of a basis of $H^0(\mc B)\subset\mc B/\partial\mc B$ are 
the Casimir elements 
\begin{equation}\label{eq:vic2}
1,u_1,\dots,u_N
\,,
\end{equation}
and representatives of a basis of $H^1(\mc B)=\Der(\mc B)/\Inder(\mc B)$
are the following derivations,
\begin{equation}\label{eq:vic3}
\frac{\partial}{\partial u_i}
\,,\,\, i=1,\dots,N
\,\,,\,\,\,\,\text{ and }\,\,
D_{ij}=\sum_{n\in\mb Z_+}\Bigl(
u_i^{(n)}\frac{\partial}{\partial u_j^{(n)}}-u_j^{(n)}\frac{\partial}{\partial u_i^{(n)}}
\Bigr)
\,,\,\, 1\leq i<j\leq N
\,.
\end{equation}
If the field $\mc F$ is linearly closed, it contains $x$ such that $\partial x=1$,
hence we have $1=\partial x\equiv0$ in $\mc B/\partial\mc B$,
and $\frac{\partial}{\partial u_i}=\{x{u_i}_\lambda\,\cdot\}|_{\lambda=0}$, $i=1,\dots,N$,
are inner derivations,
while the remaining elements in \eqref{eq:vic2} and \eqref{eq:vic3}
are linearly independent representatives.

Note that, in the case when $\mc F=\mb F$,
the elements \eqref{eq:vic2} are Casimirs of $B$,
linearly independent over $\mb F$.
Hence, $\dim(\Cas(B))\geq N+1$.
On the other hand,
by Theorem \ref{thm:lowcoho} and the inequality \eqref{eq:vic1}
the opposite inequality holds.
It follows that
$$
\dim(\Cas(B))=N+1
\,,
$$
and the elements \eqref{eq:vic2} form a basis of $\Cas(B)$.

Next, still in the case $\mc F=\mb F$, derivations \eqref{eq:vic3}
are actually derivations of the Lie conformal algebra $R$.
Hence, they uniquely extend to derivations of its universal enveloping vertex algebra $B$,
and it is easy to see that they are linearly independent modulo inner derivations of $B$.
Hence, 
$\dim(\Der(B)/\Inder(B))\geq \frac12N(N+1)$.
On the other hand,
by Theorem \ref{thm:lowcoho} and the inequality \eqref{eq:vic1},
the opposite inclusion holds.
It follows that
$$
\dim(\Der(B)/\Inder(B))=\binom{N+1}{2}
\,,
$$
and the derivations \eqref{eq:vic3} are representatives of a basis of $\Der(B)/\Inder(B)$.

Similarly, in the case when $\mc F$ is linearly closed,
we obtain 
$$
\dim(\Cas(B))=N
\,\,\text{ and }\,\,
\dim(\Der(B)/\Inder(B))=\binom{N}{2}
\,,
$$
with the same representatives as for $\mc B$, described above.


\appendix
\numberwithin{equation}{section}
\numberwithin{theorem}{section}

\section{Relation to chiral algebras} \label{sec:bd-relation}
\def\bA{ {\mathbb{A}^1}}
\def\cO{\mathcal{O}}
\def\cA{\mathcal{A}}
\def\cD{\mathcal{D}}
\def\cP{\mathcal{P}}
\def\cF{\mathcal{F}}

In \cite{BD04} the authors introduced an algebro-geometric rendition of the theory of vertex algebras, which they called \emph{chiral algebras}. In this section we outline the relation of the above results with their definitions. 

\subsection{Chiral operations.} 

Consider a smooth algebraic curve $X$ over $\mathbb{F}$. 
For any right  $\mathcal{D}_{X}$-module $\mathcal{A}$, Beilinson and Drinfeld 
construct an operad $\cP^{\ch}_\cA$ whose $k+1$-ary operations are
\[ 
\cP^{\ch}_\cA (k+1) = \mathrm{Hom}_{\cD_{X^{k+1}}\text{-mod}} 
\Bigl( j_* j^* \cA^{\boxtimes(k+1)}, \Delta_* \cA \Bigr), 
\]
where $j$ is the inclusion of the open complement of the diagonal divisor 
on $X^{k+1}$ (union of hypersurfices $x^i = x^j$ for $i \neq j$), 
and $\Delta\colon\bA \rightarrow \mathbb{A}^{k+1}$ is the diagonal embedding 
$x \mapsto \left\{ x,\dots,x \right\}$.
A \emph{non-unital chiral algebra} on $X$ is by definition a morphism of operads 
\[ \mathcal{L}ie \rightarrow \cP^{\ch}_\cA. \]
In particular it is defined by a binary operation 
\begin{equation} \label{eq:muprod} 
\cP^{\ch}_\cA(2) \ni \mu\colon j_* j^* \cA \boxtimes \cA \rightarrow \Delta_* \cA, 
\end{equation} 
satisfying skewsymmetry and Jacobi identity. 

The dualizing sheaf $\omega_X$ of $X$ carries a canonical chiral algebra structure 
given by the residue map. 
For this we define $\mu$ as the cokernel of the inclusion
\[ 
0 \rightarrow  \omega_X \boxtimes \omega_X 
\rightarrow j_* j^* \omega_X \boxtimes \omega_X \xrightarrow{\mu} \Delta_* \omega_X 
\rightarrow 0. 
\]
Skewsymmetry follows from the isomorphism 
$\omega_X \boxtimes \omega_X \simeq \omega_{X^2}$, 
which is skew-equivariant for the action of $\mathbb{Z}/2\mathbb{Z}$ by permutation 
of the two factors. The Jacobi identiy is a little subtler to prove and is a consequence 
of the Cousin resolution of $\omega_{X^3}$ with respect to the diagonal stratification 
(see \cite{BD04} or \cite{FBZ04}). 
A non-unital chiral algebra $\cA$ is called \emph{unital} 
or simply a \emph{chiral algebra} if there is a morphism 
$\omega_X \rightarrow \cA$ of $\mathcal{D}_X$-modules such that the restriction 
of the multiplication $\mu_\cA$ of $\cA$ 
to $j_* j^* \omega_X \boxtimes \cA \rightarrow j_* j^* \cA \boxtimes A$ coincides 
with the cokernel of the sequence
\begin{equation} \label{eq:push1}  
0 \rightarrow \omega_X \boxtimes \cA \rightarrow j_* j^* \omega_X \boxtimes \cA 
\rightarrow \Delta_* \cA \rightarrow 0. 
\end{equation}

\subsection{$\cD$-modules on the line.} 

In the particular case when $X = \bA$ is the affine line over $\mathbb{F}$, 
any $\mathcal{D}_X$-module $\cA$ is determined 
by the $\Gamma(\bA, \mathcal{D}_\bA)$-module $A:=\Gamma(\bA, \cA)$ of global sections. 
The same is true for the $\cD_{X^{k+1}}$-modules 
\[ 
j_* j^* \cA^{\boxtimes(k+1)} \qquad \text{and} \quad \Delta_* \cA. 
\]

Let $\cD_{k+1}$ be the algebra of regular differential operators on $k+1$ variables 
$z_0, \dots,z_k$ as in Section \ref{sec:ostart}, and let $I$ be the left ideal generated 
by $\left\{ z_0-z_i \right\}_{i=1}^k$. 
Let $\cO_{k+1} = \mathbb{F}[z_0,\dots,z_k]$ and recall the algebra $\cO^\star_{k+1}$ 
of functions defined in Section \ref{sec:ostart}. It is naturally an $\cO_{k+1}$-module, 
as is $\cD_{k+1}$. Notice that $A^{\otimes (k+1)}$ is naturally a $\cD_{k+1}$-module. 
We have
\[ 
\Gamma\Bigl(
\mathbb{A}^{k+1}, j_*j^* \cA^{\boxtimes(k+1)}
 \Bigr) 
 = \cO^\star_{k+1} \otimes_{\cO_{k+1}} A^{\otimes (k+1)},
 \]
and the $\cD_{k+1}$-module structure is by the action on the right factor. 

Consider $I \backslash \cD_{k+1}$, which is a $\cD_1$--$\cD_{k+1}$ bimodule as follows. 
The action of $\cD_{k+1}$ is by multiplication on the right. 
The action of $\cD_{1} = \mathbb{F}[z][\partial_z]$ on the left is defined by letting $z$ act as 
multiplication on the left by $z_0$ and $\partial_z$ act as multiplication on the left by 
$\sum_{i=0}^k \partial_{z_i}$. 
We have
\begin{equation} \label{eq:delta-push} 
\Gamma\left( \mathbb{A}^{k+1}, \Delta_* \cA \right) 
= 
A \otimes_{\cD_1} \left( I \backslash \cD_{k+1} \right), 
\end{equation}
with its natural right $\cD_{k+1}$-module structure by right multiplication on the right factor. 
We have 
\begin{equation}\label{eq:k-ary-chiral} 
\mathcal{P}^{\ch}_\cA(k+1) 
= 
\Hom_{\cD_{k+1}} \left( \cO^\star_{k+1} \otimes_{\cO_{k+1}} A^{\otimes(k+1)}, A \otimes_{\cD_1} 
\left( I \backslash \cD_{k+1} \right) \right). 
\end{equation}

\subsection{Equivariant $\cD$-modules} 

Let $X$ be a smooth scheme, $G$ an algebraic group acting on $X$, 
and $\cF$ a quasi-coherent sheaf of $\cO_X$-modules. 
Denote by $a\colon G \times X \rightarrow X$ the $G$-action 
and by $\pi_2\colon G \times X \rightarrow X$ the projection to the second factor. 
We say $\cF$ is \emph{$G$-equivariant} if there exist an isomorphism 
of $\cO_{G \times X}$-modules  
\begin{equation}\label{eq:equivariant-iso} 
\alpha\colon a^* \cF \rightarrow \pi_2^* \cF 
\end{equation}
such that:
\begin{enumerate}
\item 
the diagram
\begin{equation} \label{eq:equivariant-diagram}
\xymatrix{ \left( 1_G \times a \right)^* \pi_2^* \cF \ar[r] & \pi_3^* \cF \\
\left( 1_G \times a \right)^*a^*\cF \ar[u] \ar@{=}[r] & \left( m \times 1_X \right)^* a^* \cF \ar[u]
}
\end{equation}
commutes in the category of $\cO_{G \times G \times X}$-modules
\footnote{Here $\pi_3:G \times G \times X \rightarrow X$ is the projection map.};
\item 
the pullback 
\[ 
\left( e \times 1_X \right)^* \alpha\colon \cF \rightarrow \cF\,, 
\]
where $e \in G$ is the unit, is the identity map. 
\end{enumerate}
A $\cD_X$-module $\cF$ is called \emph{strongly equivariant} 
if a given $\alpha$ as in \eqref{eq:equivariant-iso} is an isomorphism 
of $\cD_{G \times X}$-modules and the diagram \eqref{eq:equivariant-diagram} 
is in the category of $\cD_{G \times G \times X}$-modules. The module $\cF$ 
is said to be \emph{weakly equivariant} if $\alpha$ is an isomorphism 
of $\cO_G \otimes \cD_X$-modules. 

\subsection{Equivariant $\cD$-modules on the line.} 

Consider the affine line $\mathbb{A}^1$ over a field $\mathbb{F}$, 
with its natural action of the additive group $\mathbb{G}_a$ by translations. 
Let $\mathcal{F}$ be a translation equivariant $\cO_{\bA}$-module. Let $0 \in \bA$ be the origin. 
The functor $0^*$ of taking the fiber at $0$ 
defines an equivalence of categories between translation equivariant quasi-coherent sheaves 
on the line and vector spaces. 
The inverse functor associates to the vector space $V$ the sheaf associated 
to the $k[x]$-module $V[x]$ and the action of $t \in \mathbb{G}_a$ is given 
by $v(x) \mapsto v(x+t)$.  The isomorphism $\alpha$ as in \eqref{eq:equivariant-iso} 
is given by 
\begin{equation} \label{eq:translation-iso}
\alpha\colon V[t,x] \rightarrow V[t,x], \qquad v(t,x) \mapsto v(t,x+t). 
\end{equation}

Notice that $V[x]$ has a canonical right $\cD_1$-module structure with $\partial_x$ 
acting by $-d/dx$.  Similarly, $V[t,x]$ has right action 
of $\cD_2 = \Gamma(\mathbb{G}_a \times \bA,\cD_{\mathbb{G}_a \times \bA})$. 
The map \eqref{eq:translation-iso} is a morphism of $\cD_2$-modules. In fact, we have 
an equivalence of categories between strongly equivariant $\cD$-modules 
on $\mathbb{A}^1$ and vector spaces.  

Now let $\partial \in \End(V)$. As in Section \ref{sec:wch}, we have a $\cD_1$-module structure 
on $V[x]$ defined by letting $x$ act as multiplication by $x$ and $\partial_{x}$ act 
by $\partial - d/dx$. The map \eqref{eq:translation-iso} no longer commutes with the action 
of $\partial_t$; hence, it defines a weakly equivariant $\cD$-module structure on the sheaf 
associated to the $\cD_1$-module $V[x]$. In other words, differentiating the $\mathbb{G}_a$ 
action on $V[x]$, we obtain that $\partial_x$ acts by $-d/dx$, which does not coincide with 
the action obtained from the $\cD_1$-module structure.  
The assignment $(V, \partial) \mapsto V[x]$ defines an equivalence of categories 
between weakly equivariant $\cD$-modules on $\mathbb{A}^1$ and pairs $(V,\partial)$ 
of a vector space and an endomorphism. 

\subsection{Equivariant operads} 
Let $\mathcal{P}$ be a symmetric operad and $G$ be a group. We say that $\cP$ 
is \emph{$G$-equivariant} if every space $\cP(n)$ carries an action of $G$ and the composition 
maps \eqref{eq:operad1} are morphisms of $G$-modules. In particular, this implies that 
the action of $G$ commutes with the symmetric group action on each $\cP(n)$. 
It is clear that
the spaces of invariants\/ $\cP(n)^G$ (respectively, coinvariants\/ $\cP(n)_G$) 
form a suboperad (respectively, quotient operad) of\/ $\cP$. 

\subsection{Equivariant chiral operations.} \label{sec:equivariant-operations} Consider a weakly equivariant $\cD$-module 
$\cA$ on the line corresponding to the pair $(V,\partial)$. The $\cD_{k+1}$-module 
\eqref{eq:delta-push} is in this case given by
\begin{equation} \label{eq:20180528-01} 
V \otimes_{\mathbb{F}[\partial]} \mathbb{F}[x][\partial_0,\dots,\partial_k], 
\end{equation}
where we view $\mathbb{F}[x][\partial_0, \dots, \partial_k]$ as 
a $\mathbb{F}[\partial]-\cD_{k+1}$-bimodule as follows. 
The left action of $\partial$ is given by $\sum_{i=0}^k \partial_i$. 
The right action of $\partial_i$ is by multiplication by $\partial_i$,
and the right action of $z_i$ is given by
\begin{equation} \label{eq:20180528-02}
f(x,\partial_0, \dots,\partial_k) \cdot z^i 
= x f(x,\partial_0, \dots,\partial_k) + \frac{\partial}{\partial_{\partial^i}} f(x,\partial_0,\dots,\partial_k). 
\end{equation}
In this case, \eqref{eq:k-ary-chiral} reads
\[ 
\mathcal{P}^{\ch}_\cA(k+1) 
= 
\Hom_{\cD_{k+1}} \left( \cO^\star_{k+1} \otimes V^{\otimes(k+1)}, 
V \otimes_{\mathbb{F}[\partial]} \mathbb{F}[x][\partial_0,\dots,\partial_k]\right). 
\]
The group $\mathbb{G}_a$ acts on these operations as follows. 
Given $t \in \mathbb{G}_a$ and $\phi \in \mathcal{P}^{\ch}_A(k+1)$, 
we obtain a new operation $\phi^{t}$ by letting
\begin{equation} \label{eq:trnequiv1} 
\phi^t(f(z_0,\dots,z_k) \otimes v_0 \otimes\cdots \otimes v_k) 
:= 
\phi\left( f(z_0-t,\dots,z_k-t) \otimes v_0 \otimes \cdots \otimes v_k \right)\bigr|_{x=x+t}. 
\end{equation} 
The set of translation invariant operations 
$\mathcal{P}_{\cA}^{T,\ch} \subset \mathcal{P}_{\cA}^{\ch}$ defines a suboperad. 
A weakly translation equivariant $\cD$-module $\cA$ on $\bA$ is called 
a non-unital \emph{weakly translation equivariant} chiral algebra 
if the multiplication \eqref{eq:muprod} is tranlation invariant. 
For instance, the unital chiral algebra $\omega_{\bA}$ is weakly translation equivariant. 
A unital chiral algebra is called \emph{translation equivariant} if the morphism 
$\omega_{\bA} \rightarrow \cA$ is equivariant for the $\mathbb{G}_a$-action. 
 
\begin{lemma} \label{lem:equivalence} 
Let\/ $V$ be an\/ $\mathbb{F}[\partial]$-module, and\/ $\cA$ be its associated weakly 
equivariant\/ $\cD$-module on\/ $\bA$. Let\/ $P^{\ch}$ be the operad 
from Proposition \ref{prop:pch} associated to\/ $V$.  
Then we have an isomorphism of operads\/ $\cP^{T,\ch}_\cA \simeq P^{\ch}$. 
\end{lemma}
\begin{proof}
Recall the algebra of translation invariant differential operators $\cD_{k+1}^T$ of Section \ref{sec:ostart}.
The action of $\mathbb{G}_a$ on $\mathbb{A}^1$ induces an action on $\Delta_* \cA$ 
and consequently on its global sections \eqref{eq:20180528-01}, which is given simply 
by $x \mapsto x+t$. The space of invariant sections is a $\cD_{k+1}^T$-module isomorphic 
to \eqref{20160722:eq1}. Indeed, we have an isomorphism
\begin{equation} \label{eq:20180528-iso1} 
\begin{split}
V \otimes_{\mathbb{F}[\partial]} \mathbb{F}[\partial_0, \dots, \partial_k] &
\xrightarrow{\sim} V[\lambda_0, \dots,\lambda_k] \slash 
\langle \partial+\lambda_0+\cdots + \lambda_k \rangle, \\
v \otimes f(\partial_0,\dots,\partial_k) &
\mapsto f(-\lambda_0,\dots,-\lambda_k)v,
\end{split}
\end{equation}
which is compatible with the action of $\cD^T_1 = \mathbb{F}[\partial]$. 
Similarly, the space of $\mathbb{G}_a$-invariant sections 
of $\cO^\star_{k+1} \otimes V^{\otimes(k+1)}$ is given 
by $\cO^{\star T}_{k+1} \otimes V^{\otimes(k+1)}$ and is a $\cD_{k+1}^T$-module 
as in Section \ref{sec:wch}. 

For $\phi \in \cP_\cA^{T,\ch}(k+1)$, 
restricting $\phi$ to $\cO^{\star T}_{k+1} \otimes V^{\otimes (k+1)}$, 
we see that by \eqref{eq:trnequiv1}
\[ f \otimes v_0 \otimes \cdots \otimes v_k \]
does not depend on $x$; therefore, by \eqref{eq:20180528-iso1}
it defines a vector in
\[ V[\lambda_0, \dots,\lambda_k]\slash \langle \partial+\lambda_0+\cdots + \lambda_k \rangle. \]
Hence, $\phi$ defines an element of $P^{\ch}(k+1)$. 

Conversely, given $X$ as in \eqref{20160629:eq2-c} satisfying the sesquilinearity 
conditions \eqref{20160629:eq4}, we extend $X$ to a morphism $\phi \in \cP^{T,\ch}_\cA(k+1)$ 
as follows. By a Taylor expansion, we can express any function 
$f(z_0,\dots,z_k) \in \cO^\star_{k+1}$ as a finite sum
\[ \sum g_i (z_0, z_1, \dots,z_k) z_0^{n_i}, \]
for some $g_i \in \cO^{\star T}_{k+1}$ and some nonnegative integers $n_i$. 
We define 
\[ \phi \left( f \otimes v_0 \otimes \cdots \otimes v_k\right) 
:= \sum x^{n_i} \, X\left( g_i \otimes v_0 \otimes \cdots \otimes v_k \right), \]
where we identified $X(g_i \otimes v_0 \otimes \cdots \otimes v_k)$ 
with a translation invariant vector in \eqref{eq:20180528-01} by \eqref{eq:20180528-iso1}.
 It is clear that $\phi$ is translation invariant and it is a morphism of $\cD_{k+1}$-modules.
\end{proof}
\begin{corollary}[{\cite[0.15]{BD04}}]
There is an equivalence of categories between weakly translation equivariant chiral algebras 
on\/ $\bA$ and vertex algebras. 
\end{corollary}
\begin{proof}
We first prove the analogous statement for non-unital algebras. 
The non-unital weakly translation equivariant chiral algebras
are given by morphisms of operads $\mathcal{L}ie \rightarrow \cP^{T,\ch}_\cA$ 
for a weakly equivariant $\cD$-module $\cA$. By Lemma \ref{lem:equivalence},
we have an $\mathbb{F}[\partial]$-module $V$ and a morphism 
of operads $\mathcal{L}ie \rightarrow P^{\ch}$. In a similar way 
as in Remark \ref{rem:lie-operad}, these correspond to an odd element 
$X \in W^{\ch}_1(\Pi V)$ satisfying $X \Box X = 0$. 
The result then follows from Theorem \ref{20160719:thm}. 
Under this equivalence, the unit vertex algebra $\mathbb{F}$ corresponds 
to the chiral algebra $\omega_{\bA}$.

Consider now a translation equivariant unital chiral algebra $V$ on $\bA$. 
Then the morphism $\omega_\bA \rightarrow \cA$ corresponds to a morphism 
of vertex algebras $\mathbb{F} \rightarrow V$. The image of $1 \in \mathbb{F}$ 
is the vacuum vector $\vac$ of $V$. Indeed, since $\omega_{\bA} \rightarrow \cA$ 
is a morphism of $\cD$-modules, we have $\partial \vac = 0$. If $X \in W^{\ch}_1(\Pi V)$ 
is the corresponding operation, it follows from \eqref{eq:push1} that
\[ X \Bigl(  \vac \otimes v  \otimes \frac{1}{z_0 - z_1} \Bigr) = v, \qquad v\in V, \]
from where the vacuum axioms follow. 
\end{proof}
\def\cLie{\mathcal{L}ie}

\subsection{Lie conformal operad}  \label{sec:staroperad}
In addition to the operad $P^{\ch}_\cA$ associated to any $\cD_X$-module $\cA$, Beilinson and Drinfeld define an operad $P^*_\cA$ by letting 
\[ 
\cP^{*}_\cA (k+1) = \mathrm{Hom}_{\cD_{X^{k+1}}\text{-mod}} 
\Bigl(\cA^{\boxtimes(k+1)}, \Delta_* \cA \Bigr). 
\]
In the case of $X = \mathbb{A}^1$ and $\cA$ a weakly equivariant $\cD$-module, we let $\cP^{T,*}_\cA$ be the suboperad of $\mathbb{G}_a$-invariant operations. We have in the same way as Lemma \ref{lem:equivalence} the following:
\begin{lemma} \label{lem:starequivalence} 
Let\/ $V$ be an\/ $\mathbb{F}[\partial]$-module, and\/ $\cA$ be its associated weakly 
equivariant\/ $\cD$-module on\/ $\bA$. Let\/ $\mathcal{C}hom$ be the operad 
from Section \ref{sec:LCA.1} associated to\/ $V$.  
Then we have an isomorphism of operads\/ $\cP^{T,*}_\cA \simeq \mathcal{C}hom$. 
\end{lemma}


\subsection{Classical operations} \label{sub:classical} 

For any smooth algebraic curve $X$ over $\mb F$
and any right $\mathcal{D}_X$-module $\mathcal{A}$ on it, in \cite[1.4.27]{BD04} 
the authors define an operad of \emph{classical operations} $\cP^c_\mathcal{A}$ as follows. 
Let $\cLie$ be the Lie operad, that is, $\cLie(k+1)$ is the vector space with a basis consisting of all formal symbols
\[ 
[x_{\sigma(0)}, [x_{\sigma(1)},\cdots[x_{\sigma(k-1)},x_{\sigma(k)}]\cdots]], 
\qquad \sigma \in S_{k+1}, \quad \sigma(0)=0.
\]
The composition in $\cLie$ is defined by replacing the corresponding variables and expanding using the Jacobi and skew-symmetry identities. For each $k\ge0$ and each $(m_0,\dots,m_{k})$-shuffle $\sigma$ as in Section \ref{sec:2.4}, we let 
\[ 
\cLie_\sigma := \cLie(m_0) \otimes \cdots \otimes \cLie(m_k), 
\]
and 
\[ 
\cP^{*}_\cA(\sigma) = \mathrm{Hom}_{\cD_{X^{k+1}}\text{-mod}} 
\Bigl(\cA^{\otimes m_0} \boxtimes \cdots \boxtimes \cA^{\otimes m_k}, \Delta_* \cA \Bigr),
\]
where $\Delta\colon X \hookrightarrow X^{k+1}$ is the small diagonal embedding. 
Finally, put 
\[ 
\cP^c_\cA(n+1) = \bigoplus_{k=0}^{n} \bigoplus_\sigma \cP^{*}_\cA(\sigma) \otimes \cLie_\sigma, 
\]
where the inner sum is over $(m_0,\dots,m_k)$-shuffles $\sigma$ such that 
$\sum_{i=0}^k m_i = n+1$. The composition in $\cP^c_\cA$ is defined as the tensor product of the compositions 
in the operad $\cP^*_\cA$ defined in \ref{sec:staroperad} and the compositions in the $\cLie$ operad (see \cite[1.4.27]{BD04} for details). 
The operad $\cP^{c}_\cA$ defined this way is graded, with the grading given by $k$ in the above sum. 

\begin{remark}
In \cite{BD04} the authors work with unordered sets and equivalence relations on these sets, 
namely, instead of defining the $n$-ary operations $\cP(n)$ for an operad, they define 
the $I$-ary operations $\cP(I)$ for any finite nonempty set $I$. Similarly, composition is defined 
for any equivalence relation $S$ in $I$ instead of a shuffle $\sigma$. For the equivalence 
of these two approaches, see \cite{ginzburg}. 
\end{remark}

In the case when $X=\mathbb{A}^1$ and $\cA$ is a weakly $\mathbb{G}_a$-equivariant $\cD_X$-module,  we consider the translation equivariant suboperad $\cP_\cA^{T,c}$, and in the same way as in \ref{sec:equivariant-operations},
we have the following:
\begin{theorem}
Let\/ $V$ be an\/ $\mathbb{F}[\partial]$-module, and\/ $\cA$ be its associated weakly 
equivariant\/ $\cD$-module on\/ $\mathbb{A}^1$. Let\/ $P^{\cl}$ be the operad 
from Theorem \ref{20170616:thm1} associated to\/ $V$.  
Then we have an isomorphism of graded operads\/ $\cP^{T,c}_\cA \simeq P^{\cl}$. 
\label{thm:isoclassical}
\end{theorem}
\begin{proof}[Sketch of the proof]
The proof relies on a theorem by Chapoton and Livernet \cite{chapoton}, which states that 
the operad of pre-Lie algebras\footnote{That is algebras satisfying an even version 
of \eqref{20170608:eq2}.} is isomorphic to the operad of rooted trees. Using this theorem, 
one can associate to any $n$-ary operation in the operad $\cLie$ a connected graph 
$\Gamma \in \mathcal{G}(n)$ as in Section \ref{sec:6a.1}. 
More generally, given an $(m_0,\dots,m_k)$-shuffle $\sigma$ and  
\[ 
\tau_0 \otimes \cdots \otimes \tau_k \in \cLie_\sigma, 
\] 
we obtain a graph $\Gamma \in \mathcal{G}(\sum m_i)$ with $k+1$ connected components, 
the $i$-th component of which being a connected graph in $\mathcal{G}(m_i)$. 
By Lemma \ref{lem:starequivalence}, 
to any element of $\cP^{T,*}_\cA(\sigma)$ we associate an operation $f^\Gamma$ satisfying 
the sesquilinearity conditions \eqref{eq:sesq1}, \eqref{eq:sesq2}. The operation $f^\Gamma$ 
satisfies in addition the cycle relations \eqref{eq:cycle1}, \eqref{eq:cycle2}, since 
the graph $\Gamma$ comes from an operation in $\cLie$ and therefore satisfies skew-symmetry 
(as opposed to the general graph that by Chapoton--Livernet's theorem defines only 
a pre-Lie algebra operation). 
This defines an isomorphism of graded vector spaces $\cP^{T,c}_\cA (n) \simeq P^{\cl}(n)$ 
for all $n$. One readily checks that this isomorphism is compatible with compositions 
in both operads. 
\end{proof}


The operad $\cP^{\ch}_\cA$ carries a natural filtration given by the diagonal stratification 
of $X^n$ for each $n$. It gives rise to the associated graded operad $\gr \cP^{\ch}_\cA$. 
In \cite[3.2.5]{BD04} the authors produce a canonical embedding of graded operads
\[ \gr \cP^{\ch}_\cA \hookrightarrow \cP^{c}_\cA, \]
and claim that it is an isomorphism if $\cA$ is a projective $\cD_X$-module. In the case 
of $X=\mathbb{A}^1$ and considering the translation invariant suboperads, this embedding 
is the geometric counterpart to Theorem \ref{thm:mor}.  




\end{document}